\documentclass[11pt]{article}
\usepackage[utf8]{inputenc}
\usepackage[T1]{fontenc}
\usepackage{comment}
\usepackage{authblk}
\usepackage{relsize}
\usepackage{changepage}
\usepackage{mathtools}
\usepackage[english]{babel}
\usepackage{amsmath}
\usepackage{amsthm}
\usepackage{amsfonts}
\usepackage{amssymb}
\usepackage{graphicx}
\usepackage[usenames,dvipsnames]{color}
\theoremstyle{plain}
\newtheorem{theorem}{Theorem}[section]
\newtheorem{corollary}[theorem]{Corollary}

\newtheorem{lemma}[theorem]{Lemma}

\newtheorem{claim}[theorem]{Claim}

\newtheorem{observation}[theorem]{Observation}

\makeatletter
\newcommand{\vast}{\bBigg@{4}}
\newcommand{\Vast}{\bBigg@{5}}

\makeatother
\definecolor{bulgarianrose}{rgb}{0.28, 0.02, 0.03}
\definecolor{pk}{rgb}{0.7,0.4,0.5}

\theoremstyle{definition}
\newtheorem{remark}[theorem]{Remark}
\newtheorem{definition}[theorem]{Definition}
\newtheorem{question}[theorem]{Question}

\newcommand{\Un}{\mathrm{Unif}([0,1])}
\newcommand{\PoNN}{G_{n,\mathrm{Po}}}

\newcommand{\diam}{\mathrm{diam}}

\newcommand{\intqi}{\mathrm{int}(q_i)}
\newcommand{\pnt}{\mathrm{par}}

\usepackage[left=2cm,right=2cm,top=2cm,bottom=2cm]{geometry}
\usepackage{amssymb}
\usepackage{hyperref}
\makeatletter
\def\namedlabel#1#2{\begingroup
    #2%
    \def\@currentlabel{#2}%
    \phantomsection\label{#1}\endgroup
}
\usepackage[normalem]{ulem}
\usepackage{dsfont}
\usepackage{accents}
\usepackage{verbatim}
\usepackage{ulem}
\usepackage{pgf,tikz,pgfplots}
\pgfplotsset{compat = 1.16}
\usepackage[all]{xy} 
\allowdisplaybreaks

\usepackage{stackengine}
\stackMath
\newcommand\tsup[2][2]{%
 \def\useanchorwidth{T}%
  \ifnum#1>1%
    \stackon[-.5pt]{\tsup[\numexpr#1-1\relax]{#2}}{\scriptscriptstyle\sim}%
  \else%
    \stackon[.5pt]{#2}{\scriptscriptstyle\sim}%
  \fi%
}

\usepackage{subfig}

\usepackage{enumerate}

\title{\scshape
  New results for the random nearest neighbor tree}

\author[1,3]{Lyuben Lichev}
\author[1,2]{Dieter Mitsche\footnote{Dieter Mitsche has been partially supported by grant Fondecyt grant 1220174 and by grant GrHyDy ANR-20-CE40-0002.}}
\affil[1]{Univ.~Jean Monnet, Institut Camille Jordan, Saint-Etienne, France}
\affil[2]{IMC, Pont.~Univ.~Católica, Santiago, Chile}
\affil[3]{Institute of Mathematics and Informatics, Bulgarian Academy of Sciences, Sofia, Bulgaria}

\begin{document}

\maketitle
 
\begin{abstract}
In this paper, we study the online nearest neighbor random tree in dimension $d\in \mathbb N$ (called $d$-NN tree for short) defined as follows. 
We fix the torus $\mathbb T^d_n$ of dimension $d$ and area $n$ and equip it with the metric inherited from the Euclidean metric in $\mathbb R^d$. 
Then, embed consecutively $n$ vertices in $\mathbb T^d_n$ uniformly at random and independently, and let each vertex but the first one connect to its (already embedded) nearest neighbor. 
Call the resulting graph $G_n$. 

We show multiple results concerning the degree sequence of $G_n$. First, we prove that typically the number of vertices of degree at least $k\in \mathbb N$ in the $d$-NN tree decreases exponentially with $k$ and is tightly concentrated by a new Lipschitz-type concentration inequality that may be of independent interest. Second, we obtain that the maximum degree of $G_n$ is of logarithmic order. Third, we give explicit bounds for the number of leaves that are independent of the dimension and also give estimates for the number of paths of length two. 
Moreover, we show that typically the height of a uniformly chosen vertex in $G_n$ is $(1+o(1))\log n$ and the diameter of $\mathbb T^d_n$ is 
$(2e+o(1))\log n$, independently of the dimension.

Finally, we define a natural infinite analog $G_{\infty}$ of $G_n$ and show that it corresponds to the local limit of the sequence of finite graphs $(G_n)_{n \ge 1}$. 
Moreover, we prove almost surely that $G_{\infty}$ is locally finite, that the simple random walk on $G_{\infty}$ is recurrent, and that $G_{\infty}$ is connected.
\end{abstract}

\section{Introduction}
Inferring the structure of growing random (spatial) networks is a major mathematical challenge with numerous important applications: one wants to know how a virus spreads, how a biological network evolves,
how rumors spread in a social network, or how information spreads in a telecommunication network.
One way to algorithmically capture the closeness between vertices in such (and other) networks is the famous $k$-nearest neighbors algorithm, invented by Fix and Hodges~\cite{FH}, which has by now become a standard tool for non-parametric classification and regression in statistics:
its idea roughly consists in embedding the data points into some predetermined metric space and classifying new arrivals according to the properties of the closest $k$ data points already processed. 
Compared to the vast amount of literature in machine learning and statistics on the $k$-nearest neighbor algorithm (see e.g.\ \cite{Alt,BJ,BC,CH,Dud} among others), 
less is known about the structure of a typical digraph originating from it by orienting edges from newly arrived data points to their $k$ nearest predecessors (see Section~\ref{relwork} for a detailed account on related results).

The goal of the present paper is to provide some insight on the topic in the particular case $k=1$. 
First, we formally define the model. Fix positive integers $d$ and $n$.  
The \emph{online nearest neighbor process} in the $d$-dimensional torus $\mathbb T_n^d$ with volume $n$, or the \emph{$d$-NN process} for short, 
is defined as follows:\footnote{The choice of the torus as an ambient space might not be the most natural. 
However, it avoids the need for boundary considerations and, in most part, does not modify the results. 
Indeed, the main proof ideas can be applied for other geometric spaces but at the cost of an increased level of technicality.} 
Starting from $0$, at every integer time step $i\in \{0,1,\dots,n-1\}$, embed a new vertex in $\mathbb T^d_n$ uniformly and independently from the positions of the previous vertices. 
For convenience, we identify every vertex with its time of embedding.
Then, if $i > 0$, connect the vertex $i$ by an edge to the closest vertex among $\{0,1,\dots,i-1\}$ with respect to the torus distance, ties being broken arbitrarily (of course, ties do not happen a.s.). 
Clearly, before step $i\in \{0,1,\dots,n-1\}$, the graph obtained by this process is a tree with $i$ vertices which we denote by $G_i$. 
Our main object of interest is the final tree $G_n$, to which we refer as \emph{the $d$-NN tree on $n$ vertices} or simply \emph{the $d$-NN tree}. 
In the sequel, we also consider an infinite version of this model that will be introduced shortly.

\subsection{Notation}

We mostly rely on standard notation. 
For a graph $G$, we denote by $V(G)$ its vertex set and by $E(G)$ its edge set. 
We also denote by $|G|$ the \emph{order} of $G$ (i.e.\ the size of $V(G)$), by $\mathrm{diam}(G)$ the diameter of $G$, by $\Delta(G)$ the maximum degree of $G$, and by $L(G)$ the number of leaves of $G$ (that is, the number of vertices of degree 1). 
Also, for $t\in \mathbb N$ and a graph $G$, a $t$-neighbor of a vertex $u$ in $G$ is a vertex at graph distance $t$ from $u$, and the $t$-neighborhood of $u$, denoted by $N^t_G[u]$, is the set of vertices at distance at most $t$ from $u$ in $G$. 
For a rooted tree $T$ and a vertex $v$ of $T$, denote by $h(v, T)$ the distance from $v$ to the root of $T$ (which we call the \emph{height of $v$ in $T$}), and by $h(T)$ the \emph{height} of $T$, that is, $h(T) = \max_{v\in V(T)} h(v,T)$. 
Whenever we speak of an oriented edge in a rooted tree below, the orientation of the edge is always towards the root. 
The lowercase letters $u,v,w$, possibly with some upper and lower indices, are reserved for vertices, and $e$ is reserved for edges. 
For a vertex $v\in V(G_n)$, we denote by $p(v)$ its position in $\mathbb T^d_n$. 

The set of non-negative integers is denoted by $\mathbb N_0$, and moreover $[n] = \{1,\dots,n\}$ and $[n]_0 = [n]\cup \{0\}$. 
We rely on standard asymptotic notation $O(\cdot),\Theta(\cdot),\Omega(\cdot),o(\cdot), \omega(\cdot)$ together with $f(n)\ll g(n)$ (or equivalently $g(n)\gg f(n)$) when $f(n) = o(g(n))$. 
We note that constants in the $O,\Theta,\Omega$ notation may depend on parameters of the problem (in particular, they may depend on $d$).

We say that a sequence of events $(E_n)_{n \ge 1}$ holds asymptotically almost surely (or a.a.s.) if $\lim_{n \to \infty}\mathbb P(E_n)=1$. 
We also use the notation $X \sim \mathcal D$ to say that a random variable $X$ is distributed according to distribution $\mathcal D$. 
Furthermore, given $x\in \mathbb R^d$ and a real number $r\ge 0$, we denote by $B(x,r)$ the closed ball with radius $r$ and center $x$ in $\mathbb R^d$.
We also remark that the notation $\mathrm{diam}_E(\cdot)$ is also used to denote the Euclidean diameter of convex bodies like balls and cubes in $\mathbb R^d$ (where the dimension is spared in the notation for convenience).

\subsection{\texorpdfstring{Two equivalent definitions and the Poisson $d$-NN model}{}}
The independence between the positions and the arrival times of the vertices of $G_n$ is a crucial ingredient in our analysis. The next definition emphasizes the aforementioned decomposition into space-time components.

\begin{definition}[Equivalent definition 1]\label{eq def 1}
Given $n$ unlabeled vertices, assign labels to them via sampling a permutation of $[n-1]_0$ uniformly at random. Then, embed these $n$ vertices independently and uniformly in $\mathbb T^d_n$. Finally, $G_n$ is formed by connecting by an edge every vertex with label $i\ge 1$ to its closest neighbor with smaller label.
\end{definition}

While in the above definition, the space-time decomposition was made explicit, the labels attributed to the vertices in the second stage are not independent. The next reformulation corrects this inconvenience.

\begin{definition}[Equivalent definition 2]\label{eq def 2}
Given $n$ vertices $\{v_0,\dots,v_{n-1}\}$, sample a family $(X_i)_{i\in [n-1]_0}$ of i.i.d.\ $\Un$ arrival times, and set $i_0$ as the index of the smallest among $(X_i)_{i\in [n-1]_0}$. Then, embed $\{v_0,\dots,v_{n-1}\}$ independently and uniformly in $\mathbb T^d_n$. Finally, $G_n$ is formed by connecting by an edge each vertex $\{v_i, i\in [n-1]_0\setminus i_0\}$ to its closest neighbor among $\{v_j: X_j < X_i\}$.
\end{definition}

Note that Definition~\ref{eq def 2} may fail to produce a tree in case of ties, but we may ignore this minor technicality since this happens with probability 0. Below, we use the two equivalent definitions of the model interchangeably, depending on which one is more useful in the particular context.

A twin model of the $d$-NN tree is the \emph{Poisson $d$-NN tree} $\PoNN$ formed by embedding $N\sim \mathrm{Po}(n)$ vertices independently and uniformly in $\mathbb T^d_n$ (which defines a Poisson Point Process with intensity 1 on $\mathbb T^d_n$, abbreviated PPP(1)). Note that the Poisson variable $N$ is assumed to be independent of all positions and arrival times. The main advantage of defining the set of vertex positions via a Poisson Point Process is motivated by the following two properties: first, the number of points that lie in any measurable set $A\subseteq \mathbb T^d_n$ of Lebesgue measure $a$ has a Poisson distribution with expectation $a$, and second, the number of points in disjoint subsets of $\mathbb T^d_n$ are independently distributed. To avoid repetition, we (somehow abusively) refer to Definitions~\ref{eq def 1}~and~\ref{eq def 2} when dealing with $\PoNN$ as well.

Another nice feature of the Poisson model is that it may be extended to non-compact spaces via Definition~\ref{eq def 2}. As above, we ignore possible ties (these happen with probability $0$ also for infinite countable sets of vertices).

\begin{definition}[The infinite Poisson $d$-NN random tree]\label{def inf PNNT}
Let $\mathcal V$ be a Poisson Point Process in $\mathbb R^d$. Sample a family $(X_v)_{v\in \mathcal V}$ of i.i.d.\ $\Un$ arrival times. Then, $G_{\infty}$ is formed by connecting by an edge each vertex $v\in \mathcal V$ to its closest neighbor among $\{u: X_u < X_v\}$.
\end{definition}

A major difference with the finite graphs constructed above is that $G_{\infty}$ is not deterministically a tree even if Definition~\ref{def inf PNNT} avoids ties. An example of a disconnected nearest neighbor graph embedded in $\mathbb R\subset \mathbb R^d$ is given by two sets of vertices on the real axis with positions $(\sum_{i=1}^n i^{-1})_{n\ge 1}\cup (-\sum_{i=1}^n i^{-1})_{n\ge 1}$ and arrival times respectively $2^{-2n+1}$ for the vertex in position $\sum_{i=1}^n i^{-1}$ and $2^{-2n}$ for the vertex in position $- \sum_{i=1}^n i^{-1}$ (see Figure~\ref{fig 1}). We will however show that $G_{\infty}$ is connected a.s.

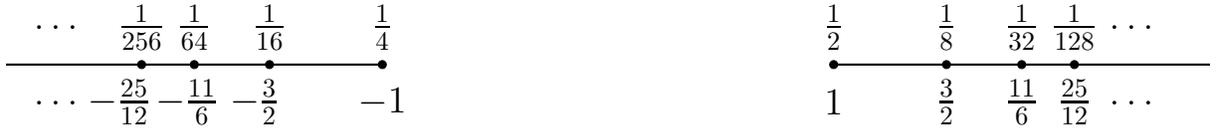
\begin{figure}
\centering
\begin{tikzpicture}[line cap=round,line join=round,x=1cm,y=1cm]
\clip(-8,-1) rectangle (8,1);
\draw [line width=0.8pt,domain=3:9.907319633388674] plot(\x,{(-0-0*\x)/6});
\draw [line width=0.8pt,domain=-9.146006214703345:-3] plot(\x,{(-0-0*\x)/6});
\begin{scriptsize}
\draw [fill=black] (-3,0) circle (1.5pt);
\draw[color=black] (-3,-0.5) node {\Large{$-1$}};
\draw[color=black] (-3,0.5) node {\Large{$\frac14$}};
\draw [fill=black] (3,0) circle (1.5pt);
\draw[color=black] (3,-0.5) node {\Large{$1$}};
\draw[color=black] (3,0.5) node {\Large{$\frac12$}};
\draw [fill=black] (4.5,0) circle (1.5pt);
\draw[color=black] (4.5,-0.5) node {\Large{$\frac32$}};
\draw[color=black] (4.5,0.5) node {\Large{$\frac18$}};
\draw [fill=black] (5.5,0) circle (1.5pt);
\draw[color=black] (5.5,-0.5) node {\Large{$\frac{11}{6}$}};
\draw[color=black] (5.5,0.5) node {\Large{$\frac{1}{32}$}};
\draw [fill=black] (6.2,0) circle (1.5pt);
\draw[color=black] (6.2,-0.5) node {\Large{$\frac{25}{12}$}};
\draw[color=black] (6.2,0.5) node {\Large{$\frac{1}{128}$}};
\draw[color=black] (7,-0.5) node {\Large{$\dots$}};
\draw[color=black] (7,0.5) node {\Large{$\dots$}};
\draw [fill=black] (-4.5,0) circle (1.5pt);
\draw[color=black] (-4.7,-0.5) node {\Large{$-\frac32$}};
\draw[color=black] (-4.5,0.5) node {\Large{$\frac{1}{16}$}};
\draw [fill=black] (-5.5,0) circle (1.5pt);
\draw[color=black] (-5.6,-0.5) node {\Large{$-\frac{11}{6}$}};
\draw[color=black] (-5.5,0.5) node {\Large{$\frac{1}{64}$}};
\draw [fill=black] (-6.2,0) circle (1.5pt);
\draw[color=black] (-6.5,-0.5) node {\Large{$-\frac{25}{12}$}};
\draw[color=black] (-6.2,0.5) node {\Large{$\frac{1}{256}$}};
\draw[color=black] (-7.3,-0.5) node {\Large{$\dots$}};
\draw[color=black] (-7.3,0.5) node {\Large{$\dots$}};
\end{scriptsize}
\end{tikzpicture}
\caption{The above nearest neighbor graph in $\mathbb R$ consists of two infinite paths. Above every vertex is denoted its arrival time, and below its position.}
\label{fig 1}
\end{figure}

\subsection{Our results}

In this section, we state our results. Recall that, unless mentioned otherwise, $d$ is a constant independent of $n$ and the constants in the results depend on $d$.
We first state a general concentration result that will be applied to several graph parameters later on. 
Consider the set $\mathcal{LT}_n$ of labeled trees with vertices $[n-1]_0$ rooted at $0$, and endow it with the metric $(T_1, T_2)\mapsto |E(T_1)\,\Delta\, E(T_2)|$ where $\Delta$ is the symmetric difference operator. 

A function $g:\mathcal{LT}_n\to \mathbb R$ is $L$-\emph{Lipschitz} for some $L > 0$ (possibly depending on $n$) if, for every pair $T_1, T_2\in \mathcal{LT}_n$, we have $|g(T_1) - g(T_2)|\le L|E(T_1)\,\Delta\, E(T_2)|$. 
Our first main results shows that Lipschitz functionals of the $d$-NN tree $G_n$ are tightly concentrated around their means. We denote by $\log(\cdot)$ the natural logarithm.

\begin{theorem}\label{thm concentration}
Fix an $L$-Lipschitz function $g:\mathcal{LT}_n\to \mathbb R$ and set $\mu = \mathbb E[g(G_n)]$. Then, there is a constant $C_0 = C_0(d) > 0$ such that, for all sufficiently large $n$ and for every $\phi = \phi(n) \in [C_0 \log n, n/3]$ and $t\ge 0$,
\begin{equation*}
    \mathbb P[|g(G_n) - \mu|\ge t]\le 2\exp\left(- \dfrac{t^2}{20nL^2\phi^2}\right) + \exp\left(-\frac{\phi\log \phi}{10}\right).
\end{equation*}
\end{theorem}
 
Theorem~\ref{thm concentration} has interesting consequences.
For example, for every $i\ge 1$, denote the set of vertices of degree at least $i$ in $G_n$ by $V_{\ge i}$.
Then, we show that, for every $i\ge 1$ that is ``not too big'', $|V_{\ge i}|$ is tightly concentrated around its expected value. 
Apart from concentration, we also show that $|V_{\ge i}|$ decreases exponentially with $i$.

\begin{theorem}\label{thm main 1}
There are constants $C_1 = C_1(d) > 0$, $C_2 = C_2(d) > C_1$ and $c = c(d) > 0$ such that a.a.s., for every $i\le c\log n$, we have $n\exp(-C_2 i)\le |V_{\ge i}|\le n\exp(-C_1 i)$ and $|V_{\ge i}| = (1+o(1))\mathbb E |V_{\ge i}|$. 
\end{theorem}

Clearly, Theorem~\ref{thm main 1} provides a lower bound on the maximum degree of $G_n$. In fact, the order of this lower bound is the correct one, as the next theorem shows.
\begin{theorem}\label{cor:maxdegree}
$\Delta(G_n)=\Theta(\log n)$ with probability $1-n^{-\omega(1)}$.
\end{theorem}

The next result deals with the typical height of a vertex as well as the height and the diameter of $G_n$. In particular, it shows that the first order term of these quantities does not depend on the dimension.

\begin{theorem}\label{thm diam}
Fix $\varepsilon > 0$ and let $v$ be a uniformly chosen vertex of $G_n$. Then, a.a.s.\
\begin{enumerate}[(i)]
    \item\label{line a}$(1-\varepsilon)\log n\le h(v, G_n)\le (1+\varepsilon) \log n$,
    \item\label{line b}$(2e-\varepsilon)\log n\le \mathrm{diam}(G_n)\le 2 h(G_n)\le (2e+\varepsilon) \log n$.
\end{enumerate}
\end{theorem}
\noindent
We remark that, as a byproduct of the proof of Theorem~\ref{thm diam}, we show that, for two uniformly chosen vertices of $G_n$, their \emph{stretch factor} (that is, the sum of the Euclidean distances of all edges on the path between them, divided by the Euclidean distance between them) is a.a.s.\ $O(1)$. 

Finally, we consider the infinite setup mentioned in the previous section. First, we recall the classical notion of local convergence: consider a sequence of rooted graphs $(H_n,o_n)_{n\ge 1}$ where, for every $n\ge 1$, $o_n$ is a uniformly chosen vertex of $H_n$.
Also, for a positive integer $r\ge 0$, a graph $G$ and a vertex $v$ in $G$, denote by $B^r_G[v]$ the graph, induced by $N^r_G[v]$ in $G$ (also known as \emph{the $r$-ball around $v$ in $G$}).
Given a probability measure $\mu$ on the set of finite rooted graphs, we say that $(H_n,o_n)_{n\ge 1}$ converges locally to $\mu$ if for every integer $r\ge 0$ and every finite rooted graph $(H,o_H)$, we have that 
\begin{equation*}
\mathbb P(B^r_{H_n}[o_n] \cong (H,o_H))\xrightarrow[n\to \infty]{} \mu(\{(G, o_G): B^r_G[o_G] \cong (H, o_H)\}),
\end{equation*}
that is, the distribution of the $r$-ball around $o_n$ in $H_n$ converges to the distribution of the $r$-ball around the root of a graph sampled according to $\mu$.
Often the limiting measure $\mu$ is identified with a (possibly random) graph. For a complete account and applications on the notion of local convergence, see for example~\cite{AldousSteele, BenjaminiSchramm, Sal11}. 
We also recall that a graph is \emph{locally finite} if it has no vertices of infinite degree.

\begin{theorem}\label{thm G infinity}
\begin{enumerate}[(i)]
    \item\label{item local limit} $G_{\infty}$ is the local limit of $(G_n)_{n\ge 1}$.
    \item\label{item loc finite} A.s.\ $G_{\infty}$ is locally finite.
    \item\label{item recurrence} A.s.\ $G_{\infty}$ is recurrent for the simple random walk.
    \item\label{item connected} A.s.\ $G_{\infty}$ is connected.
\end{enumerate}
\end{theorem}

We note that Theorem~\ref{thm G infinity}\eqref{item local limit} has some interesting consequences. For example, it allows us to explicitly compute the number of copies of any fixed tree within $G_n$. 
While this computation may be cumbersome in general, we do it explicitly to count precisely (up to smaller order terms) the number of cherries (that is, copies of a path on 3 vertices rooted at its central vertex) in dimension 1.

\begin{corollary}\label{cor:cherries}
Suppose $d=1$. Then, a.a.s.\ $G_n$ contains $n+\tfrac{(3-2\log 2 + o(1))n}{12} \approx 1.134 n$ cherries.
\end{corollary}
In addition, the following theorem provides lower and upper bounds for the number of leaves that are independent of the dimension.
\begin{theorem}\label{thm leaves}
A.a.s.\ $((4e^4)^{-1}+o(1))n \le L(G_n) \le (1/2+o(1))n$.
\end{theorem}
 
\noindent
Note that, while the bounds in Theorem~\ref{thm leaves} do not depend on $d$, the number of leaves divided by $n$ presumably converges to a constant depending on $d$ (though we are unable to prove it). 

\paragraph{Overview of proofs.}
The ingredients of the proofs are different and the difficulty varies, having the more difficult proofs concerning $G_n$ towards the end. 
Whereas Theorem~\ref{thm concentration} follows by a nice though relatively simple application of the typical bounded differences inequality from~\cite{War}, 
the proof of Theorem~\ref{thm main 1} concerning the degree sequence is far more involved. 
For the upper bound, we introduce the geometric notion of a \emph{cubic net} which roughly represents a family of disjoint cubes associated to a vertex $v$, each containing at most one child of $v$. 
The difficulty in our proof hides in choosing the size of the cubic net properly in terms of $v$ (note that, in contrast to random geometric graphs, no absolute bound on the Euclidean length of an edge can be imposed). 
To do this, we divide the vertices in $[n-1]_0$ into groups forming intervals $[a,b)$ with $a < b$ satisfying $b/a = O(1)$, and estimate the correct size of the cubic net of $v\in [a,b)$ conditionally on the positions of the already embedded vertices $[a-1]_0$. 
This is done independently for all vertices in $[a,b)$. The lower bound uses the same partition of $V(G_n)$: the main point is that a.a.s., for any two groups of vertices $[a_1,b_1)$ and $[a_2,b_2)$ satisfying $b_1 \le a_2$, a constant proportion of vertices in $[a_1,b_1)$ receive a child from $[a_2,b_2)$. Based on the proof of Theorem~\ref{thm main 1}, Theorem~\ref{cor:maxdegree} follows for free. 
 
The results provided by Theorem~\ref{thm diam} are proved using different techniques. 
The proof of Part\eqref{line a} is simpler and uses the fact that the parent of vertex $i$ is uniformly distributed among $[i-1]_0$. 
Part\eqref{line b} requires a more careful treatment. 
The upper bound on the height follows from a Chernoff-type inequality already used in Part\eqref{line a}. 
The proof of the lower bound is divided in two stages. 
We first show that, for some well-chosen $t$, the Euclidean distance between any vertex and its parent in $G_n$ is ``not too large'' a.a.s.\ (thus showing a kind of localization property for the paths towards the root). 
The second stage relies on an involved argument based on renormalization. 
More precisely, we group the vertices into a large (though constant) number of groups according to their arrival times. 
Starting with a relatively fine tessellation into cubes, in every cube we look for a long path constructed as follows: it starts with an arbitrary vertex $v$ belonging to the last group of vertices, and ends in the nearest ancestor of $v$ from a previous group.
Then, repeat this argument with a coarser tessellation and the next group of vertices, etc., until reaching the group of earliest arrived vertices. 
By suitably gluing some of the paths together, we get the desired lower bound on the height, and showing that there are two disjoint long paths, we get the desired lower bound on the diameter.

The results on the infinite graph make use of previous ideas together with some new ingredients. Proving the local convergence in Theorem~\ref{thm G infinity}\eqref{item local limit} uses a coupling and the fact that the descending tree of a typical vertex in both the finite and the infinite graph is contained a.a.s.\ in a sufficiently large ball.
The proof of the local finitude is simple and relies on cubic nets while the proof of the recurrence is a bit more involved: 
more precisely, by comparison with an auxiliary branching process, we show that there is no infinite path on which the arrival times of the vertices grow. 
This proves that (every connected component of) the graph $G_{\infty}$ can be constructed from the graph $(\mathbb N, \{i, i+1\}_{i\ge 1})$ via attaching finite trees to every integer, and the simple random walk on the latter graph is recurrent. 
Regarding connectivity, we rely on an infinite-volume version of the main result from~\cite{BBCS23}. 
We show that, for every $M > 0$ and $\varepsilon > 0$, all vertices in the ball $B(0,M)$ share a common ancestor with sufficiently small arrival time (depending on $M$ and $\varepsilon$) with probability at least $1-\varepsilon$.

Finally, in the proof of Theorem~\ref{thm leaves},  
the upper bound is the more difficult part. 
It is done by an intricate stochastic comparison between $G_n$ and the uniform attachment tree. 
This comparison is based on the intuitive fact that a vertex of high degree in $G_n$ shall have many vertices close to it, 
which should reduce the probability of the following vertices to attach to it. 
The lower bound is more direct. Roughly speaking, it comes from bounding from below the number of vertices in the second half of $[n-1]_0$ with parents in the first half of the set.

\paragraph{Discussion of the results.} Our results provide interesting insights for the structure of the online nearest neighbor tree: for example, local graph parameters (in the sense of Benjamini-Schramm) such as the number of vertices with given degree apparently depend on the dimension.
Our partial results on cherries corroborate the intuition that the constant in front of the leading term of such local counts should depend on $d$
and Monte-Carlo simulations for checking the number of leaves give evidence to support this conjecture as well. In fact, after the submission of our paper, this was confirmed by Casse~\cite{Cas23} precisely for the number of cherries.
In fact, based on our knowledge on the local limit, explicit integral expressions for counting subgraphs of constant size can be derived. In particular, being equipped with the suitable numerical integration tools, one might be able to analyze in more depth the question of determining the dimension $d$ of the $d$-NN tree asymptotically almost surely given only the combinatorial structure of $G_n$. 
In particular, we conjecture that the number of leaves, as well as the number of cherries, is strictly increasing in the dimension, and therefore both parameters may serve as consistent estimators of $d$. 
(Note that while~\cite{Cas23} provides an integral expression for the cherries, its monotonicity remains an open question.)
On the other hand, results on non-local parameters, such as the typical distance of a random vertex to the root or the diameter of a graph, do not show a first order dependency in terms of the dimension. This makes us wonder if there are non-local parameters that distinguish the dimension.

\begin{question}\label{conj:local}
Is there a graph property that cannot be derived from the local limit of $G_{\infty}$ of $G_n$ but depends in a non-trivial way from the dimension of the ambient space?
\end{question}

\subsection{Related work}\label{relwork}
The current subsection makes a fairly detailed account of related models and problems, and therefore may be consulted in part or skipped on a first read.

The $d$-dimensional $k$-nearest neighbor model belongs to two groups of models that have both received considerable attention: geometric models on the one hand, and models of randomly growing networks on the other hand. For $k = 1$, several authors were interested in the total length of the edges of the resulting random tree.
Its analysis dates back to Steele~\cite{Ste89} who gave a deterministic upper bound on the total length of the edges of the tree when the ambient space is $[0,1]^d$. 
Later, Penrose and Wade~\cite{PW08} showed concentration of the total length of the edges of the nearest neighbor tree,
and Wade~\cite{Wad09} provided a finer upper bound on the variance of this functional when the dimension of the ambient space is sufficiently high. 
Very recently, Trauthwein~\cite{Tra22} went even further by proving a central limit theorem.

Aldous~\cite{Aldous} analyzed the nearest neighbor tree from a different perspective. 
He considered a process which starts with $r\ge 2$ uniformly random points in different colors, and then connects each of the points arriving later (also embedded uniformly at random and independently) to the closest already embedded point, thereby inheriting its color. 
In fact, by giving every point in the square the color of its closest point in the process, this defines a sequence of random colorings of the square.
Aldous showed that the colored regions converge in some sense to a random partition of the plane, leaving topological properties of the regions (such as the question if these are connected) open. 
In the case of two colors, a recent work of Basdevant, Blanc, Curien and Singh~\cite{BBCS23} answered one of Aldous' questions by showing that the colored regions converge in the Hausdorff sense to closed sets with boundary of Hausdorff dimension between $d-1$ and $d$.

Geometric graphs on a vertex set chosen according to some finite density in a bounded subset of $\mathbb{R}^d$ have a long history. 
In the classical model of random geometric graph, the positions of $n$ points, identified with the vertex set, are uniformly chosen in $[0,1]^d$ (or a PPP($n$) is performed), 
and two vertices are connected by an edge if the Euclidean distance between them is at most a certain threshold radius $r=r(n)$. 
This model was introduced in the seminal paper of Gilbert in the early 60's~\cite{Gil61} as a model for telecommunication systems 
(the threshold radius $r$ is interpreted as the maximal distance for the interference of frequencies of antennae). 
By now, this model is very well understood from a mathematical point of view, and we refer the reader to the monograph of Penrose~\cite{Pen} for a detailed account. 
In order to make the model more flexible and to account for devices / antennae that are more powerful than others, the restriction of having a fixed connectivity threshold was then replaced by random radii, having an edge between two vertices if the distance is less than the sum of the associated radii (see e.g.\ the work of Meester and Roy~\cite{MR96}; for recent results on the infinite version of such a model, known as continuum percolation, also see the paper of Ahlberg, Tassion and Teixeira~\cite{ATT16}). 
For further generalizations to Poisson Point Processes distributed according to a random intensity measure, called Cox point processes, 
we refer to the recent paper of Jahnel and T\'{o}bi\'{a}s~\cite{Coxprocesses} and the references therein.

Now, we present the model of the $d$-dimensional $k$-nearest neighbor graph. 
The vertex set of the graph is given by $n$ uniformly distributed points in $[0,1]^d$ or a PPP(1) in $[0,n^{1/d}]^d$ (again, the second definition allows extension to $\mathbb R^d$) and each vertex has $k$ outgoing edges to its $k$ nearest neighbors. 
After ignoring orientations and identifying parallel edges, we obtain a graph with minimum degree $k$ but unbounded maximum degree (in terms of $k$). 
On the one hand, Xue und Kumar in~\cite{Kumar}, and also Balister, Bollob\'{a}s, Sarkar and Walters showed that the critical value of $k$ for dimension $d=2$ to obtain a connected graph is $k=\Theta(\log n)$ (see~\cite{BBSW1,  BBSW2}). 
By a result of H\"{a}ggstr\"{o}m and Meester~\cite{HM} it is known that, for any integer $d$, the 1-nearest neighbor graph has no giant component (that is, no component of linear size), whereas for $d$ sufficiently large, the 2-nearest neighbor graph does have a giant component. 
Moreover, for all $d$, there exists a giant component if $k$ is chosen sufficiently large (see the paper of Teng and Yao~\cite{TY07}). 
The exact threshold value in terms of $k$ for the existence of a giant component in dimension $d=2$ is perhaps the most studied problem due to its connection to the percolation threshold in the infinite setup. 
In the aforementioned paper, Teng and Yao~\cite{TY07} showed that the 213-nearest neighbor graph contains a giant component, 
and this was later improved by Bagchi and Bansal~\cite{Bagchi} to show the same result for the 183-nearest neighbor graph. More than ten years ago, Balister and Bollob\'{a}s~\cite{BB11} pointed out that the 11-nearest neighbor graph already contains a giant component.
They also made Monte-Carlo simulations indicating that the smallest value of $k$ for which there is a giant component in two dimensions is $k=3$. 
At the same time, the 1-nearest neighbor graph in fact has a linear number of components for all $d$ (see Eppstein, Paterson and Tao~\cite{Eppstein}). 
In the version where an edge between $u$ and $v$ is present if simultaneously $u$ is among the $k$ nearest neighbors of $v$ and $v$ is among the $k$ nearest neighbors of $u$, it is known that for $k=2$ and $d=2$, the resulting bidirectional 2-nearest neighbor graph does not have a component of linear size, see a paper of Jahnel and T\'{o}bi\'{a}s~\cite{JahnelTobias}.
 
The motivation for studying randomly growing networks is twofold. 
On the one hand, in numerous applications networks grow over time and one wants to infer information about the entities of the network in general, and about the source of the network in particular. 
Examples include rumor spreading in social networks (see for example~\cite{Fanti1, Fanti2, Fanti3, Shah2}), finding the source of a computer virus in a telecommunication network (see~\cite{Shah}) and the evolution of biological networks~\cite{Navlakha}, to name a few. 
On the other hand, the rigorous study of randomly growing networks has also its own mathematical interest. 
The arguably most famous growing model for networks is the preferential attachment model (the PA model for short) formalized by Barab\'{a}si and Albert~\cite{BarAlb} (introduced in fact already much earlier by Yule to explain the power-law distribution of certain plants~\cite{Yule}): 
starting with a small seed graph (usually consisting of a single vertex), new vertices are added to the graph one after the other, 
and the probability that the newly arrived vertex $v$ connects to a particular vertex $u$ is proportional to some function of the degree of $u$ at the moment of arrival of $v$. 
The definition may or may not allow loops and multiple edges; irrespectively, the principle is that vertices already having high degree are more likely to attach by edges to newly arrived ones, which illustrates the well-known \emph{rich-get-richer} principle. 
This results in a degree sequence that follows a power law whose exponent depends on the precise connection probability function (see e.g.~\cite{BRST, BBBCR03} and also Chapter 8 of Volume 1 of Van der Hofstad's book~\cite{Remco1} for more general results). 
The maximum degree of the PA graph with linear attachment rule is roughly $\Theta(\sqrt{n})$ (see the paper~\cite{Flax} by Flaxman, Frieze and Fenner); 
for results on the maximum degree in a more general context, see also the paper of M\'{o}ri~\cite{Mori}. 
If one allows loops and each new vertex connects to $m$ vertices at its arrival, it is known that the PA graph is not connected for $m=1$ (in fact there are $\Theta(\log n)$ connected components a.a.s.), whereas for $m=2$, the PA graph is connected a.a.s.\ (see the paper of Bollob\'{a}s and Riordan~\cite{BollobasRiordan}). 
The last paper also shows that for $m \ge 2$, the diameter of the original model is a.a.s.\ $\Theta(\log n/\log \log n)$. 
Let us remark that Frieze, Pra\l{}at, P\'{e}rez-Gim\'{e}nez and Reiniger~\cite{PawelHC} and Acan~\cite{Acan} also show that, for all sufficiently large $m$, the PA graph (on an even number of vertices) a.a.s.\ has a perfect matching, and for even larger $m$, it also contains a Hamilton cycle. 
Once again, in a more general setup, depending on the precise connection probability, diameters and typical distances in the PA model vary between $\Theta(\log n)$, $\Theta(\log n/\log \log n)$ and $\Theta(\log \log n)$, see Volume 2 of the monograph of Van der Hofstad~\cite{Remco2}. 
The case of the PA model with $m=1$ without loops and parallel edges yields a tree and was studied in even more detail.
For example, the depth of node $n$ and the height of such a tree are again a.a.s.\ of logarithmic order (see the paper of Mahmoud~\cite{Mahmoud1992} for the leading constant in front of the depth of node $n$ of such a tree, and the paper of Pittel for the leading constant of the height and the diameter of this tree~\cite{Pittel1994}). 
For further results on this model, we refer the reader to Chapters~15.3 and~15.4 of the book of Frieze and Karo\'{n}ski~\cite{FKRandom} and the references therein, and for the height of more general weighted tree models, see also the papers of S\'{e}nizergues, and of S\'{e}nizergues and Pain~\cite{Seni, SeniPain}.

Another similar model in this framework is the uniform attachment model (or the UA model for short): starting again with a small seed graph (usually consisting of a single vertex), 
vertices arrive one after the other, and each newly arrived vertex connects by an edge to a uniformly chosen set of $m$ of the previous vertices (again, loops and multiple edges might or might not be are allowed). 
Once again, for a large constant $m$, the UA graph on an even number of vertices a.a.s.\ has a perfect matching, and for even larger $m$, a.a.s.\ it has a Hamilton cycle~\cite{PawelHC, Acan}. 
The particular case of $m=1$ without allowing loops, known as the uniform attachment tree or also the random recursive tree (the second term assumes that one starts with exactly one vertex), received particular attention due to its connection to sorting algorithms in computer science. 
Devroye~\cite{Dev1988} and Mahmoud~\cite{Mah1991} showed that the depth of the last node is a.a.s.\ $(1+o(1))\log n$, whereas Devroye~\cite{Dev1987} and Pittel~\cite{Pittel1994} showed that the height of the tree is a.a.s.\ $(e+o(1))\log n$. 
Devroye, Fawzi and Fraiman~\cite{DFF} later generalized these results and showed that the height remains logarithmic for more general attachment rules, with the leading constant depending on the particular rule, 
and in the already mentioned paper~\cite{SeniPain} S\'{e}nizergues and Pain generalized this to weighted recursive trees. 
For various statistics of UA trees, we refer the reader to the book of Drmota~\cite{Drmota2009} and Chapter 15.2 in the book of Frieze and Karo\'{n}ski~\cite{FKRandom}. 
Let us point out that the UA tree can be intuitively seen as the online $d$-NN tree that we study here when $d=\infty$: 
roughly speaking, when $d=\infty$, the two models coincide since every newly arrived vertex has infinitely many new degrees of freedom, 
and since there are only finitely many pairs of vertices, they impose constraints that are described by finitely many dimensions at any moment. 
In particular, every new vertex has each of the previously arrived ones as its nearest neighbor with the same probability.

The infinite setup considered here has similarities with the radial spanning tree (RST) on a Poisson Point Process considered by Baccelli and Bordenave~\cite{Baccelli} defined as follows: 
Sample a PPP(1) $\mathcal N$ in $\mathbb{R}^2$. The RST has vertex set $\mathcal N\cup (0,0)$, and every vertex $v\in \mathcal N$ is connected by an edge to the nearest vertex $u\in (\mathcal N\cup (0,0))\setminus v$ satisfying $|u| < |v|$. 
It is clear that this construction yields a tree a.s.\ (every vertex different from the origin connects to a vertex closer to the origin), and Baccelli and Bordenave showed that this tree is a.s.\ locally finite (i.e.\ no vertex has infinite degree), 
that a.s.\ every semi-infinite path starting from the origin has an asymptotic direction (that is, all paths of vertices $(0,0), Y_1, Y_2, \dots$ with $Y_i \in \mathcal N$ are such that $(Y_n / |Y_n|)_{n\ge 1}$ a.s.\ has a limit in the unit sphere), 
that a.s.\ for every point $u$ on the unit sphere there exists at least one semi-infinite path with asymptotic direction $u$, 
that a.s.\ the set of points $u$ on the unit sphere having more than one semi-infinite path with asymptotic direction $u$ is dense on the sphere, 
that a.s.\ for all $k$ large enough, the set of points at graph distance at most $k$ from the root is contained in a ball of radius $\Theta(k)$ around the origin, 
and they also described the distribution of the length of an edge between a vertex $X \in \mathcal N$ and its parent and other functionals. 
For further results on the length of a path to the origin, see also the paper of Bordenave~\cite{Charles2}. 
Also, Baccelli, Coupier and Tran showed in~\cite{Coupier} that the expected number of edges intersecting the sphere of radius $r$ around the origin is $o(r)$ as $r\to \infty$ (see also the later paper of Coupier~\cite{Coupier2} for further results). 
There are a few major differences with our model: first, in our case, vertices arrive online, second, it is not obvious that our model yields a tree (we show this later on), and moreover, there is no distinguished direction towards the origin.

Although our paper is not directly related to reconstructing the root of stochastically growing networks, this question has been studied a lot in machine learning. 
Recently, R\'{a}cz and Sridhar~\cite{Racz} looked at the following general question: consider any model of randomly growing graphs, such as uniform attachment, preferential attachment or the model considered here. 
Given such a model, a pair of graphs $(G_1, G_2)$ is grown in two stages: until time $t_*$ they are grown together (i.e.\ $G_1 = G_2$), after which they grow independently according to the underlying growth model. 
They showed that, whenever the seed graph has an influence on the final graph, correlation can be detected even if $G_1$ and $G_2$ are coupled for only one single time step after the seed. In fact,~\cite{Racz} generalized some previous papers: Bubeck, Mossel and R\'{a}cz~\cite{Bubeck} studied the influence of a seed in PA models. 
They used statistics based on the maximum degree to show that, for any two seed trees $S$ and $T$ on at least three vertices with different degree profiles, the total variation distance between the random graphs, built from $S$ and from $T$ respectively, is positive. 
Moreover, Curien, Duquesne, Kortchemski and Manolescu~\cite{Nicolas} showed that the same holds for any two non-isomorphic trees $S$ and $T$ on at least three vertices, and an analog of the same result was shown by Bubeck, Eldan, Mossel and R\'{a}cz in~\cite{Bubeck2} for the model of UA trees using a certain centrality statistic.

A similar question was studied by Bubeck, Devroye and Lugosi~\cite{Adam} for both UA and PA models: using a centrality measure based on subtree sizes, for every $\varepsilon > 0$, the authors found a set of $K = K(\varepsilon)$ vertices containing the root with probability at least $1-\varepsilon$ (note that $K$ is independent of the size of the final tree). 
More precisely, they prove that, in the UA model the optimal $K$ is subpolynomial in $1/\varepsilon$, whereas in the PA model, it is polynomial in $1/\varepsilon$ (in fact, they give almost tight bounds on $K$ in both models). 
This centrality measure was then also used by Lugosi and
Pereira~\cite{Gabor2} and by Devroye and Reddad~\cite{Luc2} for the more general problem of obtaining confidence intervals for the seed graph of a UA tree; see also~\cite{Khim} for related results. 
Moreover, Jog and Loh~\cite{Jog1, Jog2} show that the centroid, as defined by~\cite{Adam}, changes only finitely many times in the process of constructing the tree. 
For root-finding algorithms in a more general class of random growing trees, see also the paper of Banerjee and Bhamidi~\cite{BanBhamidi}.

Related questions on the following broadcasting model were studied by Addario-Berry, Devroye, Lugosi and Velona~\cite{broadcasting}: 
Assign a random bit (in $\{0,1\}$) to the root and start growing either the UA tree or a linear PA tree. 
At each step, when attaching a vertex, retain the bit of its parent with probability $q$ and flip it with probability $1-q$ for some $q \in [0,1]$. 
The goal is to estimate the bit of the root by either observing the bits of all vertices (since the root's identity is not known, also the root bit is revealed) or, in a harder version, only by observing the bits of all leaves. In~\cite{broadcasting}, the authors bound the probability of error for the root bit in both versions by using majority votes as well as the idea of centrality as above.

\paragraph{Organization of the paper.}
In Section~\ref{section:prelim} we introduce tools and concepts that are used in later sections, 
and in Section~\ref{sec:concentration} we prove the general concentration result (Theorem~\ref{thm concentration}). 
Section~\ref{sec degrees} deals with the results on the degree sequence and the maximum degree (Theorems~\ref{thm main 1}~and~\ref{cor:maxdegree}), 
and in Section~\ref{sec trees}, we show the results on the number of leaves and simulations on them (Theorem~\ref{thm leaves}). 
In Section~\ref{sec distance} we then prove the results on the typical distance to the root, the height and the diameter of $G_n$ (Theorem~\ref{thm diam}). 
Next, in Section~\ref{sec infty}, we provide the proofs of the results concerning the infinite graph (Theorem~\ref{thm G infinity}) and derive Corollary~\ref{cor:cherries}. Finally, in Section~\ref{sec conclusion}, we conclude with several open questions.

\section{Preliminaries and key concepts}\label{section:prelim}
\subsection{Basic tools}
We begin this section with a variant of \emph{Chernoff's bound}, see e.g.\ Corollary 2.3 in~\cite{JLR} in the case of the Binomial distribution and~\cite{Can} in the case of the Poisson distribution. 
Let $X \sim \textrm{Bin}(n,p)$ (resp.\ $X \sim \textrm{Po}(\lambda)$) be a random variable distributed according to a Binomial distribution with parameters $n$ and $p$ (resp.\ Poisson distribution with parameter $\lambda$). 
Then, for every $t \ge 0$, we have 
\begin{eqnarray}
\max\{\mathbb P(X \ge \mathbb E[X] + t ), \mathbb P(X \le \mathbb E[X] - t )\} &\le& \exp \left( - \frac{t^2}{2 (\mathbb E[X] + t)} \right). \label{chern}
\end{eqnarray}

The next lemma is famous under the name \emph{bounded difference inequality} or also \emph{McDiarmid's inequality}.

\begin{lemma}[\cite{JLR}, Corollary 2.27]\label{bentkus}
Let $X_1,X_2,\dots,X_n$ be independent random variables with $X_k$ taking values in $\Lambda_k\subseteq \mathbb R$. Assume that a function $f:\Lambda_1\times \dots \times \Lambda_n \to \mathbb R$ satisfies the following Lipschitz condition for some numbers $(c_k)_{k=1}^n$:

\begin{align}\label{Lip cond}
&\text{For any } k\in [n] \text{ and any two vectors } \mathbf{x_1, x_2}\in \Lambda_1\times \dots \times \Lambda_n \nonumber\\ 
&\text{that differ only in the } k \text{-th coordinate, }|f(\mathbf{x_1})-f(\mathbf{x_2})|\le c_k.
\end{align}

\noindent
Then, for every $t \ge 0$, the random variable $Z = f(X_1, \dots, X_n)$ satisfies
\begin{align}
    & \mathbb P(Z\ge \mathbb E Z +t)\le \exp\left(-\dfrac{t^2}{2\sum_{k=1}^n c_k^2}\right),\label{eq:bdd1}\\
    & \mathbb P(Z\le \mathbb E Z -t)\le \exp\left(-\dfrac{t^2}{2\sum_{k=1}^n c_k^2}\right).\label{eq:bdd2}
\end{align}
\end{lemma}
In fact, we note that~\eqref{eq:bdd1} and~\eqref{eq:bdd2} can be strengthened (the factor 2 can be put in the numerator) but the weaker version we present is sufficient for our purposes.

The following generalization of the bounded difference inequality (Lemma~\ref{bentkus}) from~\cite{War} will be key in the proof of Theorem~\ref{thm concentration}. 
The difference with Lemma~\ref{bentkus} is that, with small probability, exceptionally large differences are allowed.

\begin{theorem}[\cite{War}, Theorem 2 (two-sided version)]\label{typical BDI}
Let $X_1, X_2, \dots, X_n$ be independent random variables with $X_k$ taking values in $\Lambda_k$ and $X = (X_1, \ldots, X_n)$. Let $\Gamma\subseteq \prod_{i\in [n]} \Lambda_i$ be an event and assume that the function $f:\prod_{i\in [n]} \Lambda_i \to \mathbb R$ satisfies the following typical Lipschitz condition:\\

$\mathrm{(TL):}$ There are numbers $(c_k)_{k\in [n]}$ and $(d_k)_{k\in [n]}$ with $c_k\le d_k$ such that, whenever $x, \tilde{x}\in \prod_{i\in [n]} \Lambda_i$ differ only in the $k$-th coordinate, we have
\[|f(x) - f(\tilde{x})|\le \left\{
\begin{array}{ll}
c_k, \text{ if } x\in \Gamma,\\
d_k, \text{ otherwise.}
\end{array}
\right.\]

Then, for any numbers $(\gamma_k)_{k\in [n]}$ with $\gamma_k\in (0,1]$, there is an event $\mathcal B = \mathcal B(\Gamma, (\gamma_k)_{k\in [n]})$ satisfying
\begin{equation*}
    \mathbb P(\mathcal B)\le \mathbb P(X\not\in \Gamma) \sum_{k\in [n]} \gamma^{-1}_k \text{ and } \overline{\mathcal B}\subseteq \Gamma,
\end{equation*}
and such that for $\mu = \mathbb E[f(X)], e_k = \gamma_k(d_k - c_k)$ and every $t\ge 0$, we have
\begin{equation*}
    \mathbb P(|f(X) - \mu|\ge t \text{ and } \overline{\mathcal B})\le 2\exp\left(-\dfrac{t^2}{2\sum_{k\in [n]} (c_k + e_k)^2}\right).
\end{equation*}
\end{theorem}

We also need the following bounds on the tails of sums of i.i.d.\ exponential random variables:
\begin{lemma}[\cite{levy2021density}, Section 5]\label{lem erlang}
Let $(X_i)_{i\ge 1}$ be a family of i.i.d.\ exponential random variables of mean $\lambda$. Then, for every $k\ge 1$ and $\xi > 0$, we have that
\begin{equation*}
\mathbb P(X_1+\dots+X_k\ge \xi) = \sum_{s=0}^{k-1}\dfrac{1}{s!} \exp\left(-\dfrac{\xi}{\lambda}\right) \left(\dfrac{\xi}{\lambda}\right)^s.
\end{equation*}

\noindent
In particular: 
\begin{enumerate}[(i)]
    \item if $k\ge 2\xi/\lambda$, then 
\begin{equation*}
\mathbb P(X_1+\dots+X_k\le \xi)\le \dfrac{2}{k!} \exp\left(-\dfrac{\xi}{\lambda}\right) \left(\dfrac{\xi}{\lambda}\right)^k.
\end{equation*}
    \item if $k\le \xi/2\lambda$, then
\begin{equation*}
\mathbb P(X_1+\dots+X_k\ge \xi) \le \dfrac{1}{k!} \exp\left(-\dfrac{\xi}{\lambda}\right) \left(\dfrac{\xi}{\lambda}\right)^k.
\end{equation*}
\end{enumerate}
\end{lemma}

\subsection{Palm theory}\label{sec:Palm}
Palm theory is one of the main tools for studying Poisson Point Processes. 
The \emph{Palm distribution} of a stationary point process $\mathcal N$ can be interpreted as the conditional distribution given that $0\in \mathcal N$. 
It is well-known that a stationary point process is Poisson if and only if its Palm distribution is the distribution of the original process with an extra point added at the origin (see for example Chapter 9 of~\cite{Last}). 
While the above fact has some deep consequences, we only use one simple corollary:

\begin{theorem}[\cite{Pen}, Theorem 1.6]\label{Palm:Penrose}
Fix $\lambda > 0$ and a Poisson Point Process $P_{\lambda}$ on a finite measurable set $\Omega\subseteq \mathbb{R}^d$. 
Also, fix $j \in \mathbb{N}$ and a bounded measurable function $h(Y,X)$ defined on all pairs of the form $(Y,X)$ with $X$ a finite subset of $\Omega$ and $Y$ a subset of $X$. 
Suppose that $h(Y,X)=0$ except when $Y$ has $j$ elements. 
Then,
$$
\mathbb E \left[\sum_{Y \subseteq P_{\lambda}} h(Y, P_{\lambda})\right] = \frac{\lambda^j}{j!}\mathbb E\, h(\chi_j, \chi_j \cup P_{\lambda}),
$$
where the sum on the left-hand side is over all subsets $Y$ of the random point set $P_{\lambda}$ and, on the right-hand side, the set $\chi_j$ corresponds to a set of $j$ points distributed independently, uniformly in $\Omega$, and also independently of $P_{\lambda}$. (Note that the expectation on the left is with respect to the probability measure associated to $P_{\lambda}$, while the expectation on the right is with respect to the joint measure associated to $(\chi_j, P_{\lambda})$.)
\end{theorem}

\subsection{Cubic nets}\label{subsec cubic nets}
We now introduce the key concept of a \emph{cubic net}. Roughly speaking, it represents a tessellation of $\mathbb T^d_n$, centered at a given point, into axis-parallel cubes of varying sizes. 
We use this tool to bound from above the degree of a vertex embedded at the center of the cubic net.

More formally, for a point $x\in \mathbb T^d_n$ and a real number $\ell > 0$, $\ell\le \ell_{max} = n^{1/d}/2$, 
define $W(x, \ell) = \prod_{i\in [d]} [x_i-\ell, x_i+\ell]$, which we call \emph{the $\ell$-window around $x$}. In particular, the $\ell_{max}$-window around $x$ covers the entire torus $\mathbb T^d_n$ for every point $x\in \mathbb T^d_n$.
Fix $K = K(d) = 2\lceil \sqrt{d}\rceil + 3$. 
Now, tessellate $W(x, \ell)$ into subcubes of side length $2\ell/K$. All subcubes incident to the boundary of $W(x,\ell)$ are called \emph{boundary cubes}. 
Now, remove the annulus of boundary cubes of $W(x, \ell)$ and consider the following operation with the remaining \emph{central box} of side length $2 \ell - \frac{4\ell}{K}=\frac{2(K-2)\ell}{K}$:
first, merge all smaller cubes that the central box contains. 
If the side length of the central box is at most $1/2$, then stop. 
Otherwise, tessellate the central box into smaller subcubes of side length $(2 \ell - 4\ell/K)/ K$ (in other words, at each step, we tessellate the resulting central box into $K^d$ subcubes). 
All smaller subcubes incident to the boundary of the central box are put aside, and the remaining subcubes are merged again to form a new central box, for which the operation then is applied again. 
By definition, this operation is applied $\alpha_{\ell} = \alpha(\ell, d) := \lceil \log(2\ell)/\log(K/(K-2))\rceil$ many times. 
All cubes obtained throughout the procedure together form the \emph{cubic net} $N(x,\ell)$. In particular, $N(x,\ell)$ contains:
\begin{itemize}
    \item a central cube containing $x$ of side length at most $1/2$,
    \item boundary cubes, which share a common boundary with $W(x, \ell)$, and
    \item outer cubes, which are all cubes except the central one. In particular, boundary cubes are also outer cubes.
\end{itemize}

\begin{figure}
\centering
\begin{tikzpicture}[line cap=round,line join=round,x=1cm,y=1cm]
\clip(-10.70579928738711,-3.8283301786510413) rectangle (5.774852934823306,4.448477588704503);
\draw [line width=0.8pt] (-6,4)-- (-6,-3);
\draw [line width=0.8pt] (-6,-3)-- (1,-3);
\draw [line width=0.8pt] (1,-3)-- (1,4);
\draw [line width=0.8pt] (1,4)-- (-6,4);
\draw [line width=0.8pt] (-4,4)-- (-4,-3);
\draw [line width=0.8pt] (-3,4)-- (-3,-3);
\draw [line width=0.8pt] (-2,4)-- (-2,-3);
\draw [line width=0.8pt] (-1,-3)-- (-1,4);
\draw [line width=0.8pt] (1,3)-- (-6,3);
\draw [line width=0.8pt] (-6,1)-- (1,1);
\draw [line width=0.8pt] (1,0)-- (-6,0);
\draw [line width=0.8pt] (-6,-2)-- (1,-2);
\draw [line width=0.8pt] (-5,4)-- (-5,-3);
\draw [line width=0.8pt] (0,4)-- (0,-3);
\draw [line width=0.8pt] (-6,2)-- (1,2);
\draw [line width=0.8pt] (-6,-1)-- (1,-1);
\draw [line width=0.8pt] (-2.5,3)-- (-2.5,0.5);
\draw [line width=0.8pt] (-3,4)-- (-2,3);

\draw [decorate,decoration={brace,amplitude=6pt},xshift=0pt,yshift=0pt]
(1,-3) -- (-6,-3) node {};
\begin{scriptsize}
\draw [fill=black] (-2.5,0.5) circle (0.8pt);
\draw [fill=black] (-2.5,3) circle (0.8pt);

\draw [fill=black] (-2.7,0.5) node {\large{$x$}};
\draw [fill=black] (-2.5,-3.5) node {\large{$2\ell$}};
\end{scriptsize}
\end{tikzpicture}
\caption{The case $d = 2$. Since $K(2) = 7$, the division operation divides the square $W(x,\ell)$ into a $7\times 7$ grid and then merges the $25$ non-boundary squares of side length $2\ell/7$. Our choice of $K$ ensures that the diagonal of a square in the grid is shorter than the distance from $x$ to the boundary squares of $N(x,\ell)$.}
\label{fig 2}
\end{figure}
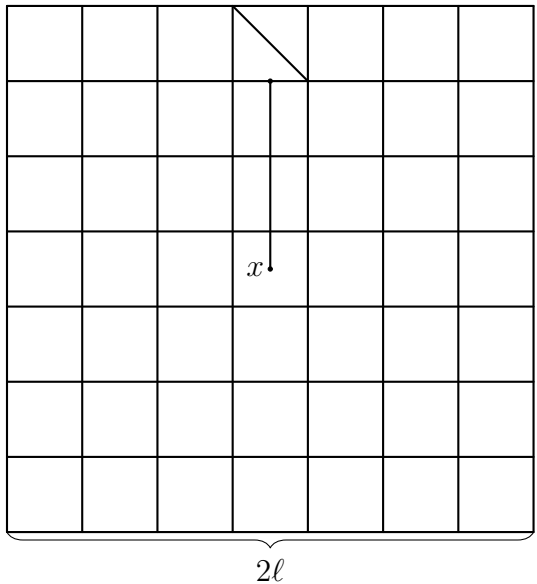

Note that by definition $K\ge 5$ for every dimension. 
Hence, one may readily check that, for every $x\in \mathbb T^d_n$ and $\ell\ge 1$, the number of outer cubes in $N(x, \ell)$ is 
\begin{equation}\label{eq degree}
   \Psi \alpha_{\ell} := (K^d - (K-2)^d) \alpha_{\ell} = O(\log \ell).
\end{equation}

The following observation bounds from above the number of neighbors of a vertex using cubic nets.

\begin{observation}\label{ob distance}
Given a vertex $v\in V(G_n)$ and $\ell\in [1,\ell_{max}]$, $v$ connects by an edge of $G_n$ to at most one vertex in every outer cube in $N(v, \ell)$. Moreover, if there is one vertex embedded prior to $v$ in every boundary cube of $N(v, \ell)$, then $v$ has at most one neighbor outside $W(v, \ell)$.
\end{observation}
\begin{proof}
For the first statement, it is sufficient to prove that the distance from $v$ to any point in a given outer cube is larger than the diameter of the outer cube itself. Since every outer cube in $N(v, \ell)$ is a boundary cube in $N\left(v, \left(\frac{K-2}{K}\right)^i \ell\right)\subseteq N(v,\ell)$ for some $i\in [\alpha_{\ell}-1]_0$, we show this claim only for boundary cubes. In this case, the distance from $v$ to the boundary of $W(v, (K-2)\ell/K)$ is $(K-2)\ell/K$, and the diameter of any boundary cube of $N(v, \ell)$ is $\sqrt{d} (2\ell/K) < (K-2)\ell/K$ by definition of $K$.

For the second statement, we show that every child of $v$ in $G_n$ (seen as a tree rooted at $0$) is in $W(v,\ell)$. Indeed, if a vertex $u$ is outside $W(v, \ell)$, then choose any point $x$ on the segment $uv$ inside some boundary cube $q$ of $N(v, \ell)$. Let $w$ be a vertex of $G_n$ in $q$ embedded prior to $v$. Then, by the triangle inequality for $uwx$, the fact that $\diam_{E}(q) < |vx|$ and the collinearity of $v, x$ and $u$,
\begin{equation}\label{eq triangle ineq}
    |wu|\le |wx|+|xu|\le \diam_{E}(q)+|xu| < |vx|+|xu| = |vu|,
\end{equation}
thereby proving the observation.
\end{proof}

\section{Proof of Theorem~\ref{thm concentration}}\label{sec:concentration}
In this section, we use Theorem~\ref{typical BDI} for the family of independent random variables $(U_i)_{i\in [n-1]_0}\cup (X_i)_{i\in [n-1]_0}$ where $U_i\sim \mathrm{Unif}(\mathbb T^d_n)$ and $X_i\sim \mathrm{Unif}[0,1]$.
Fix an integer function $\phi = \phi(n)\in [C_0 \log n, n/3]$ for some suitably large $C_0 > 0$ (which is well-defined for all sufficiently large $n$), and denote by $\Gamma$ the event that, 
for every vertex $v\in V(G_n)$ and every point $x\in \mathbb T^d_n$, there are at most $\phi$ vertices $u\in V(G_n)\setminus v$ for which $d_{\mathbb T^d_n}(u,x)\le d_{\mathbb T^d_n}(u, \{w\in V(G_n)\setminus v: X_w < X_u\})$. 
Otherwise said, $\Gamma$ is equivalent to the non-existence of a pair $(v,x)\in V(G_n)\times \mathbb T^d_n$ such that a vertex $v_x$ at position $x$ that arrives before all vertices in $V(G_n)\setminus v$ would have degree more than $\phi$ in the $d$-NN tree generated by $v_x\cup (V(G_n)\setminus v)$ (where the order of arrival of the vertices in $V(G_n)\setminus v$ is the same as in $G_n$). 
Note that, in particular, the event $\Gamma$ includes the event that the maximum degree of $G_n$ is at most $\phi+1$ 
(since we allow the position of $x$ to coincide with the position of any of the vertices of $G_n$). 

\begin{observation}\label{ob sec 4}
For every large enough $n$, $\mathbb P(\Gamma) \ge 1-\exp\left(-\frac{\phi \log \phi}{5}\right)$.
\end{observation}
\begin{proof}
Recall the definition of a cubic net introduced in Section~\ref{section:prelim}. 
By construction, the vertex $v_x$ in position $x$ may have at most one neighbor in each of the $\Psi\alpha_{n^{1/d}/2} = O(\log n)$ outer cubes of the net $N(x,n^{1/d}/2)$. 

We prove that, by choosing $C_0$ sufficiently large, a.a.s.\ there is no axis-parallel cube of side length $1/2$ containing $t = \lfloor \phi/2\rfloor$ vertices of $G_n$. 
Set $n_1 = \lfloor n^{1/d}\rfloor$ and fix a tessellation $\mathcal F$ in $\mathbb T^d_n$ into cubes of side length $s = n^{1/d}/n_1$. Note that every axis-parallel cube of side length $1/2$ is contained in a cube of side length $s$ in one of the tessellations, obtained by translation of $\mathcal F$ by some vector among $(s/2)\cdot \{0,1\}^d$. Since in total there are $2^d\cdot (n_1)^d = (2^d+o(1))n$ such cubes, a union bound gives
\begin{align*}
\mathbb P(\overline{\Gamma})
&\le\hspace{0.3em} (2^d + o(1))n \mathbb P(|[0,s]^d\cap V(G_n)|\ge t)\\ 
&=\hspace{0.3em} (2^d+o(1))n \exp(-(1+o(1))) \dfrac{(1+o(1))^{t}}{t!}\\
&\le\hspace{0.3em} \exp\left(-\frac{t \log t}{2}\right)\le \exp\left(-\frac{\phi \log \phi}{5}\right),
\end{align*}
and the proof of the lemma is finished.
\end{proof}

\noindent
Note that Observation~\ref{ob sec 4} already shows the upper bound on the maximum degree in Theorem~\ref{cor:maxdegree}. Next, we use it to derive Theorem~\ref{thm concentration}.

\begin{proof}[Proof of Theorem~\ref{thm concentration}:]
We apply Theorem~\ref{typical BDI} for the family $(U_i)_{i\in [n-1]_0}\cup (X_i)_{i\in [n-1]_0}$ and the event $\Gamma$ defined above. 
Moreover, since the tree $G_n$ is a deterministic function of the above random variables (ties being broken in an arbitrary way), define
\begin{equation*}
    f:((u_i)_{0\le i\le n-1},(x_i)_{0\le i\le n-1})\in (\mathbb T^d_n)^n\times [0,1]^n\mapsto g\left(G_n \text{ generated by } \{U_i = u_i, X_i = x_i\}\text{ for } i\in [n-1]_0\right).
\end{equation*}

Then, for every vector $(\mathbf{u}_1, \mathbf{x}_1)\in \Gamma$ and every $(\mathbf{u}_2, \mathbf{x}_2)\in (\mathbb T^d_n)^n\times [0,1]^n$ which differs from $(\mathbf{u}_1, \mathbf{x}_1)$ in exactly one coordinate, one has by definition of $\Gamma$ that 
$$|f(\mathbf{u}_1, \mathbf{x}_1) - f(\mathbf{u}_2, \mathbf{x}_2)|\le c := L(2\phi+2)\le Ln.$$ 
Indeed, if an edge $e$ is in the symmetric difference of the trees generated by $(\mathbf{u}_1, \mathbf{x}_1)$ and $(\mathbf{u}_2, \mathbf{x}_2)$, 
then it is incident to the vertex whose arrival time or position has been modified, and in both trees this vertex has degree at most $\phi+1$. 
If $(\mathbf{u}_1, \mathbf{x}_1)$ does not necessarily belong to $\Gamma$, then
\begin{equation*}
    |f(\mathbf{u}_1, \mathbf{x}_1)-f(\mathbf{u}_2, \mathbf{x}_2)|\le Ln.
\end{equation*}

Now, set $\gamma_k = \gamma := (Ln)^{-1}$. 
We conclude that $e_k = e := \gamma(Ln-c)\le 1$ and thus $(c+e)^2\le (2L\phi +2L+1)^2 \le 5 L^2 \phi^2$. We deduce by Theorem~\ref{typical BDI} that there is an event $\mathcal B$ with probability at most $2n\mathbb P(\overline{\Gamma}) \gamma^{-1} \le \exp(-\tfrac{1}{10}\phi\log \phi)$ such that $\overline{\mathcal B}\subseteq \Gamma$ and, for every $t\ge 0$,
\begin{equation*}
    \mathbb P(|f((U_i)\cup (X_i)) - \mathbb E[f((U_i)\cup (X_i))]|\ge t \text{ and }\overline{\mathcal B})\le 2\exp\left(-\dfrac{t^2}{4n\cdot 5 L^2 \phi^2}\right).
\end{equation*}

Since $\mathbb P(\mathcal B)\le \exp(-\tfrac{1}{10}\phi\log \phi)$, we conclude that
\begin{align*}
\mathbb P[|g(G_n) - \mu|\ge t]
&=\hspace{0.3em} \mathbb P(|f((U_i)\cup (X_i)) - \mathbb E[f((U_i)\cup (X_i))]|\ge t \text{ and }\overline{\mathcal B}) + \mathbb P(\mathcal B)\\
&\le\hspace{0.3em} 2\exp\left(- \dfrac{t^2}{20nL^2\phi^2}\right) + \exp\left(-\frac{\phi\log \phi}{10}\right),
\end{align*}
and the proof is completed.
\end{proof}

\section{The degree sequence and the maximum degree}\label{sec degrees}

In this section, we abuse notation and, unless ambiguity arises, identify a vertex in $V(G_n) = [n-1]_0$ with its position in $\mathbb T^d_n$ (recall that, in general, we also identify a vertex with its arrival time). 
The section is dedicated to proving Theorem~\ref{thm main 1} and the lower bound in Theorem~\ref{cor:maxdegree}.

\subsection{Upper bound on the number of vertices of given degree}
We first give a high level overview of the proof. 
It goes roughly as follows: We embed the vertices of $G_n$ one by one in $\mathbb T^d_n$. 
At every step, we choose a positive real number $\ell = \ell(v)$ appropriately so that no child of $v$ in $G_n$ is situated outside $W(v,\ell)$. 
Note that minimizing the value of $\ell(v)$ is the main difficulty in the proof, and it requires several technical twists and turns. 
After choosing $\ell$, we divide the children of $v$ into two groups: 
those positioned in the central cube of $N(v,\ell)$ and others positioned in outer cubes. 
The number of children of the first type is bounded from above by an easy concentration argument, while for children of the second type, we rely on some geometric properties of cubic nets.

Let us proceed to the details.
We fix $k\ge 0$ and assume that $n = 2^{dk}$ (the idea behind this assumption is to ensure that a subsequent divisibility condition is satisfied and avoid the excessive use of upper and lower integer parts). 
Fix a family of tessellations $(\mathcal F_i)_{i=0}^k$ of side lengths $(2^i)_{i=0}^k$, respectively. Note that, on the event that a unit cube $q$ in $\mathbb T^d_n$ contains $s$ vertices of $G_n$, there exists a cube of $\mathcal F_0$ that contains at least $s/2^d$ of these vertices and intersects $q$. 
Define $s_{max} := \lfloor \log n/3\log\log n \rfloor$. The following easy observation uses standard concentration results to show that ``most unit cubes do not contain too many vertices'' (with the idea to bound the number of neighbors of a vertex in its central cube from above):

\begin{observation}\label{ob close neighbors}
A.a.s.\ for every $s\in [s_{max}]$, the number of vertices $v$ of $G_n$ for which $|W(v, 1/2)\cap V(G_n)|\ge 2^d s$ is at most $3^{d}\cdot en/(s-1)!$.
\end{observation}
\begin{proof}
First, we condition on the a.s.\ event that no vertex lies on the boundary of any of the cubes in $\mathcal F_0$. 
Denote $\mathcal S_s = \{q\in \mathcal F_0: |q\cap V(G_n)|\ge s\}$. Note that every unit cube containing at least $2^d s$ vertices of $G_n$ intersects a cube of $\mathcal S_s$. 
Our goal is to provide an a.a.s.\ upper bound for the number of vertices, embedded at distance at most 1 to $\mathcal S_s$.

Our main tool is the bounded differences inequality (Lemma~\ref{bentkus}). For every cube $q\in \mathcal F_0$, define $g_q: S \in (\mathbb T^d_n)^n \mapsto \mathds{1}_{|q\cap S|\ge s} |q\cap S|$, and let $f: S\in (\mathbb T^d_n)^n \mapsto \sum_{q\in \mathcal F_0} g_q(S)$. 
It is easy to notice that $f$ satisfies the Lipschitz condition~\eqref{Lip cond} with $c_1 = \dots = c_n = s$. Moreover, for every large enough $n$,
\begin{equation*}
    \mathbb E[g_q(V(G_n))] = \sum_{i\ge s} i\binom{n}{i}\dfrac{1}{n^i}\left(1-\dfrac{1}{n}\right)^{n-i} \in \left[\dfrac{1}{3(s-1)!}, \frac{e}{2(s-1)!}\right],
\end{equation*}
where we used that
\[\sum_{i\ge s} i\binom{n}{i}\dfrac{1}{n^i}\left(1-\dfrac{1}{n}\right)^{n-i} \le \frac{1+o(1)}{e}\sum_{i\ge s} \frac{1}{(i-1)!}\le \frac{1}{2(s-1)!}\sum_{j\ge 1} \frac{1}{j!}\le \frac{e}{2(s-1)!}.\]
Consequently,
\begin{equation*}
    \mathbb E[f(V(G_n))] = |\mathcal F_0| \mathbb E[ g_q(V(G_n))] \ge \dfrac{n}{3(s-1)!}.
\end{equation*}

Moreover, by Lemma~\ref{bentkus} applied to $V(G_n)$ and $f$, 
\begin{equation*}
    \mathbb P(f(V(G_n))\ge 2\mathbb E[f(V(G_n)))]\le \exp\left(-\dfrac{n}{18s!^2}\right) = \exp(-n^{\Omega(1)}).
\end{equation*}
Thus, a.a.s.\ the cubes in $\mathcal S_s$ contain at most $en/(s-1)!$ 
vertices. Hence, the number of vertices $v\in V(G_n)$ such that $W(v,1/2)\cap (\cup_{q\in \mathcal S_s} q) \neq \emptyset$ is at most $3^d\cdot (en/(s-1)!)$ with probability $1 - \exp(-n^{\Omega(1)})$: indeed, denoting by $q+\vec{x}$ the translation of the cube $q$ by a vector $\vec{x}$, we get
\begin{equation*}
    |\{v\in V(G_n): W(v,1/2)\cap \left(\cup_{q\in \mathcal S_s} q\right) \neq \emptyset\}|\le \sum_{\vec{x}\in \{-1,0,1\}^d} \sum_{q\in \mathcal S_s} |(q+\vec{x})\cap V(G_n)|\le 3^d \sum_{q\in \mathcal S_s} |q\cap V(G_n)|.
\end{equation*}
A union bound over the complementary events for all values $s\in [s_{\max}]$ finishes the proof.
\end{proof}

Recall that the cubic net $N(v, \ell)$ was designed so that, conditionally on having at least one vertex $w < v$ in every boundary cube of the net, the children of the central vertex may only be situated in $W(v, \ell)$ (with the additional restriction of having at most one child per outer cube). 
With Observation~\ref{ob close neighbors} in hand, it remains to take care of the neighbors of $v$ in the outer cubes of $N(v, \ell)$, and in particular to find the correct choice of $\ell$ to satisfy the above requirement regarding the boundary cubes.

Define $\Gamma(v,\ell)$ to be the event that, at the moment of embedding the vertex $v$ in $\mathbb T^d_n$, there is at least one vertex in every boundary cube of $N(v,\ell)$. 
The difficulty of finding the correct value of $\ell$, for which $\Gamma(v, \ell)$ holds, stems from the fact that the events $(\Gamma(v, \ell))_{v, \ell}$ are not independent for different vertices. 
To circumvent this, divide the vertices into groups of $[2^{(i-1)d}, 2^{id})_{i\in [k]}$ (we discard the root for convenience since the degree of a single vertex does not matter in the sequel). 
We define new events, $\Gamma'(v,\ell)$, as follows: if $v\in [2^{(i-1)d}, 2^{id})$, then $\Gamma'(v, \ell)$ is the event of having at least one vertex from $[2^{(i-1)d}-1]_0$ in every boundary cube of the cubic net $N(v, \ell)$. 
By definition, $\Gamma'(v, \ell)\subseteq \Gamma(v, \ell)$, but this time the events $(\Gamma'(v, \ell))_{v\in [2^{(i-1)d}, 2^{id})}$ are independent conditionally on the positions of the vertices in $[2^{(i-1)d}-1]_0$. 
Another assumption that we impose is to choose $\ell$ of the form $K\cdot 2^{i+1}$ for $i\ge 0$ (recall that $K = 2\lceil\sqrt{d}\rceil + 3$). In this way, we ensure that the side of a boundary cube of $N(v, \ell)$ is $2^{i+1}$.

The main idea is to associate to every point $x\in \mathbb T^d_n$ and every boundary cube $q$ of $N(x,K\cdot 2^{i+1})$ a smaller cube of side length $2^i$ belonging to $\mathcal F_i$ which is completely contained in $q$ (arbitrarily in case there are more choices). 
We show that, in any case, no cube of $\mathcal F_i$ is  associated to ``too many points'' $x\in \mathbb T^d_n$. 
We proceed to the details: for every $i\le \lfloor\log_2(\ell_{max}/K) - 1\rfloor$ and every point $x$ in $\mathbb T^d_n$, 
we associate a family $\mathcal C(x,i)$ of $\Psi = K^d - (K-2)^d$ cubes of $\mathcal F_i$ with the property that every boundary cube of $N(x, K\cdot 2^{i+1})$ contains exactly one cube of $\mathcal C(x,i)$. 
Note that for one fixed $x$, all such cubes are disjoint by construction, and for $x, x' \in \mathbb T_n^d$ close to each other, it is possible that $\mathcal C(x, i) \cap \mathcal C(x',i) \neq \emptyset$, see Figure~\ref{fig 5}. 
For every cube $q\in \mathcal F_i$, set $\mathcal R_i(q)$ to be the region $\{x\in \mathbb T^d_n:\hspace{0.2em} q\in \mathcal C(x,i)\}$. It is easy to see that any choice of $(\mathcal C(x,i))_{x,i}$ implies that $(\mathcal R_i(q))_{i, q\in \mathcal F_{i}}$ are all measurable subsets of $\mathbb T^d_n$. We now show the intuitively clear fact that $|\mathcal R_i(q)|$ is not too large.

\begin{figure}[ht]
\centering
\begin{tikzpicture}[line cap=round,scale=0.7,line join=round,x=1cm,y=1cm]
\clip(-10.9,-6.2) rectangle (12.1,8.2);
\fill[line width=0.8pt,color=black,fill=black,fill opacity=0.10000000149011612] (-6,7.5) -- (-6,6.5) -- (-5,6.5) -- (-5,7.5) -- cycle;
\fill[line width=0.8pt,color=black,fill=black,fill opacity=0.10000000149011612] (-4,6.5) -- (-3,6.5) -- (-3,7.5) -- (-4,7.5) -- cycle;
\fill[line width=0.8pt,color=black,fill=black,fill opacity=0.10000000149011612] (-2,6.5) -- (-1,6.5) -- (-1,7.5) -- (-2,7.5) -- cycle;
\fill[line width=0.8pt,color=black,fill=black,fill opacity=0.10000000149011612] (0,6.5) -- (1,6.5) -- (1,7.5) -- (0,7.5) -- cycle;
\fill[line width=0.8pt,color=black,fill=black,fill opacity=0.10000000149011612] (2,6.5) -- (3,6.5) -- (3,7.5) -- (2,7.5) -- cycle;
\fill[line width=0.8pt,color=black,fill=black,fill opacity=0.10000000149011612] (4,6.5) -- (5,6.5) -- (5,7.5) -- (4,7.5) -- cycle;
\fill[line width=0.8pt,color=black,fill=black,fill opacity=0.10000000149011612] (6,6.5) -- (7,6.5) -- (7,7.5) -- (6,7.5) -- cycle;
\fill[line width=0.8pt,color=black,fill=black,fill opacity=0.10000000149011612] (6,4.5) -- (7,4.5) -- (7,5.5) -- (6,5.5) -- cycle;
\fill[line width=0.8pt,color=black,fill=black,fill opacity=0.10000000149011612] (6,2.5) -- (7,2.5) -- (7,3.5) -- (6,3.5) -- cycle;
\fill[line width=0.8pt,color=black,fill=black,fill opacity=0.10000000149011612] (6,0.5) -- (7,0.5) -- (7,1.5) -- (6,1.5) -- cycle;
\fill[line width=0.8pt,color=black,fill=black,fill opacity=0.10000000149011612] (6,-1.5) -- (7,-1.5) -- (7,-0.5) -- (6,-0.5) -- cycle;
\fill[line width=0.8pt,color=black,fill=black,fill opacity=0.10000000149011612] (6,-3.5) -- (7,-3.5) -- (7,-2.5) -- (6,-2.5) -- cycle;
\fill[line width=0.8pt,color=black,fill=black,fill opacity=0.10000000149011612] (6,-5.5) -- (7,-5.5) -- (7,-4.5) -- (6,-4.5) -- cycle;
\fill[line width=0.8pt,color=black,fill=black,fill opacity=0.10000000149011612] (4,-5.5) -- (5,-5.5) -- (5,-4.5) -- (4,-4.5) -- cycle;
\fill[line width=0.8pt,color=black,fill=black,fill opacity=0.10000000149011612] (2,-5.5) -- (3,-5.5) -- (3,-4.5) -- (2,-4.5) -- cycle;
\fill[line width=0.8pt,color=black,fill=black,fill opacity=0.10000000149011612] (0,-5.5) -- (1,-5.5) -- (1,-4.5) -- (0,-4.5) -- cycle;
\fill[line width=0.8pt,color=black,fill=black,fill opacity=0.10000000149011612] (-2,-5.5) -- (-1,-5.5) -- (-1,-4.5) -- (-2,-4.5) -- cycle;
\fill[line width=0.8pt,color=black,fill=black,fill opacity=0.10000000149011612] (-4,-5.5) -- (-3,-5.5) -- (-3,-4.5) -- (-4,-4.5) -- cycle;
\fill[line width=0.8pt,color=black,fill=black,fill opacity=0.10000000149011612] (-6,-5.5) -- (-5,-5.5) -- (-5,-4.5) -- (-6,-4.5) -- cycle;
\fill[line width=0.8pt,color=black,fill=black,fill opacity=0.10000000149011612] (-6,-3.5) -- (-5,-3.5) -- (-5,-2.5) -- (-6,-2.5) -- cycle;
\fill[line width=0.8pt,color=black,fill=black,fill opacity=0.10000000149011612] (-6,-1.5) -- (-5,-1.5) -- (-5,-0.5) -- (-6,-0.5) -- cycle;
\fill[line width=0.8pt,color=black,fill=black,fill opacity=0.10000000149011612] (-6,0.5) -- (-5,0.5) -- (-5,1.5) -- (-6,1.5) -- cycle;
\fill[line width=0.8pt,color=black,fill=black,fill opacity=0.10000000149011612] (-6,2.5) -- (-5,2.5) -- (-5,3.5) -- (-6,3.5) -- cycle;
\fill[line width=0.8pt,color=black,fill=black,fill opacity=0.10000000149011612] (-6,4.5) -- (-5,4.5) -- (-5,5.5) -- (-6,5.5) -- cycle;
\draw [line width=0.8pt] (-6.5,8)-- (7.5,8);
\draw [line width=0.8pt] (7.5,8)-- (7.5,-6);
\draw [line width=0.8pt] (7.5,-6)-- (-6.5,-6);
\draw [line width=0.8pt] (-6.5,-6)-- (-6.5,8);
\draw [line width=0.8pt] (-4.5,6)-- (5.5,6);
\draw [line width=0.8pt] (5.5,6)-- (5.5,-4);
\draw [line width=0.8pt] (5.5,-4)-- (-4.5,-4);
\draw [line width=0.8pt] (-4.5,-4)-- (-4.5,6);
\draw [line width=0.8pt] (-6.5,6)-- (-4.5,6);
\draw [line width=0.8pt] (-4.5,8)-- (-4.5,6);
\draw [line width=0.8pt] (-6.5,4)-- (-4.5,4);
\draw [line width=0.8pt] (-6.5,2)-- (-4.5,2);
\draw [line width=0.8pt] (-6.5,0)-- (-4.5,0);
\draw [line width=0.8pt] (-6.5,-2)-- (-4.5,-2);
\draw [line width=0.8pt] (-6.5,-4)-- (-4.5,-4);
\draw [line width=0.8pt] (-4.5,-4)-- (-4.5,-6);
\draw [line width=0.8pt] (-2.5,-4)-- (-2.5,-6);
\draw [line width=0.8pt] (-0.5,-4)-- (-0.5,-6);
\draw [line width=0.8pt] (1.5,-4)-- (1.5,-6);
\draw [line width=0.8pt] (3.5,-4)-- (3.5,-6);
\draw [line width=0.8pt] (5.5,-4)-- (5.5,-6);
\draw [line width=0.8pt] (5.5,-4)-- (7.5,-4);
\draw [line width=0.8pt] (5.5,-2)-- (7.5,-2);
\draw [line width=0.8pt] (5.5,0)-- (7.5,0);
\draw [line width=0.8pt] (5.5,2)-- (7.5,2);
\draw [line width=0.8pt] (5.5,4)-- (7.5,4);
\draw [line width=0.8pt] (5.5,6)-- (7.5,6);
\draw [line width=0.8pt] (5.5,6)-- (5.5,8);
\draw [line width=0.8pt] (3.5,6)-- (3.5,8);
\draw [line width=0.8pt] (1.5,6)-- (1.5,8);
\draw [line width=0.8pt] (-0.5,6)-- (-0.5,8);
\draw [line width=0.8pt] (-2.5,6)-- (-2.5,8);
\draw [line width=0.8pt,color=black] (-6,7.5)-- (-6,6.5);
\draw [line width=0.8pt,color=black] (-6,6.5)-- (-5,6.5);
\draw [line width=0.8pt,color=black] (-5,6.5)-- (-5,7.5);
\draw [line width=0.8pt,color=black] (-5,7.5)-- (-6,7.5);
\draw [line width=0.8pt,color=black] (-4,6.5)-- (-3,6.5);
\draw [line width=0.8pt,color=black] (-3,6.5)-- (-3,7.5);
\draw [line width=0.8pt,color=black] (-3,7.5)-- (-4,7.5);
\draw [line width=0.8pt,color=black] (-4,7.5)-- (-4,6.5);
\draw [line width=0.8pt,color=black] (-2,6.5)-- (-1,6.5);
\draw [line width=0.8pt,color=black] (-1,6.5)-- (-1,7.5);
\draw [line width=0.8pt,color=black] (-1,7.5)-- (-2,7.5);
\draw [line width=0.8pt,color=black] (-2,7.5)-- (-2,6.5);
\draw [line width=0.8pt,color=black] (0,6.5)-- (1,6.5);
\draw [line width=0.8pt,color=black] (1,6.5)-- (1,7.5);
\draw [line width=0.8pt,color=black] (1,7.5)-- (0,7.5);
\draw [line width=0.8pt,color=black] (0,7.5)-- (0,6.5);
\draw [line width=0.8pt,color=black] (2,6.5)-- (3,6.5);
\draw [line width=0.8pt,color=black] (3,6.5)-- (3,7.5);
\draw [line width=0.8pt,color=black] (3,7.5)-- (2,7.5);
\draw [line width=0.8pt,color=black] (2,7.5)-- (2,6.5);
\draw [line width=0.8pt,color=black] (4,6.5)-- (5,6.5);
\draw [line width=0.8pt,color=black] (5,6.5)-- (5,7.5);
\draw [line width=0.8pt,color=black] (5,7.5)-- (4,7.5);
\draw [line width=0.8pt,color=black] (4,7.5)-- (4,6.5);
\draw [line width=0.8pt,color=black] (6,6.5)-- (7,6.5);
\draw [line width=0.8pt,color=black] (7,6.5)-- (7,7.5);
\draw [line width=0.8pt,color=black] (7,7.5)-- (6,7.5);
\draw [line width=0.8pt,color=black] (6,7.5)-- (6,6.5);
\draw [line width=0.8pt,color=black] (6,4.5)-- (7,4.5);
\draw [line width=0.8pt,color=black] (7,4.5)-- (7,5.5);
\draw [line width=0.8pt,color=black] (7,5.5)-- (6,5.5);
\draw [line width=0.8pt,color=black] (6,5.5)-- (6,4.5);
\draw [line width=0.8pt,color=black] (6,2.5)-- (7,2.5);
\draw [line width=0.8pt,color=black] (7,2.5)-- (7,3.5);
\draw [line width=0.8pt,color=black] (7,3.5)-- (6,3.5);
\draw [line width=0.8pt,color=black] (6,3.5)-- (6,2.5);
\draw [line width=0.8pt,color=black] (6,0.5)-- (7,0.5);
\draw [line width=0.8pt,color=black] (7,0.5)-- (7,1.5);
\draw [line width=0.8pt,color=black] (7,1.5)-- (6,1.5);
\draw [line width=0.8pt,color=black] (6,1.5)-- (6,0.5);
\draw [line width=0.8pt,color=black] (6,-1.5)-- (7,-1.5);
\draw [line width=0.8pt,color=black] (7,-1.5)-- (7,-0.5);
\draw [line width=0.8pt,color=black] (7,-0.5)-- (6,-0.5);
\draw [line width=0.8pt,color=black] (6,-0.5)-- (6,-1.5);
\draw [line width=0.8pt,color=black] (6,-3.5)-- (7,-3.5);
\draw [line width=0.8pt,color=black] (7,-3.5)-- (7,-2.5);
\draw [line width=0.8pt,color=black] (7,-2.5)-- (6,-2.5);
\draw [line width=0.8pt,color=black] (6,-2.5)-- (6,-3.5);
\draw [line width=0.8pt,color=black] (6,-5.5)-- (7,-5.5);
\draw [line width=0.8pt,color=black] (7,-5.5)-- (7,-4.5);
\draw [line width=0.8pt,color=black] (7,-4.5)-- (6,-4.5);
\draw [line width=0.8pt,color=black] (6,-4.5)-- (6,-5.5);
\draw [line width=0.8pt,color=black] (4,-5.5)-- (5,-5.5);
\draw [line width=0.8pt,color=black] (5,-5.5)-- (5,-4.5);
\draw [line width=0.8pt,color=black] (5,-4.5)-- (4,-4.5);
\draw [line width=0.8pt,color=black] (4,-4.5)-- (4,-5.5);
\draw [line width=0.8pt,color=black] (2,-5.5)-- (3,-5.5);
\draw [line width=0.8pt,color=black] (3,-5.5)-- (3,-4.5);
\draw [line width=0.8pt,color=black] (3,-4.5)-- (2,-4.5);
\draw [line width=0.8pt,color=black] (2,-4.5)-- (2,-5.5);
\draw [line width=0.8pt,color=black] (0,-5.5)-- (1,-5.5);
\draw [line width=0.8pt,color=black] (1,-5.5)-- (1,-4.5);
\draw [line width=0.8pt,color=black] (1,-4.5)-- (0,-4.5);
\draw [line width=0.8pt,color=black] (0,-4.5)-- (0,-5.5);
\draw [line width=0.8pt,color=black] (-2,-5.5)-- (-1,-5.5);
\draw [line width=0.8pt,color=black] (-1,-5.5)-- (-1,-4.5);
\draw [line width=0.8pt,color=black] (-1,-4.5)-- (-2,-4.5);
\draw [line width=0.8pt,color=black] (-2,-4.5)-- (-2,-5.5);
\draw [line width=0.8pt,color=black] (-4,-5.5)-- (-3,-5.5);
\draw [line width=0.8pt,color=black] (-3,-5.5)-- (-3,-4.5);
\draw [line width=0.8pt,color=black] (-3,-4.5)-- (-4,-4.5);
\draw [line width=0.8pt,color=black] (-4,-4.5)-- (-4,-5.5);
\draw [line width=0.8pt,color=black] (-6,-5.5)-- (-5,-5.5);
\draw [line width=0.8pt,color=black] (-5,-5.5)-- (-5,-4.5);
\draw [line width=0.8pt,color=black] (-5,-4.5)-- (-6,-4.5);
\draw [line width=0.8pt,color=black] (-6,-4.5)-- (-6,-5.5);
\draw [line width=0.8pt,color=black] (-6,-3.5)-- (-5,-3.5);
\draw [line width=0.8pt,color=black] (-5,-3.5)-- (-5,-2.5);
\draw [line width=0.8pt,color=black] (-5,-2.5)-- (-6,-2.5);
\draw [line width=0.8pt,color=black] (-6,-2.5)-- (-6,-3.5);
\draw [line width=0.8pt,color=black] (-6,-1.5)-- (-5,-1.5);
\draw [line width=0.8pt,color=black] (-5,-1.5)-- (-5,-0.5);
\draw [line width=0.8pt,color=black] (-5,-0.5)-- (-6,-0.5);
\draw [line width=0.8pt,color=black] (-6,-0.5)-- (-6,-1.5);
\draw [line width=0.8pt,color=black] (-6,0.5)-- (-5,0.5);
\draw [line width=0.8pt,color=black] (-5,0.5)-- (-5,1.5);
\draw [line width=0.8pt,color=black] (-5,1.5)-- (-6,1.5);
\draw [line width=0.8pt,color=black] (-6,1.5)-- (-6,0.5);
\draw [line width=0.8pt,color=black] (-6,2.5)-- (-5,2.5);
\draw [line width=0.8pt,color=black] (-5,2.5)-- (-5,3.5);
\draw [line width=0.8pt,color=black] (-5,3.5)-- (-6,3.5);
\draw [line width=0.8pt,color=black] (-6,3.5)-- (-6,2.5);
\draw [line width=0.8pt,color=black] (-6,4.5)-- (-5,4.5);
\draw [line width=0.8pt,color=black] (-5,4.5)-- (-5,5.5);
\draw [line width=0.8pt,color=black] (-5,5.5)-- (-6,5.5);
\draw [line width=0.8pt,color=black] (-6,5.5)-- (-6,4.5);
\draw [decorate,decoration={brace,amplitude=6pt},xshift=0pt,yshift=0pt]
(1.5,6) -- (-0.5,6) node {};
\begin{scriptsize}
\draw [fill=black] (0.5,1) circle (1.5pt);
\draw[color=black] (0.75,1.25) node {\large{$x$}};
\draw[color=black] (0.5,5.4) node {\large{$2^{i+1}$}};
\end{scriptsize}
\end{tikzpicture}
\caption{The case $d = 2$ (hence $K = 7$): $\mathcal C(x,i)$ is the set of all grey squares (they all belong to $\mathcal F_i$).}
\label{fig 5}
\end{figure}
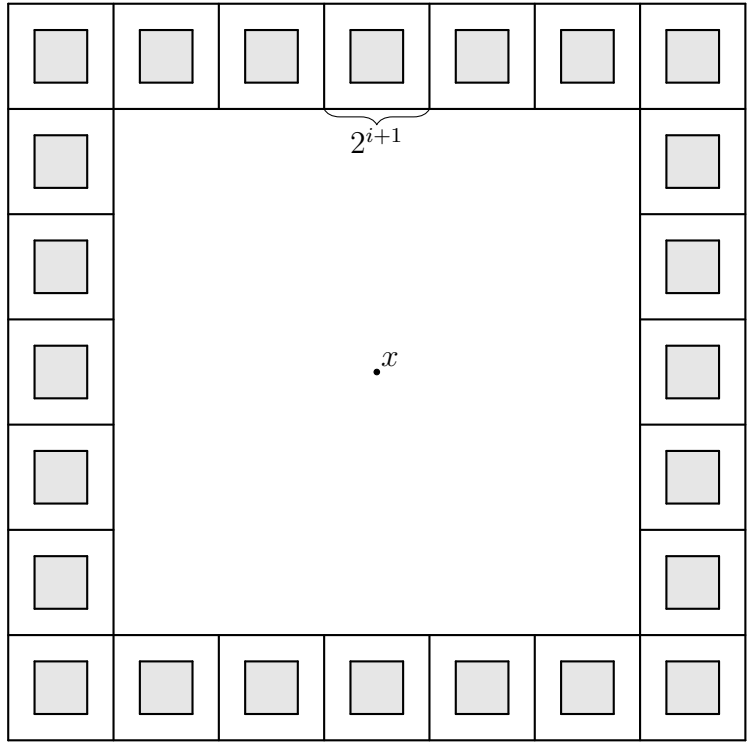

\begin{observation}\label{ob basin}
For every cube $q\in \mathcal F_i$, $|\mathcal R_i(q)|\le \Psi 2^{(i+1)d}$.
\end{observation}
\begin{proof}
Let $y$ be the center of the cube $q$. Then, every point $x\in \mathbb T^d_n$, for which $q$ is contained in some boundary cube of $N(x, K\cdot 2^{i+1})$, is itself contained in $W(y, K\cdot 2^{i+1})\setminus W(y, (K-2)\cdot 2^{i+1})$, whose volume is exactly $\Psi 2^{(i+1)d}$.
\end{proof}

For every $m\in [k-1]_0$ and every vertex $v\in [2^{md}, 2^{(m+1)d})$, define
\begin{equation}\label{eq def l(v)}
    \ell(v) = \min\{K\cdot 2^{s+1}:\hspace{0.2em} \forall q\in \mathcal C(v,s), q\cap [2^{md}-1]_0 \neq \emptyset \text{ and  }s\ge m\}.
\end{equation}
Intuitively, we are looking for the smallest value of $\ell=\ell(v)$ (depending on the label of $v$) that ensures that all boundary cubes in $N(v,\ell)$ contain at least one vertex. 
In particular, the number of children of $v$ outside $W(v, 1/2)$ is at most $\Psi \alpha_{\ell(v)}$. The next observation studies the proportion of cubes of $\mathcal F_i$ that do not contain a vertex among $[2^{(k-j)d}-1]_0$ for $j\le i$ (recall that $n=2^{kd}$). Define
\begin{equation*}
    \mu = \mu(i,j) := 2^{(k-i)d} \left(1 - 2^{-(k-i)d}\right)^{2^{(k-j)d}}.
\end{equation*}
Note that when $k-i\gg 1$, we have that $\mu =  (1+o(1)) 2^{(k-i)d}\exp(-2^{(i-j)d})$. The meaning of $\mu$ will become clear in the next observation. 
\begin{observation}\label{ob 2.5}
Fix $i = i(k)$ such that $k-i\gg 1$ and $j\in [i]_0$.
The expected number of cubes in $\mathcal F_i$ not containing a vertex among $[2^{(k-j)d}-1]_0$ is $\mu$.
Moreover, the number of such cubes is at most:
\begin{itemize}
    \item $k^3 \mu$ with probability at least $1 - k^{-3}$,
    \item $4\mu$ with probability at least $1 - 4 (3/4)^{4\mu}$.
\end{itemize}
\end{observation}
\begin{proof}
Since the number of cubes in $\mathcal F_i$ is $2^{(k-i)d}$, the statement about the expected number of cubes is clear. The first bullet then follows from Markov's inequality. 

We concentrate on the second bullet. For every large enough $k$, the probability that $s\ge 4\mu$ cubes do not contain a vertex is bounded from above by 
\begin{align*}
\binom{2^{(k-i)d}}{s} \left(1-\dfrac{s}{2^{(k-i)d}}\right)^{2^{(k-j)d}}
\le \left(\dfrac{2^{(k-i)d}e}{s} \exp(-2^{(i-j)d})\right)^{s}\le \left(\dfrac{3}{4}\right)^{s}.
\end{align*}

A union bound over all $s\in [\lceil 4\mu\rceil, 2^{(k-i)d}]$ proves the claim since 
\begin{equation*}
    \sum_{s\ge \lceil 4\mu\rceil} \left(\dfrac{3}{4}\right)^s = 4\left(\dfrac{3}{4}\right)^{\lceil 4\mu\rceil}\le 4\left(\dfrac{3}{4}\right)^{4\mu},
\end{equation*}
and the lemma is proved.
\end{proof}

For every $j\le i$, define $\mathcal A_{i,j}$ to be the event that at most $M_{i,j} := 4\mu \mathds{1}_{\mu\ge k^3}+k^3 \mu \mathds{1}_{\mu<k^3}$ cubes of $\mathcal F_i$ do not contain a vertex among $[2^{(k-j)d}-1]_0$. Also, fix an integer $t = t(k)$ satisfying $k-t\gg 1$ and set 
$$\mathcal A = \bigcap_{i\le t}\hspace{0.5em}\bigcap_{j\le i} \mathcal A_{i,j}.$$

\begin{corollary}\label{cor technical}
For every $j\le t - 1$ and $i\in [j,t]$, conditionally on $\mathcal A_{i, j}$, the probability that a vertex $v\in [2^{(k-j)d}, 2^{(k-j+1)d})$ has $\ell(v)\le K\cdot 2^{i+1}$ is at least 
$$1 - \Psi\left( 5 \exp(-2^{(i-j)d})\mathds{1}_{\mu\ge k^3} + \dfrac{k^6\mathds{1}_{1/k^3 \le \mu < k^3}}{2^{(k-i)d}}\right).$$
\end{corollary}
\begin{proof}
For every $i\in [j,t]$ and a vertex $v\in [2^{(k-j)d}, 2^{(k-j+1)d})$, if $\mathcal C(v, i)$ contains $\Psi$ squares of $\mathcal F_i$ where each square contains at least one vertex among $[2^{(k-j)d}-1]_0$, then $\ell(v)\le K\cdot 2^{i+1}$. 
Conditionally on $\mathcal A_{i,j}$, by Observation~\ref{ob basin} and Observation~\ref{ob 2.5} we have that any given cube of $\mathcal F_i$ does not contain a vertex among $[2^{(k-j)d}-1]_0$ with probability at most 
\begin{equation*}
\begin{cases}
    & \dfrac{4\mu}{2^{(k-i)d}} \le 5 \exp(-2^{(i-j)d}), \text{ if } \mu\ge k^3 \text{ and } n \text{ is sufficiently large},\\
    & \dfrac{k^3 \mu}{2^{(k-i)d}}, \text{ if } \mu\in [1/k^3, k^3],\\
    & 0, \text{ otherwise.}
\end{cases}
\end{equation*}

\noindent
Thus, conditionally on $\mathcal A_{i,j}$, for every vertex $v\in [2^{(k-j)d}, 2^{(k-j+1)d})$, $\ell(v) > K\cdot 2^{i+1}$ with probability at most 
\begin{equation*}
    \dfrac{\Psi( 4\mu\mathds{1}_{\mu\ge k^3} + k^3 \mu \mathds{1}_{1/k^3 \le \mu < k^3})}{2^{(k-i)d}} \le \Psi\left(5\exp(-2^{(i-j)d})\mathds{1}_{\mu\ge k^3} + \dfrac{k^6\mathds{1}_{1/k^3 \le \mu < k^3}}{2^{(k-i)d}}\right),
\end{equation*}
and the proof of the corollary is finished.
\end{proof}

To conclude the proof of the upper bound of Theorem~\ref{thm main 1}, it only remains to show that a.a.s.\ the number of vertices in $[2^{(k-j)d}, 2^{(k-j+1)d})$ satisfying $\ell(v) = K\cdot 2^{i+1}$, where $i\in [j, t]$, is concentrated around its expected value. 

We start with a claim that follows directly from Chernoff's inequality~\eqref{chern}.

\begin{claim}\label{claim 2.9}
Fix a sequence $(p_i)_{i\ge 0}$ of non-negative real numbers of summing up to $1$. 
Given a random variable $X$ taking values in $\mathbb{Z}^+$ a.s. such that  $\mathbb P(X=i) = p_i$ for all integer $i \ge 0$, 
and $m := 2^{(k-j+1)d} - 2^{(k-j)d}$ independent copies $X_1, \dots, X_m$ of $X$, let $Y_r$ be the number of $(X_i)_{i\in [m]}$ that are at least $r$. Then, for every $r\ge 1$ and $\delta \ge 0$,
\begin{equation*}
    \mathbb P(|Y_r - \mathbb E Y_r|\ge \delta \mathbb E Y_r)\le 2\exp\left(-\frac{\delta^2 \mathbb E Y_r}{2(1+\delta)}\right) = 2\exp\left(-\frac{\delta^2 m \sum_{i\ge r} p_i}{2(1+\delta)}\right).
\end{equation*}
\end{claim}

\begin{proof}[Proof of the upper bound of Theorem~\ref{thm main 1}]
We are now ready for the proof of the upper bound of Theorem~\ref{thm main 1}. 
First, fix $t = k - 20\log_2 k$ and $j\le i\le t$ for which $\mu\ge k^3$. 
Then, every vertex $v\in [2^{(k-j)d}, 2^{(k-j+1)d})$ has the same probability to satisfy $\ell(v) > K\cdot 2^{i+1}$. 
Moreover, conditionally on $\mathcal A_{i,j}$, Observation~\ref{ob basin} implies that this probability is bounded from above by $4\mu\cdot \Psi 2^{(i+1)d}\cdot 2^{-kd}$. 
Then, by Claim~\ref{claim 2.9}, applied conditionally on $\mathcal A_{i,j}$ for $(\mathds{1}_{\ell(v) > K\cdot 2^{i+1}})_{v\in [2^{(k-j)d}, 2^{(k-j+1)d})}$ and $\delta=1.5$, 
the number of vertices $v\in [2^{(k-j)d}, 2^{(k-j+1)d})$ that satisfy $\ell(v) > K\cdot 2^{i+1}$ is at most 
\begin{equation*}
1.5 \cdot (2^{(k-j+1)d} - 2^{(k-j)d})\cdot 4\mu\cdot \Psi 2^{(i+1)d}\cdot 2^{-kd} < 2^{2d+3} \Psi\cdot 2^{(k-j)d} \exp(-2^{(i-j)d})    
\end{equation*}
with probability at least $1-\exp(-\Omega_d(\mu 2^{(i-j)d})) = 1 - o(1/k^3)$, where the latter inequality follows from our assumptions that $i\ge j$ and $\mu \ge k^3$. Note that we are allowed to apply Claim~\ref{claim 2.9} since conditionally on the positions of $[2^{(k-j)d}-1]_0$ the events $\{\ell(v) > K\cdot 2^{i+1}\}_{v\in [2^{(k-j)d}, 2^{(k-j+1)d})}$ are independent.

Now, suppose that $j \le i \le t$ and $\mu < k^3$. Then, conditionally on $\mathcal A_{i,j}$, at most $k^7$ vertices $v\in [2^{(k-j)d}, 2^{(k-j+1)d})$ have $\ell(v) > K\cdot 2^{i+1}$ with probability at least $1 - O(1/k^4) = 1 - o(1/k^3)$ (this last bound follows by Markov's inequality). 

Denote by $\mathcal D_{i,j}$ the event that the number of vertices $v\in [2^{(k-j)d}, 2^{(k-j+1)d})$ with $\ell(v) > K\cdot 2^{i+1}$ is at most
\begin{align*}
    2^{2d+3} \Psi\cdot 2^{(k-j)d} \exp(-2^{(i-j)d}) + k^7,
\end{align*}
and let 
$$\mathcal D = \bigcap_{i\le t}\hspace{0.5em}\bigcap_{j\le i} \mathcal D_{i,j}.$$ 
Then, we get that for every pair $(i,j)$, $\mathbb P(\overline{\mathcal D_{i,j}}\mid \mathcal A_{i,j}) = o(1/k^3)$. Thus, 
\begin{align*}\label{eq:UB}
    \mathbb P(\overline{\mathcal D})\le\hspace{0.3em} 
    &\sum_{i\le t}\hspace{0.5em}\sum_{j\le i}  \mathbb P(\overline{\mathcal D_{i,j}})
    \le \sum_{i\le t}\hspace{0.5em}\sum_{j\le i} \big(\mathbb P(\overline{\mathcal D_{i,j}}\cap \mathcal A_{i,j}) + \mathbb P(\overline{\mathcal A_{i,j}})\big)\\
    \le\hspace{0.3em}
    &\sum_{i\le t}\hspace{0.5em}\sum_{j\le i} \left(\mathbb P(\overline{\mathcal D_{i,j}}\mid \mathcal A_{i,j}) + o(1/k^3)\right) = o(1/k) = o(1).
\end{align*}

Thus, $\mathcal D$ is an a.a.s.\ event. 
Finally, note that, for every $i\le t$, under the event $\cap_{j\le i} \mathcal D_{i,j}$, the number of vertices with $\ell(v) > K\cdot 2^{i+1}$ is at most 
\begin{align*}
    2^{(k-i)d} + \sum_{j\le i} 2^{2d+3} \Psi\cdot  2^{(k-j)d} \exp(-2^{(i-j)d}) + k^7 =\hspace{0.3em} O(n 2^{-id} + (\log n)^{10}).
\end{align*}
Recalling that for $\ell(v)=K \cdot 2^{i+1}$, the number of children in outer cubes is at most $\Psi \alpha(K \cdot2^{i+1},d)$, we deduce that there exists $C_1=C_1(d) > 0$ sufficiently small so that, 
for all $i\le t = \tfrac{1}{d}\log_2 n - 20\log_2(\tfrac{1}{d}\log_2 n)$, the number of vertices $v\in V(G_n)$ with strictly more than $\Psi \alpha(K\cdot 2^{i+1}, d) +1= \Theta(i)$ neighbors outside $W(v,1/2)$ is at most $n \exp(- 2C_1 i)$. 
Also, recall that for all $i \le s_{max}$, by Observation~\ref{ob close neighbors}, the number of vertices $v\in V(G_n)$ with at least $2^d i$ vertices of $G_n$ contained in $W(v,1/2)$ is
\begin{equation*}
O\left(\frac{n}{(i-1)!}\right)=O\left(n\exp(-(1+o_{i\to \infty}(1))i\log i)\right).    
\end{equation*}
Moreover, for all $i > s_{\max}$, the number of vertices $v\in V(G_n)$ with at least $2^d i$ vertices of $G_n$ contained in $W(v,1/2)$ must also be $O(n/(s_{max}-1)!) = O(n^{2/3+o(1)})$. 
Thus, by choosing $c=c(d) > 0$ sufficiently small, 
we deduce that, for all $i \le c \log n$, the number of vertices $v\in V(G_n)$ with at least $2^d i$ vertices of $G_n$ contained in $W(v,1/2)$ is at most $n \exp(-2C_1 i)$. 
Finally, summing over neighbors inside $W(v,1/2)$ and those outside $W(v,1/2)$ proves the upper bound of Theorem~\ref{thm main 1}.
\end{proof}

\subsection{\texorpdfstring{Lower bound on the number of vertices of degree at least $i$ and concentration around its expectation}{}}

As above, we assume for convenience that $n = 2^{kd}$. Recall that, by abuse of notation, in this section we identify the vertices in $V(G) = [n-1]_0$ with their geometric positions. 

For integers $i,j\in [k]_0$, $j\le i$, define $\mathcal B_{i,j}$ as the event that, after embedding the vertex $2^{(k-j)d}-1$, at least $2^{(k-i)d}/3^{d+3}$ of the Voronoi cells of vertices $v$ in $[2^{(k-i)d}-1]_0$ contain the ball $B(v, 2^j/6)$ in their interiors. Recall that $t = t(k)$ is an integer satisfying $k-t\gg 1$.

\begin{lemma}\label{lem B_i,j holds}
For every $i\le t$ and $j\in [i]_0$, the event $\mathcal B_{i,j}$ holds with probability $1 - \exp(-\Omega(2^{(k-i)d}))$.
\end{lemma}
\begin{proof}
We start with an easy geometric observation. 
Fix a Voronoi tessellation of $\mathbb T^d_n$, generated by the vertices $[2^{(k-j)d}-1]_0$. 
For a vertex $v\in [2^{(k-j)d}-1]_0$, if the ball $B(v,r)$ does not contain any vertex in $[2^{(k-j)d}-1]_0\setminus \{v\}$, then the Voronoi cell of the vertex $v$ in the presence of the vertices $[2^{(k-j)d}-1]_0$ contains the ball $B(v, r/2)$ in its interior. 
We now show that the number $X_i$ of vertices $v$ in $[2^{(k-i)d}-1]_0$ such that $B(v, 2^j/3)$ contains no vertex in $[2^{(k-j)d}-1]_0\setminus v$ is at least $3^{-d-3} 2^{(k-i)d}$ with probability $1 - \exp(-\Omega(2^{(k-i)d}))$.

Recall that $\mathcal F_j$ is a tessellation of $\mathbb T^d_n$ into cubes of side length $2^j$. Let $N_j\sim \mathrm{Po}(2^{(k-j)d+1})$. Then, by~\eqref{chern}, $2^{(k-j)d}\le N_j < 3\cdot 2^{(k-j)d}$ with probability $1 - \exp(-\Omega(\mathbb E N_j)) = 1 - \exp(-\Omega(2^{(k-j)d}))$. Thus, by embedding the first $N_j$ vertices of $G_n$ in $\mathbb T^d_n$, we obtain a Poisson Point Process with intensity $2^{-jd+1}$ that dominates the binomial process on $[2^{(k-j)d}-1]_0$ with probability $1 - \exp(-\Omega(2^{(k-j)d}))$. Denote by $\tilde{X}_i$ the number of squares $q\in \mathcal F_j$ satisfying the following three conditions:
\begin{itemize}
    \item $q$ contains a single vertex among $[N_j-1]_0$ (which holds with probability $2\exp(-2)$),
    \item this vertex is embedded at distance at least $2^j/3$ from the boundary of $q$ (which holds with probability $3^{-d}$ conditionally on the previous event),
    \item this vertex has label at most $N_j 2^{-(i-j)d}/3$ (which holds with probability $2^{-(i-j)d}/3$ conditionally on the first event; the second condition is independent of this event).
\end{itemize}
Note that under the event $2^{(k-j)d}\le N_j < 3\cdot 2^{(k-j)d}$ one has that $\tilde{X}_i\le X_i$: indeed, if a cube $q$ contains a unique vertex $v\le N_j 2^{-(i-j)d}/3 < 2^{(k-i)d}$ at distance at least $2^j/3$ from its boundary, the ball around $v$ with radius $2^j/3$ contains no other vertex in $[2^{(k-j)}-1]_0\subseteq [N_j-1]_0$. On the other hand, by Chernoff's inequality~\eqref{chern},
\begin{equation*}
    \mathbb P(\tilde{X}_i\le \mathbb E \tilde{X}_i/2)\le \exp(-\Omega(\mathbb E \tilde{X}_i)) = \exp(-\Omega(2^{(k-j)d}\cdot 2\exp(-2)\cdot 2^{-(i-j)d}/3^{d+1})) = \exp(-\Omega(2^{(k-i)d})).
\end{equation*}
Consequently, at least $\mathbb E\tilde{X}_i/2 = 2^{(k-j)d}\cdot \exp(-2)\cdot 2^{-(i-j)d}/3^{d+1}  > 2^{(k-i)d}/3^{d+3}$ vertices $v\in [2^{(k-i)d}-1]_0$ satisfy $W(v, 2^j/3)\cap [2^{(k-j)d}-1]_0 = \{v\}$ (so also $B_d(v, 2^j/3)\cap [2^{(k-i)d}-1]_0 = \{v\}$) with probability $1-\exp(-\Omega(2^{(k-i)d}))$, which completes the proof.
\end{proof}

\begin{lemma}\label{lem accumulating neighbors}
Fix $i\le t$ and $j\in [i]_0$. Then, there is a constant $\gamma > 0$ such that the number of vertices in $[2^{(k-i)d}-1]_0$ that have at least one neighbor in $[2^{(k-j)d}, 2^{(k-j+1)d})$ is at least $\gamma 2^{(k-i)d}/3^{d+3}$ with probability $1 - \exp(-\Omega(2^{(k-i)d}))$. 
\end{lemma}
\begin{proof}
In this proof, we work conditionally on $\mathcal B_{i,j}$. 
Let $V_{i,j}$ be the set of $2^{(k-i)d}/3^{d+3}$ vertices provided by the event $\mathcal B_{i,j}$. Note that, for any two vertices $v, v' \in V_{i,j}$, $B_d(v, 2^j/6) \cap B_d(v',2^j/6)=\emptyset$ (the interiors of the Voronoi cells of $v$ and $v'$ are disjoint and contain respectively $B_d(v, 2^j/6)$ and $B_d(v',2^j/6)$ by assumption). Fix $v\in V_{i,j}$. Then, the probability that at least one vertex of $[2^{(k-j)d}, 2^{(k-j+1)d})$ falls into $B_d(v, 2^j/6)$, and the first of these is embedded into $B(v, 2^j/12)$ (this event ensures that $v$ is a parent of this new vertex), is
\begin{equation*}
    \left(1 - \left(1 - \dfrac{|B(v,1)|}{2^{(k-j)d} 6^d}\right)^{2^{(k-j+1)d}-2^{(k-j)d}}\right) \cdot \frac{1}{2^d},
\end{equation*}
which is bounded from below by a positive constant $\gamma' = \gamma'(d)$. 
As before, since $B(v, 2^j/6) \cap B(v',2^j/6)=\emptyset$, by a Poisson approximation (as in the proof of Lemma~\ref{lem B_i,j holds}) and Chernoff's inequality~\eqref{chern}, the probability that there are at least $\gamma' |V_{i,j}|/2$ vertices, for which the above event holds, is $1 - \exp(-\Omega(2^{(k-i)d}))$. Thus, the event 
\begin{equation*}
    \mathcal E = \{\text{at least a } \gamma'/2\text{-proportion of the vertices in } V_{i,j} \text{ have at least one neighbor in } [2^{(k-j)d}, 2^{(k-j+1)d})\}
\end{equation*}
holds with probability $1 - \exp(-\Omega(2^{(k-i)d}))$ conditionally on $\mathcal B_{i,j}$. Thus, by Lemma~\ref{lem B_i,j holds},
\begin{equation*}
    \mathbb P(\mathcal E) = 1 - \mathbb P(\overline{\mathcal E})\ge 1 - \mathbb P(\overline{\mathcal E}\mid \mathcal B_{i,j}) - \mathbb P(\overline{\mathcal B_{i,j}}) = 1 - \exp(-\Omega(2^{(k-i)d})),
\end{equation*}
which finishes the proof of the lemma for $\gamma = \gamma'/2$.
\end{proof}

\begin{proof}[Proof of the lower bounds of Theorem~\ref{thm main 1} and Theorem~\ref{cor:maxdegree}]
Set $t = \lfloor k - 3\log_2 k\rfloor$.
The proof of the lower bound of Theorem~\ref{cor:maxdegree} directly follows from Theorem~\ref{thm main 1} for $i\le t$, so we concentrate on the latter. 
Fix $i\le t$. By Lemma~\ref{lem accumulating neighbors} and a union bound over all $j \le i$ we conclude that, for all $j\le i$, at least $\gamma 2^{(k-i)d}/3^{d+3} $ of the vertices in $[2^{(k-i)d}-1]_0$ are parents to vertices in $[2^{(k-j)d}, 2^{(k-j-1)d})$ with probability $1 - k\exp(-\Omega(2^{(k-i)d}))$. 
For convenience, if a vertex in $[2^{(k-i)d}-1]_0$ receives more than one child in $[2^{(k-j)d}, 2^{(k-j-1)d})$, we take into account only the first one (and thus count one child only).

We show that each of the $\gamma 2^{(k-i)d-1}/3^{d+3} = (\gamma/(2\cdot 3^{d+3})) 2^{-id} n$ vertices in $[2^{(k-i)d}-1]_0$ that receive the largest number of children throughout all $i+1$ steps of embedding $([2^{(k-j)d}, 2^{(k-j+1)d}))_{j=0}^{i}$ has at least $\gamma i/(4\cdot 3^{d+3})$ children. 
We argue by contradiction, so suppose not.
Then, none of the $(1-\gamma/(2\cdot 3^{d+3})) 2^{(k-i)d} < 2^{(k-i)d}$ vertices with the smallest number of children would have more than $\gamma i/(4\cdot 3^{d+3})$ children. 
Hence, the total number of children accumulated by all vertices in $[2^{(k-i)d}-1]_0$ is at most
$$
   \frac{(i+1) \gamma 2^{(k-i)d-1}}{3^{d+3}} +\dfrac{\gamma i 2^{(k-i)d}}{4\cdot 3^{d+3}}=\frac{3}{4} \frac{\gamma (i+1) 2^{(k-i)d}}{3^{d+3}}.
$$
This is less than the total number of children accumulated over all $i+1$ steps (which is at least $\gamma (i+1) 2^{(k-i)d}/3^{d+3}$), thus yielding a contradiction. The proof is completed by a union bound over all values of $i\le t$.
\end{proof}

Equipped with the lower bound on the number of vertices with degree at least $i$, by using Theorem~\ref{thm concentration}, we obtain as a corollary that this number is well concentrated around its mean, thereby concluding the proof of Theorem~\ref{thm main 1}.
\begin{corollary}\label{cor leaves}
Fix $i\in \mathbb N$ and denote $Y_i = |\{v\in V(G_n): \deg(v)\ge i\}|$. Then, for every $\varepsilon \ge 0$,
\begin{equation*}
\mathbb P(|Y_i - \mathbb E Y_i|\ge \varepsilon \mathbb EY_i)\le \exp(-\Omega(2^{(k-i)d/3})),
\end{equation*}
where the implicit constant in the exponent depends on both $d$ and $\varepsilon$.
\end{corollary}
\begin{proof}
First, note that $Y_i$ changes by at most two after replacement of a single edge in any labeled rooted tree, so it defines a 2-Lipschitz function on $\mathcal{LT}_n$. Since, by Theorem~\ref{thm main 1}, $\mathbb E Y_i = \Omega(2^{(k-i)d})$, applying Theorem~\ref{thm concentration} with $\phi(n) = 2^{(k-i)d/3}$ and $t = \varepsilon \mathbb EY_i$ finishes the proof of the corollary.
\end{proof}

\section{Typical distance, height and diameter}\label{sec distance}
In this section, we provide several results concerning the metric structure of $G_n$ and, in particular, give a proof of Theorem~\ref{thm diam}. 
By abuse of notation, when clear from the context, we again identify a vertex with its position.

\subsection{Bounding the Euclidean distance between the endvertices of decreasing paths}

The aim of this subsection is to provide an answer to the following question: 
starting from vertex $\ell > m$, what can be said about the Euclidean distance between $\ell$ and its nearest ancestor in $[m]_0$ with respect to the graph distance of $G_n$? 
This is a key step in the analysis of the diameter and the height of the $d$-NN tree, as well as the \emph{stretch factor} of a uniformly chosen pair of vertices, 
that is, the ratio of the sum of the Euclidean lengths of the edges on the path (in $G_n$) between the two vertices and the Euclidean distance between their positions.

\begin{theorem}\label{thm 4.1}
Fix $d\ge 1$ and $m = m(n) = n^{\Omega(1)}$. Then, the event $\mathcal A$ that the Euclidean distance between every $\ell > m$ and the unique integer $m' = m'(\ell)\in [m]_0$ satisfying $d_{G_n}(\ell,m') = d_{G_n}(\ell, [m]_0)$ is at most $2\sqrt{d}\left(\frac{(\log m)^{d+1} n}{m}\right)^{1/d}$ holds with probability at least $1-3/m$.
\end{theorem}

\begin{corollary}\label{cor PoNN}
Fix $d\ge 1$ and any sequence $(x_n)_{n\ge 1}$ of real numbers in $[0,1]$ with $nx_n = n^{\Omega(1)}$. Then, with probability at least $1 - 4/(nx_n)$, the Euclidean distance between any vertex in $\PoNN$ and its closest ancestor (in terms of graph distance) with arrival time in $[0,x_n]$ is at most $3\sqrt{d}\left(\frac{(\log(nx_n))^{d+1} n}{nx_n}\right)^{1/d}$. 
\end{corollary}
\begin{proof}
Since $|\PoNN|\sim \mathrm{Po}(n)$ and $|\{v\in V(\PoNN): X_v\le x_n\}|\sim \mathrm{Po}(nx_n)$, by~\eqref{chern} one has that a.a.s.\ $|\PoNN|\in [n-n^{2/3},n+n^{2/3}]$ and $|\{v\in V(\PoNN): X_v\le x_n\}|\in [nx_n-(nx_n)^{2/3}, nx_n+(nx_n)^{2/3}]$. 
Conditionally on these two events, the corollary follows from Theorem~\ref{thm 4.1} with $m = \lfloor nx_n-(nx_n)^{2/3}\rfloor \ge 6nx_n/7$.
\end{proof}

The rough idea behind the proof of Theorem~\ref{thm 4.1} is to decompose the path from $l$ to $[m]_0$ into shorter paths whose endvertices have labels within a factor of two, and to bound their lengths from above with (sufficiently) high probability. We defer the formal proof to the end of the section.

Fix $k_0 = \lceil (\log n)^2\rceil$.
For every $k\in [k_0, n-1]$, let the event $\mathcal A_k$ be that, after the vertex with label $k-1$ has been embedded in $\mathbb T^d_n$, 
the Euclidean ball with radius $\frac{\sqrt{d}}{5} \left(\frac{3 n \log k}{k}\right)^{1/d}$ around any of the $k$ vertices contains another vertex as well (note that $k\ge k_0$ ensures that $\frac{\sqrt{d}}{5} \left(\frac{3 n \log k}{k}\right)^{1/d} < \frac{n^{1/d}}{2}$ so that the balls in $\mathbb T^d_n$ coincide with the Euclidean balls).
Roughly speaking, one may think of the event $\mathcal A_k$ as having $k$ vertices spread relatively regularly. Also, for every $k\in [k_0-1, n-2]$, define the event $\mathcal B_k = \cap_{i=k+1}^{n-1} \mathcal A_i$. Roughly speaking, the event $\mathcal B_k$ says that every vertex in $[k, n-1]$ has its parent ``close'' to it. 
We first show that, for all $k\ge k_0-1$, a.a.s.\ the event $\mathcal B_k$ holds.

\begin{lemma}\label{lem B_k}
For all $k\ge k_0$, the event $\mathcal B_k$ holds with probability at least $1-1/k$.
\end{lemma}
\begin{proof}
Fix $k\ge k_0$. Then, for every point $x\in \mathbb T^d_n$ and every $r\le n^{1/d}/2$,
\begin{equation}\label{eq:ballvolume}
    |B(x,r)| = r^d |B(x,1)| = r^d \dfrac{\pi^{d/2}}{\Gamma(d/2+1)}\le r^d\dfrac{\pi^{d/2}}{(d/2e)^{d/2}}\le \left(\dfrac{5 r}{\sqrt{d}}\right)^d.
\end{equation}
Thus, 
\begin{equation*}
    \mathbb P(\overline{\mathcal B_k}) \le \sum_{i=k+1}^{n-1}  \mathbb P(\overline{\mathcal A_i})\le \sum_{i=k+1}^{n-1} i \left(1-\dfrac{3 \log i}{i}\right)^{i-1} \le \sum_{i=k+1}^{n-1} \dfrac{1}{i^2}\le \sum_{i=k+1}^{+\infty} \left(\dfrac{1}{i-1} - \dfrac{1}{i}\right) = \dfrac{1}{k},
\end{equation*}
where the second inequality comes from a union bound and the fact that, by~\eqref{eq:ballvolume}, the ball in $\mathbb T^d_n$ with radius $\frac{\sqrt{d}}{5} \left(\frac{3n\log i}{i}\right)^{1/d}\le n^{1/d}/2$ has area at most $\frac{3n\log i}{i}$, and the third inequality comes from the fact that, for every $i\ge 3$, $(1 - 3\log i/i)^{i-1} \le 1/i^3$. The lemma follows.
\end{proof}

Denote by $\pnt(v)$ the parent of the vertex $v$ (by convention, set $\pnt(0) = 0$), and by $\pnt^{k}(v)$ its $k$-th ancestor. 
In particular, $\pnt(v) = \pnt^1(v)$. 
We state the following observation used several times without explicit mention.

\begin{observation}
The distribution of $\pnt^k(v)$ conditionally on $(\pnt^i(v))_{i\in [k-1]}$ is uniform among $[\pnt^{k-1}(v)-1]_0$.
\end{observation}

As a result, one may naturally couple the construction of the random path from $v$ to the root with the following process constructed via a family of iid random variables $(U_k)_{k\ge 1}$ distributed uniformly in $[0,1)$. 
Set $v_0 = v$. Then, for all $k\ge 1$, if $v_{k-1}\neq 0$, connect it by an edge to the vertex $v_k := \lfloor v_{k-1} U_k\rfloor$ (one may easily check that this choice is uniform among $0,1,2,\dots, v_{k-1}-1$). In particular, for all $k\ge 0$, such that $\pnt^{k-1}(v)\neq 0$ we have that $\pnt^{k-1}(v) U_k - 1\le \pnt^{k}(v)\le \pnt^{k-1}(v) U_k$, or also 
\begin{equation}\label{eq path constr}
    v U_1\dots U_k - k\le \pnt^k(v)\le v U_1\dots U_k.
\end{equation}

Next, we use the above construction to show that the path in $G_n$ from the vertex $k$ to $\left[\lfloor k/2\rfloor\right]_0$ has a ``relatively small'' Euclidean length, that is, the sum of the lengths of its edges is small.

\begin{lemma}\label{lem 1 step}
For every $k\ge 2k_0$, conditionally on the event $\mathcal B_{\lfloor k/2\rfloor}$, the path from $k$ to $\left[\lfloor k/2\rfloor\right]_0$ in $G_n$ is of Euclidean length at most $\frac{3\sqrt{d}}{5} \left(\frac{n\log(k/2)}{k/2}\right)^{1/d}$ with probability at least $1 - \left(\frac{e\log 2}{3\log k}\right)^{3\log k}$.
\end{lemma}
\begin{proof}
Define the hitting time $\tau_k = \min\{t\in \mathbb N:\hspace{0.2em} U_1 U_2\dots U_t\le 1/2\}$, which can be rewritten as $\tau_k = \min \{t\in \mathbb N:\hspace{0.2em} X_1+\dots+X_t\ge \log 2\}$ for $X_i := -\log U_i$ for all $i\ge 1$. 
By Lemma~\ref{lem erlang}(i) and the fact that $(X_i)_{i\ge 1}$ are independent exponential random variables with mean 1, for every $k\ge 2$, we conclude that
\begin{equation*}
    \mathbb P(\tau_k\ge \lceil 3\log k\rceil) = \mathbb P\left(\sum_{i=1}^{\lceil 3\log k\rceil} X_i\le \log 2\right)\le \dfrac{2(\log 2)^{\lceil 3\log k\rceil}}{\lceil 3\log k\rceil !}\le \dfrac{(\log 2)^{\lceil 3\log k\rceil}}{(\lceil 3\log k\rceil/e)^{\lceil 3\log k\rceil}}\le \left(\dfrac{e \log 2}{3\log k}\right)^{3\log k}.
\end{equation*}
\noindent
Moreover, by~\eqref{eq path constr} we have $\{\tau_k\le \lfloor 3\log k\rfloor\}\subseteq \{\pnt^{\lfloor 3\log k\rfloor}(k)\le k/2\}$, and conditionally on $\mathcal B_{\lfloor k/2\rfloor}$, 
for all $t > k/2$, every vertex in $G_t$ is at Euclidean distance at most $\frac{\sqrt{d}}{5} \left(\frac{3 n \log(k/2)}{k/2}\right)^{1/d}$ to the remaining vertices in $G_t$. 
\end{proof}

\begin{lemma}\label{lem length}
Fix integers $\ell > m\ge 1$ with $m$ sufficiently large. 
Conditionally on the event $\mathcal B_{\lfloor m/2\rfloor}$, the (Euclidean) length of the path from $\ell$ to $[m]_0$ in $G_n$ is at most $2\sqrt{d}\left(\frac{(\log m)^{d+1} n}{m}\right)^{1/d}$ with probability at least $1-\exp(-\log m\log\log m)$.
\end{lemma}
\begin{proof}
We regroup the vertices in $[m+1,\ell]$ into groups with indices $[m+1, 2m], [2m+1,4m], \dots, [2^s m+1, l]$ where $s$ is the unique integer satisfying $2^s m + 1\le \ell \le 2^{s+1} m$. For any $k\in [s]_0$, by Lemma~\ref{lem 1 step} and a union bound over the vertices with indices in $[2^{k} m+1, 2^{k+1} m]$, we deduce that with probability at least $1 - 2^k m \left(\frac{e \log 2}{3\cdot \log(2^{k} m)}\right)^{3\log(2^{k} m)}$, for all $v\in [2^{k} m+1, 2^{k+1} m]$, the path from $v$ to $[2^k m]_0$ in $G_n$ has (Euclidean) length at most $\frac{\sqrt{d} \log(2^{k+1} m)}{5} \left(\frac{3 \log(2^{k} m)}{2^{k} m}\right)^{1/d} n^{1/d}$.
For all $\xi > 1$, define the functions
\begin{equation*}
    H_1(\xi) = \sum_{k=0}^{\infty} 2^k \xi\left(\dfrac{e\log 2}{3\cdot \log(2^k \xi)}\right)^{3 \log(2^k \xi)},
\end{equation*}
and
\begin{equation*}
    H_2(\xi) = \sum_{k=0}^{\infty} \dfrac{\sqrt{d} \log(2^{k+1} \xi)}{5} \left(\dfrac{3 \log(2^k \xi)}{2^k \xi}\right)^{1/d}.
\end{equation*}
\noindent
On the event $\mathcal B_{\lfloor m/2\rfloor}$, a union bound over all $s+1$ groups shows that, with probability at least $1-H_1(m)$, the (Euclidean) length of the path from $\ell$ to $[m]_0$ in $G_n$ is at most $H_2(m) n^{1/d}$.

By standard analysis, for all large enough $m$, we have
$$H_1(m) = O\left(m\left(\dfrac{e\log 2}{3\log m}\right)^{3\log m}\right) \le \exp\left(-\log m\log\log m \right),$$
and
\begin{align}
    H_2(m) &=\hspace{0.3em} \sum_{k=0}^{\infty} \dfrac{\sqrt{d} \log(2^{k+1} m)}{5} \left(\dfrac{3 \log(2^k m)}{2^k m}\right)^{1/d}\nonumber\\
    &\le\hspace{0.3em} \sum_{k=0}^{\infty} \dfrac{1.1 \sqrt{d} \log(2^k m)}{5} \left(\dfrac{3 \log(2^k m)}{2^k m}\right)^{1/d}\nonumber\\
    &=\hspace{0.3em} \sum_{k=0}^{\infty} \dfrac{1.1 \sqrt{d} \log m}{5} \left(\dfrac{3 \log m}{m}\right)^{1/d} \left(\dfrac{1}{2^k}\left(\dfrac{k \log 2}{\log m}+1\right)^{d+1}\right)^{1/d}\nonumber\\
    &\le\hspace{0.3em} \dfrac{1.2 \sqrt{d} \log m}{5} \left(\dfrac{3 \log m}{m}\right)^{1/d} \hspace{0.3em}\sum_{k=0}^{\infty} 2^{-k/d}.\nonumber
\end{align}
Hence, for all sufficiently large $m$,
\begin{equation}\label{upper bound eq}
H_2(m)\le \frac{3^{1/d}\cdot 1.2 \sqrt{d}\log m}{5(1-2^{-1/d})}\left(\frac{\log m}{m}\right)^{1/d}\le 2\sqrt{d}\left(\frac{(\log m)^{d+1}}{m}\right)^{1/d},
\end{equation}
and this concludes the proof of the lemma.
\end{proof}

\begin{proof}[Proof of Theorem~\ref{thm 4.1}]
Fix $m$ as in Theorem~\ref{thm 4.1}.  
On the one hand, by Lemma~\ref{lem length} and a union bound,
\begin{equation*}
    \mathbb P(\overline{\mathcal A}\mid \mathcal B_{\lfloor m/2\rfloor})\le \sum_{\ell=m+1}^{n-1} \exp(-\log m\log\log m)\le n \exp(-\log m\log\log m) = o(1/m).
\end{equation*}
On the other hand, by Lemma~\ref{lem B_k}, the event $\overline{\mathcal B_{\lfloor m/2\rfloor}}$ holds with probability at most $1/\lfloor m/2\rfloor \le 2.1/m$. This concludes the proof since $\mathbb P(\overline{\mathcal A})\le \mathbb P(\overline{\mathcal A}\mid\mathcal B_{\lfloor m/2\rfloor}) + \mathbb P(\overline{\mathcal B_{\lfloor m/2\rfloor}}) \le 3/m$.
\end{proof}

We finish this subsection with an unsurprising but nevertheless important property of the embedded tree $G_n$. 
Recall that, for two vertices $u$ and $v$, the stretch factor between $u$ and $v$ is the ratio of the Euclidean length of the path in $G_n$ between $u$ and $v$ and the Euclidean distance between $u$ and $v$. 
Clearly, the stretch factor is at least $1$; we show as a consequence of Lemma~\ref{lem length} that for two uniformly chosen vertices it is indeed of order 1 a.a.s.

\begin{corollary}\label{cor stretch}
For every $\varepsilon\in (0, d^{-d/2})$, there exists a constant $c=c(\varepsilon, d) > 0$ such that, with probability at least $1-\varepsilon$, the stretch factor of two uniformly chosen vertices $u$ and $v$ is at most $c$. 
\end{corollary}
\begin{proof}
Fix $\varepsilon\in (0, d^{-d/2})$ and a sufficiently large integer $M = M(\varepsilon)\ge 1$ so that, first, 
Lemma~\ref{lem length} holds for all $m\ge M$ and $\ell\ge m+1$, and second, $\exp(-\log M\log\log M)\le \varepsilon/2$. 
Set also $c_1 = c_1(\varepsilon, d):= 2\cdot 2\sqrt{d}\left(\frac{(\log M)^{d+1}}{M}\right)^{1/d} + M \sqrt{d}$. 
Then, by Lemma~\ref{lem length}, the Euclidean length of the path in $G_n$ from any vertex to $[M]_0$ is at most $2\sqrt{d}\left(\frac{(\log M)^{d+1} n}{M}\right)^{1/d}$ with probability at least $1 - \varepsilon/2$. Thus, the Euclidean length of the path between $u$ and $v$ is at most $c_1 n^{1/d}$ with probability at least $1 - \varepsilon/2$.

On the other hand, with probability at least $1-\varepsilon/2$, the Euclidean distance between $u$ and $v$ is at least $c_2 n^{1/d} := \sqrt{d}(\varepsilon n/2)^{1/d}/5 < n^{1/d}/4$: indeed, by~\eqref{eq:ballvolume} one has that $|B(u,c_2 n^{1/d})|\le (5c_2 n^{1/d}/\sqrt{d})^d = \varepsilon n/2$ and, therefore, the probability that $v$ falls in the ball $B(u,c_2 n^{1/d})$ is at most $\varepsilon/2$. 
We deduce that with probability at least $1-\varepsilon$, the stretch factor between $u$ and $v$ is at most $c = c_1/c_r$, which finishes the proof.
\end{proof}

\begin{remark}
Note that the previous corollary cannot be improved in the following sense: 
for every $d \ge 2$ and every constant $C > 0$, there exists $\varepsilon = \varepsilon(C) > 0$ such that, with probability at least $\varepsilon$, the stretch factor between two uniformly chosen vertices of $G_n$ at least $C$. 
We provide the following non-rigorous (but hopefully convincing) justification: tessellate $\mathbb T^d_n$ into $d$-cubes of side length $n^{1/d}/K$ for $K=K(C,d)$ large enough, 
and order the $K^d$ $d$-cubes so that every two consecutive cubes share a common $(d-1)$-side. 
The following events hold together with constant probability depending only on $K$: 
\begin{itemize}
    \item For every vertex $i \in [K^d-1]_0$, the vertex $i$ is embedded in the cube $i+1$ at distance at least $n^{1/d}/4K$ from its boundary. Moreover, if $i\ge 1$, vertex $i$ connects by an edge to $i-1$.
    \item The closest ancestor of $u$ among $[K^d-1]_0$ is $0$, and the closest ancestor of $v$ among $[K^d-1]_0$ is $K^d-1$.
\end{itemize}
Then, the Euclidean distance between $u$ and $v$ is bounded from above by $\sqrt{d} n^{1/d}$ but the Euclidean length of all edges on the path between them is bounded from below by $(K^{d}-1) n^{1/d}/2K$, which is at least $C \sqrt{d} n^{1/d}$ for all sufficiently large $K$.
\end{remark}

\subsection{\texorpdfstring{The height and the diameter of $G_n$}{}}\label{sec height of G_n}
Recall that, for a rooted tree $T$ and a vertex $v\in T$, we denote by $h(v, T)$ the distance from $v$ to the root of $T$, by $h(T)$ the height of $T$, and by $\mathrm{diam}(T)$ the diameter of $T$. 
The aim of this section is to analyze $h(G_n)$ and $\diam(G_n)$ (where $G_n$ is seen as a tree rooted at vertex 0), thus proving Theorem~\ref{thm diam}.

\begin{proof}[Proof of Theorem~\ref{thm diam}\eqref{line a} and the upper bound in Theorem~\ref{thm diam}\eqref{line b}]
We first prove \eqref{line a}. Note that it is equivalent to prove the statement if the path to 0 starts from an artificially added vertex with label $n$ (by doing this, we only increase the length of the path in $G_n$ by one since $\pnt(n)$ is chosen uniformly in $[n-1]_0$). 
Also, recall the construction of the path from a fixed vertex to the root via $(U_i)_{i\ge 1}$ given just above Lemma~\ref{lem 1 step}. Let
\begin{align*}
&T_1 = \min\{k: n U_1 U_2\dots U_k < 1\},\\
&T_2 = \min\{k: n U_1 U_2\dots U_k - k < 1\}.
\end{align*}
By Lemma~\ref{lem erlang} applied to the random variables $X_i = -\log U_i$, we may conclude that $T_1\le (1+\varepsilon)\log n$ and $T_2\ge (1-\varepsilon) \log n$ a.a.s.: indeed,
\begin{align}
    \mathbb P(T_1 > (1+\varepsilon)\log n)
    &=\hspace{0.3em} \mathbb P(X_1+\dots+X_{\lfloor (1+\varepsilon)\log n\rfloor}\le \log n)\nonumber\\
    &=\hspace{0.3em}\sum_{i=\lfloor (1+\varepsilon)\log n\rfloor}^{\infty} \dfrac{1}{i!} \exp(-\log n) (\log n)^i\nonumber\\
    &\le\hspace{0.3em} \exp(-\log n) \left(\dfrac{e \log n}{\lfloor (1+\varepsilon) \log n\rfloor}\right)^{\lfloor (1+\varepsilon) \log n\rfloor} \sum_{k=0}^{\infty} \dfrac{1}{(1+\varepsilon)^k}\nonumber\\
    &\le\hspace{0.3em} \left(\dfrac{1}{e}\left(\dfrac{e}{1+\varepsilon}\right)^{1+\varepsilon+o(1)}\right)^{\log n}\label{eq bla bla}.
\end{align}

Note that, by standard analysis, the function $x\in [1, +\infty)\mapsto (e/x)^x$ is decreasing, so $(e/x)^x < e$ for every $x > 1$. 
Thus, \eqref{eq bla bla} tends to 0 with $n$ for any fixed $\varepsilon > 0$. 
A similar computation ensures that $\mathbb P(T_2 < (1-\varepsilon) \log n) = o(1)$. 
This finishes the proof of Theorem~\ref{thm diam}\eqref{line a} since the hitting time of 0 from $n$ is dominated by $T_1$ and dominates $T_2$ due to~\eqref{eq path constr}.

For the upper bound in~\eqref{line b}, we first show that, for any $\varepsilon > 0$, the path from the artificially added vertex $n$ to $0$ has length more than $(e+\varepsilon) \log n$ with probability at most $o(1/n)$. 
Indeed, the length of the path from $n$ to $0$ stochastically dominates the length of any other path to the root (note that, for every $j\in [n-1]_0$, one has $\pnt^{k+1}(n) \sim \lfloor U_{k+1}\pnt^k(n)\rfloor$ and $\pnt^{k+1}(j) \sim \lfloor U_{k+1}\pnt^k(j)\rfloor$ for all $k\ge 0$). Moreover, $T_1$ dominates the length of the path from $n$ to $0$, and by Lemma~\ref{lem erlang}(i) for $X_i = -\log U_i$ we have
\begin{align*}
    \mathbb P(T_1 > (e+\varepsilon)\log n)
    &\le\hspace{0.3em} \mathbb P(X_1+\dots+X_{\lfloor(e+\varepsilon)\log n\rfloor}\le \log n)\\
    &\le\hspace{0.3em} \dfrac{2}{\lfloor(e+\varepsilon)\log n\rfloor!}\dfrac{(\log n)^{\lfloor(e+\varepsilon)\log n\rfloor}}{n}\\
    &\le\hspace{0.3em} \dfrac{2}{n}\left(\dfrac{e\log n}{\lfloor(e+\varepsilon)\log n\rfloor}\right)^{\lfloor(e+\varepsilon)\log n\rfloor} = o\left(\dfrac{1}{n}\right).
\end{align*}
The upper bound now follows from a union bound over all $n$ vertices.
\end{proof}

We turn our attention to the lower bound in Theorem~\ref{thm diam}\eqref{line b}. 
Recall that the Poisson $d$-NN tree $\PoNN$ is defined by sampling a Poisson random variable $N$ with mean $n$ 
and constructing the nearest neighbor tree on $N$ vertices in $\mathbb T^d_n$.

\begin{theorem}\label{thm diam poisson}
Fix $d\ge 1$ and $\varepsilon > 0$. Then, a.a.s.\ $\mathrm{diam}(\PoNN)\ge (2e-\varepsilon)\log n$.
\end{theorem}
\begin{proof}[Proof of the lower bound of Theorem~\ref{thm diam}\eqref{line b} assuming Theorem~\ref{thm diam poisson}]
We have that
\begin{equation}\label{eq LB poisson}
    \mathbb P(N\le n) = \sum_{i=0}^{n} \dfrac{\exp(-n) n^i}{i!}\ge \sum_{i=0}^{n} \dfrac{\exp(i-n)(n/i)^{i}}{e\sqrt{i}},
\end{equation}
where the inequality above follows from the classical bound $i!\le (i/e)^i e\sqrt{i}$ for all $i\ge 1$. Moreover, if $i\in [\lfloor n-\sqrt{n}\rfloor,n]$, there is a constant $C > 0$ such that
\begin{equation*}
    \left(\dfrac{n}{i}\right)^i = \exp\left(i \log\left(1 + \dfrac{n-i}{i}\right)\right) \ge \exp(n-i - C),
\end{equation*}
so we conclude that \eqref{eq LB poisson} is bounded from below by
\begin{equation*}
    \dfrac{1}{e\sqrt{n}}\sum_{i=\lfloor n-\sqrt{n}\rfloor}^{n} \exp(-C) \ge \dfrac{1}{\exp(1+C)} > 0.
\end{equation*}
\noindent
Thus, since the diameter of a tree is an increasing parameter under addition of new vertices, we have
\begin{align*}
    \mathbb P(\diam(G_n)\le (2e-\varepsilon)\log n)
    &\le\hspace{0.3em} \mathbb P(\diam(\PoNN)\le (2e-\varepsilon)\log n\mid N\le n)\\ 
    &\le\hspace{0.3em} \dfrac{\mathbb P(\diam(\PoNN)\le (2e-\varepsilon)\log n)}{\mathbb P(N\le n)} = o(1),
\end{align*}
and the proof is completed.
\end{proof}

The last part of this section is dedicated to proving Theorem~\ref{thm diam poisson}. We now present a high level overview of the main strategy, which is based on a renormalization argument. Given $\varepsilon > 0$, fix a sufficiently large positive integer $k = k(\varepsilon)$ and for each of the $N$ vertices of $\PoNN$, sample independently a vector of $k$ i.i.d.\ Bernoulli random variables $(Y_{v,1}, Y_{v,2}, \dots, Y_{v,k})$ with parameters $n^{-1/k}$. To every vertex, associate a color $c(v)$ among $\{c_1,c_2,\dots,c_k,c_{k+1}\}$, where $c(v) = c_i$ if $Y_{v,1} = \dots = Y_{v,i-1} = 1$, and if $i \le k$, then in addition we also must have $Y_{v,i} = 0$. The aim of these colors is to regroup the vertices of $\PoNN$ according to their arrival time, where vertices with color $c_i$ arrive before vertices of color $c_j$ if $i > j$, but at the same time keep the sizes of the color classes random to avoid dependencies. By our choice of parameters almost all vertices will have color $c_1$, a lot less will have color $c_2$, even less will have color $c_3$, and so on.

Now, define $R = R(n, d, k) := \lfloor (\log n)^{k(d+2)/d}\rfloor^d$ and $L = L(n, d, k) := R \lfloor R^{-1/d} n^{1/k d} (\log n)^{-(d+2)/d}\rfloor^{d}$, 
and embed the $N$ vertices in $\mathbb T^d_n$ without revealing anything but their geometric positions, and consider a tessellation $\mathcal T_1$ of $\mathbb T^d_n$ into $R L^{k-1}$ $d$-cubes, each of volume $(1+o(1)) n^{1/k}(\log n)^{d+2}$ but larger than $n^{1/k}(\log n)^{d+2}$. 
For every cube $q_1$ in $\mathcal T_1$ and every vertex $v_1$ in $q_1$, we assign a label associated to $v_1$ which equals $|\{u \in V(\PoNN: u \in q_1, X_u < X_{v_1})\}|$, that is, the number of vertices $u \in q_1$, which arrived before $v_1$. We say that the label of $v_1$ with respect to $q_1$ is its 1-label. Next, in every cube $q_1\in \mathcal T_1$, we look for a decreasing path (with respect to the 1-labels) of length $(e-\varepsilon)(\log n)/k$ inside this cube between an arbitrary vertex in this cube and one of color different from $c_1$. In the construction of such a path, the colors and the 1-labels are revealed consecutively: if the parent of a vertex is detected, first its color is revealed, and if it is $c_1$, then its 1-label is revealed as well. Once a vertex of color different from $c_1$ is found, we stop the construction of the decreasing path without revealing the label of this last vertex (at this point we only know that its color is not $c_1$). Roughly speaking, this path consists of ancestors belonging to the same smallest cube, see Figure~\ref{fig 6}.

At this point, we show that a sufficiently large number of cubes contain decreasing paths of length $(e-\varepsilon)(\log n)/k$ in their interior. 
Then, we regroup the cubes of $\mathcal T_1$ into larger sets that form a coarser tessellation $\mathcal T_2$ containing $R L^{k-2}$ cubes, each of volume $(1+o(1)) n^{2/k}(\log n)^{2(d+2)}$ but larger than $n^{2/k}(\log n)^{2(d+2)}$. 
For every cube $q_2$ in $\mathcal T_2$ and every vertex $v_2$ in $q_2$, we assign a 2-label associated to $v_2$ that is equal to the number of vertices that arrived before $v_2$ in $q_2$. 
Note that the order, induced by the 2-labels in a cube $q_1\in \mathcal T_1$, is the same as the order given by their 1-labels in $q_1$.  
Next, we look for decreasing paths (with respect to the 2-labels) between the vertices which are ends of paths of length $(e-\varepsilon)(\log n)/k$ from the previous stage to the vertices in colors different from $c_1$ and $c_2$. 
Once again, we show that sufficiently many such cubes contain paths with length at least $(e-\varepsilon)(\log n)/k$.

We repeat the above procedure $k-3$ times. 
Then, by gluing the paths constructed at all $k-3$ stages, we show that a.a.s.\ one may find a decreasing path of length $(k-3)(e-\varepsilon)(\log n)/k$ (see Figure~\ref{fig 6}). 
We keep at least two such disjoint paths to make the transition from $h(\PoNN)$ to $\diam(\PoNN)$.

\begin{figure}
\centering
\begin{tikzpicture}[line cap=round,scale=0.6,line join=round,x=1cm,y=1cm]
\clip(-9.5,-20.99310136694685) rectangle (34.032489021253255,5.896464733479574);
\draw [line width=0.8pt,dash pattern=on 1pt off 1.5pt] (0.3,-2.5)-- (0.3,-2.4);
\draw [line width=0.8pt,dash pattern=on 1pt off 1.5pt] (0.3,-2.4)-- (0.5,-2.3);
\draw [line width=0.8pt,dash pattern=on 1pt off 1.5pt] (0.5,-2.3)-- (0.6,-2.7);
\draw [line width=0.8pt,dash pattern=on 1pt off 1.5pt] (0.6,-2.7)-- (0.8,-2.5);
\draw [line width=0.8pt,dash pattern=on 1pt off 1.5pt] (0.8,-2.5)-- (0.6,-2.9);
\draw [line width=0.8pt,dash pattern=on 1pt off 1.5pt] (0.6,-2.9)-- (0.2,-2.8);
\draw [line width=0.8pt,dash pattern=on 0.8pt off 3pt] (0.2,-2.8)-- (-1,-2.5);
\draw [line width=0.8pt,dash pattern=on 0.8pt off 3pt] (-1,-2.5)-- (-1,-3.5);
\draw [line width=0.8pt,dash pattern=on 0.8pt off 3pt] (-1,-3.5)-- (0,-4.5);
\draw [line width=0.8pt,dash pattern=on 0.8pt off 3pt] (0,-4.5)-- (1.5,-4);
\draw [line width=0.8pt,dash pattern=on 0.8pt off 3pt] (1.5,-4)-- (2.5,-4.5);
\draw [line width=0.8pt,dash pattern=on 0.8pt off 3pt] (2.5,-4.5)-- (2.5,-3.5);
\draw [line width=0.8pt,dash pattern=on 0.8pt off 3pt] (2.5,-3.5)-- (2,-2.5);
\draw [line width=0.8pt,dash pattern=on 0.8pt off 3pt] (2,-2.5)-- (1.5,-1);
\draw [line width=0.8pt,dash pattern=on 0.8pt off 0.8pt] (1.5,-1)-- (0,-1);
\draw [line width=0.8pt,dash pattern=on 0.8pt off 0.8pt] (0,-1)-- (-1,-0.5);
\draw [line width=0.8pt] (0,-2)-- (0,-3);
\draw [line width=0.8pt] (0,-3)-- (1,-3);
\draw [line width=0.8pt] (1,-3)-- (1,-2);
\draw [line width=0.8pt] (1,-2)-- (0,-2);
\draw [line width=0.8pt] (-2,0)-- (3,0);
\draw [line width=0.8pt] (3,-5)-- (3,0);
\draw [line width=0.8pt] (-2,-5)-- (3,-5);
\draw [line width=0.8pt] (-2,-5)-- (-2,0);
\draw [line width=0.8pt] (-1,-0.5)-- (2,4);
\draw [line width=0.8pt] (2,4)-- (6,0);
\draw [line width=0.8pt] (6,0)-- (12,2);
\draw [line width=0.8pt] (12,2)-- (16,-4);
\draw [line width=0.8pt] (16,-4)-- (10,-8);
\draw [line width=0.8pt] (10,-8)-- (10,-14);
\draw [line width=0.8pt] (10,-14)-- (12,-18);
\draw [line width=0.8pt] (12,-18)-- (10,-18);
\draw [line width=0.8pt] (10,-18)-- (2,-10);
\draw [line width=0.8pt] (2,-10)-- (-2,-16);
\draw [line width=0.8pt] (-2,-16)-- (-6,-8);
\draw [line width=0.8pt] (-7,5)-- (18,5);
\draw [line width=0.8pt] (18,5)-- (18,-20);
\draw [line width=0.8pt] (18,-20)-- (-7,-20);
\draw [line width=0.8pt] (-7,-20)-- (-7,5);
\end{tikzpicture}
\caption{Gluing together long paths -- towards the proof of Theorem~\ref{thm diam}\eqref{line b}.}
\label{fig 6}
\end{figure}

\vspace{0.5em}
We now proceed to the formal realization of the above strategy. 
We start with the formal definition of the tessellations. 
Define $(\mathcal T_i)_{1\le i\le k-2}$ as the family of tessellations such that, first, $\mathcal T_i$ refines $\mathcal T_{i+1}$ for every $i\in [k-2]$, 
and second, the cubes in $\mathcal T_i$ are all congruent and axis-parallel and have sides of length $(n/R L^{k-i})^{1/d} = (1+o(1)) n^{i/dk} (\log n)^{i(d+2)/d}$. 
Moreover, set $S = S(d,n,i) := 3\sqrt{d} ((\log n)^{d+1} n^{i/k})^{1/d}$ and, for every $i\in [k-2]$ and every $q_i\in \mathcal T_i$, define $\intqi$ ($\overline{\intqi}$, respectively) to be the central axis-parallel subcube of $q_i$ with boundary formed by the points at distance $2S$ ($S$, respectively) from the boundary of $q_i$ (see Figure~\ref{fig 7}). 

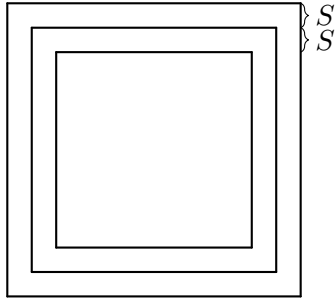
\begin{figure}
    \centering
\begin{tikzpicture}[line cap=round,scale=0.65,line join=round,x=1cm,y=1cm]
\clip(-6.316028466187328,-5.081302215065078) rectangle (10.282350920657251,2.8003181856064043);
\draw [line width=0.8pt] (-1,2)-- (-1,-4);
\draw [line width=0.8pt] (-1,-4)-- (5,-4);
\draw [line width=0.8pt] (5,-4)-- (5,2);
\draw [line width=0.8pt] (5,2)-- (-1,2);
\draw [line width=0.8pt] (-0.5,1.5)-- (4.5,1.5);
\draw [line width=0.8pt] (0,1)-- (4,1);
\draw [line width=0.8pt] (4.5,-3.5)-- (-0.5,-3.5);
\draw [line width=0.8pt] (4,-3)-- (0,-3);
\draw [line width=0.8pt] (-0.5,1.5)-- (-0.5,-3.5);
\draw [line width=0.8pt] (0,1)-- (0,-3);
\draw [line width=0.8pt] (4,1)-- (4,-3);
\draw [line width=0.8pt] (4.5,1.5)-- (4.5,-3.5);
\draw [decorate,decoration={brace,amplitude=3pt},xshift=0pt,yshift=0pt]
(5,1.5)-- (5,1) node {};
\draw [decorate,decoration={brace,amplitude=3pt},xshift=0pt,yshift=0pt]
(5,2)-- (5,1.5) node {};
\draw[color=black] (5.5,1.25) node {$S$};
\draw[color=black] (5.5,1.75) node {$S$};
\end{tikzpicture}
\caption{$q_i$ is the largest cube, $\overline{\intqi}$ is the second largest, and $\intqi$ is the smallest cube. The exploration process in $q_i$ is stopped if it leaves $\overline{\intqi}$ before finding a vertex of color among $(c_j)_{j=i+1}^{k+1}$.}
\label{fig 7}
\end{figure}

Embed $N$ unlabeled and uncolored vertices in $\mathbb T^d_n$ uniformly at random. Our randomized algorithm performs $k-3$ steps. At step $i\in [k-3]$, expose the random variables $(Y_{v,j})_{v\in [N-1]_0,\, j\in [i-1]}$ and delete all vertices with colors among $c_1, \dots, c_{i-1}$. Then, fix a cube $q_i\in \mathcal T_i$ and an arbitrary vertex $v_i$ (if it exists, otherwise stop the process) remaining in $\intqi$. Start building a decreasing path from $v_i$ by, first, revealing $Y_{\pnt(v_i),i}$, and then:
\begin{itemize}
    \item if $Y_{\pnt(v_i),i} = 0$ (that is, if $c(v_i) = c_i$), then return the $i$-label of this vertex in $q_i$ and repeat the analysis for $\pnt(v_i)$ itself,
    \item if not, stop the process in $q_i$.
\end{itemize}
Also, stop the process if $\pnt(v_i)$ leaves $\overline{\intqi}$.

\begin{lemma}\label{lem stoppage}
Fix $\varepsilon > 0$, $i\in [k-2]$ and a cube $q_i\in \mathcal T_i$. 
Also, fix a vertex $v_i\in \intqi$ satisfying $Y_{v_i,1} = \dots = Y_{v_i,i-1} = 1$. 
Then, there is $\delta  = \delta(\varepsilon) \in (0,1/k)$ such that the following event holds with probability at least $n^{-1/k+\delta}$: 
when the $i$-th stage of the algorithm terminates, the decreasing path from $v_i$ ends in $\overline{\intqi}$ and has length (with respect to the graph distance) at least $(e-\varepsilon)\log n/k$. 
\end{lemma}
\begin{proof}
By Corollary~\ref{cor PoNN} applied with $x_n = n^{-i/k}$ (this corresponds to vertices $v$ with $Y_{v,1}=\dots=Y_{v,i}=1$) the Euclidean length of any path in $\PoNN$, ending at a vertex of color among $\{c_{i+1},\dots,c_{k+1}\}$, is dominated by $S$ with probability at least $1 - 4n^{-(k-i)/k}$. We call this event $\mathcal E_i$.

Now, we estimate the probability that the path from $v_i$ hits the set of vertices in $q_i$ in color among $\{c_{i+1},\dots,c_{k+1}\}$ only after $(e-\varepsilon) (\log n)/k$ or more steps. Recall the construction of the path using the family $(U_i)_{i\ge 1}$. 
Then, since $\mathcal E_i$ holds a.a.s., by~\eqref{eq path constr} it remains to estimate the probability of the event
\begin{equation}\label{eq good event}
\{ n^{(k-i)/k}U_1U_2\dots U_{\lceil (e-\varepsilon)(\log n)/k\rceil}\ge n^{(k-i-1)/k} + \lceil (e-\varepsilon)(\log n)/k\rceil \}.
\end{equation}
Applying Lemma~\ref{lem erlang} to $X_i = -\log U_i$, we obtain that, for every small enough $\varepsilon > 0$, there is a sufficiently small $\delta = \delta(\varepsilon)\in (0,1/k)$ such that
\begin{align*}
&\mathbb P(n^{(k-i)/k}U_1U_2\dots U_{\lceil (e-\varepsilon)(\log n)/k\rceil}
\ge n^{(k-i-1)/k} + \lceil (e-\varepsilon)(\log n)/k\rceil)\\
=\hspace{0.3em} 
&\mathbb P(X_1 + \dots + X_{\lceil (e-\varepsilon)(\log n)/k\rceil}\le (1+o(1)) (\log n)/k) \\
\ge\hspace{0.2em} 
&\dfrac{1}{\lceil (e-\varepsilon)(\log n)/k\rceil!}\exp\left(-(1+o(1))\dfrac{\log n}{k}\right) \left((1+o(1))\log n/k\right)^{\lceil (e-\varepsilon)(\log n)/k\rceil}\ge n^{-1/k+2\delta}.
\end{align*}

\noindent
This finishes the proof of the lemma since the event $\mathcal E_i$ holds jointly with~\eqref{eq good event} with probability at least $n^{-1/k+2\delta} - 4n^{-(k-i)/k} \ge n^{-1/k+\delta}$ by our choice of $i$ and $\delta$.
\end{proof}

\begin{corollary}\label{cor diam final}
For every $\varepsilon > 0$, there is $\delta = \delta(\varepsilon) > 0$ such that, for all $i\in [k-2]$ and $q_i\in \mathcal T_i$, $q_i$ contains a decreasing path of length at least $(i-1) (e-\varepsilon) (\log n)/k$ with probability at least $n^{-1/k+\delta}/2$.
\end{corollary}
\begin{proof}
We argue by induction. Fix $\varepsilon > 0$ and $\delta = \delta(\varepsilon)$ given in Lemma~\ref{lem stoppage}. 

For $i=1$, the statement is trivial. Suppose that the induction hypothesis holds for $j = i-1$. Note that, since $S$ in negligible compared to the side length of $q_i$, $\intqi$ contains at least $(\log n)^{d+1} n^{1/k}$ cubes of $\mathcal T_{i-1}$ (call these smaller cubes \emph{central}). 
Then, by Lemma~\ref{lem stoppage} combined with Chernoff's bound~\eqref{chern} for a fixed cube $q_i\in \mathcal T_i$, $\overline{\intqi}$ contains at least one end of a path of length $(i-1) (e-\varepsilon) (\log n)/k$ which is entirely contained in a central cube with probability $1-o(1/n)$. 
Denote this event by $\mathcal A(q_i)$ and note that a.a.s.\ $\{\mathcal A(q_i): q_i\in \mathcal T_i\}$ all hold together. Under $\mathcal A(q_i)$, assign to $q_i$ one vertex $v_i\in q_i$ with color in $\{c_i, \dots, c_{k+1}\}$ which is the end of a path of length $(i-1)(e-\varepsilon)(\log n)/k$ and start constructing a decreasing path in $q_i$ from $v_i$. 
Under $\mathcal A(q_i)$, by Lemma~\ref{lem stoppage} one reaches a vertex positioned in $\overline{\intqi}$ and with color in $\{c_{i+1},\dots, c_{k+1}\}$ only after at least $(e-\varepsilon) (\log n)/k$ more steps with probability at least $n^{-1/k+\delta}$. Denote this event by $\mathcal B(q_i)$. Thus, we get that
\begin{equation*}
\mathbb P(\mathcal B(q_i))\ge \mathbb P(\mathcal B(q_i)\mid\mathcal A(q_i))\mathbb P(\mathcal A(q_i))\ge n^{-1/k+\delta}/2,
\end{equation*}
and this concludes the induction.
\end{proof}

\begin{proof}[Proof of Theorem~\ref{thm diam poisson}]
Fix two cubes $q'_{k-1}$ and $q''_{k-1}$ in $\mathcal T_{k-1}$. Both $\mathrm{int}(q'_{k-1})$ and $\mathrm{int}(q''_{k-1})$ contain at least $n^{1/k} (\log n)^{d+1}$ cubes in $\mathcal T_{k-2}$, and by Chernoff's bound~\eqref{chern} and Corollary~\ref{cor diam final} some of these a.a.s.\ contain decreasing paths of length at least $(k-3)(e-\varepsilon)(\log n)/k$. 
The fact that these decreasing paths are disjoint (since they are entirely contained in different cubes of $\mathcal T_{k-1}$) finishes the proof of Theorem~\ref{thm diam poisson} up to an appropriate choice of the values of $k$ and $\varepsilon$.
\end{proof}

\section{Proof of Theorem~\ref{thm G infinity} -- on the properties of \texorpdfstring{$G_{\infty}$}{}}\label{sec infty}
This section is dedicated to the analysis of $G_{\infty}$.

\begin{proof}[Proof of Theorem~\ref{thm G infinity}\eqref{item local limit}]
Fix any positive integer $M$ and a vertex $v_n$ of $G_n$ (resp.\ $v_{\infty}$ of $G_{\infty}$) in position $0\in \mathbb T^d_n$ (resp.\ $0\in \mathbb R^d$). 
We will construct a local coupling of $G_n$ and $G_{\infty}$ around $v_n$ and $v_{\infty}$ such that a.a.s.\ the balls $B_{G_n}(v_n, M)$ and $B_{G_{\infty}}(v_{\infty}, M)$ coincide. 
The following arguments tacitly assume the results from Palm theory developed in Subsection~\ref{sec:Palm}, in particular, 
 artificially adding the vertex $v_n$ (resp.\ $v_{\infty}$) at the origin leaves the intensity of the Poisson Point Process unchanged.

The coupling goes as follows: sample the same PPP(1) in the axis-parallel cube $Q_n$ (resp.\ $Q_{\infty}$) of side length $n^{1/2d}$ centered at the origin of $\mathbb T^d_n$ (resp.\ of $\mathbb R^d$). 
Moreover, for every pair of vertices $w_n$ and $w_{\infty}$ in $Q_n$ and $Q_{\infty}$ that correspond to each other, we couple their arrival times so that $X_{w_n} = X_{w_{\infty}}$. 
In the sequel, we refer to both $v_n$ and $v_{\infty}$ (resp.\ $Q_n$ and $Q_{\infty}$) as $v$ (resp.\ as $Q$) when the arguments allow to identify the finite with the infinite setting. 
For all $i\in [M]_0$, denote by $u_i$ the $i$-th ancestor of $v$ (so $u_0 = v$). 
Also, denote by $\mathcal A$ the event $X_{u_M}\ge (\log n)^{-M-1}$. 
One has
\begin{equation*}
\mathbb P(\overline{\mathcal A}) \le \mathbb P\left(\{X_{u_0}\le (\log n)^{-1}\}\cup \bigcup_{i\in [M]}\left\{X_{u_i}\le (\log n)^{-1} X_{u_{i-1}}\right\}\right)\le (M+1)(\log n)^{-1} = o(1).
\end{equation*}
Under the event $\mathcal A$, by Theorem~\ref{thm 4.1} for $i = \lfloor n (\log n)^{-M-1}\rfloor$ and Corollary~\ref{cor PoNN} for $x_n = (\log n)^{-M-1}$ there is a sufficiently large constant $C = C(d,M) > 0$ such that the descending tree $T_{u_M}$ of $u_M$ is a.a.s.\ contained in the ball $B(u_M, (\log n)^{C})\subseteq B(v, 2(\log n)^{C})$. 
Thus, the positions of the vertices of $G_n$ in $\mathbb T^d_n\setminus Q_n$ (resp.\ of $G_{\infty}$ in $\mathbb R^d\setminus Q_{\infty}$) 
a.a.s.\ do not influence the structure of the tree $T_{u_M}$ because the ball $B(v, 2(\log n)^{C})$ has diameter $4(\log n)^{C}$ and is at distance $n^{\Omega_d(1)}\gg 4(\log n)^{C}$ 
from $\mathbb T^d_n\setminus Q_n$, resp.\ from $\mathbb R^d\setminus Q_{\infty}$ (see Figure~\ref{fig 3}). 
Since the ball with center $v_n$ in $G_n$ (resp. $v_{\infty}$ in $G_{\infty}$) is contained in $T_{u_M}$, for every $M\ge 1$, under the constructed coupling a.a.s.\ $B_{G_n}(v_n, M) = B_{G_{\infty}}(v_{\infty}, M)$, 
which proves the local convergence of $(G_n)_{n\ge 1}$ to $G_{\infty}$.
\end{proof}

\begin{figure}
\centering
\begin{tikzpicture}[line cap=round,line join=round,x=1cm,y=1cm]
\clip(-10.4,-4.146057605080144) rectangle (7.90375770354596,5.124930735890613);
\draw [line width=0.8pt] (-6,5)-- (-6,-4);
\draw [line width=0.8pt] (-6,-4)-- (3,-4);
\draw [line width=0.8pt] (3,-4)-- (3,5);
\draw [line width=0.8pt] (3,5)-- (-6,5);
\draw [line width=0.8pt] (-1.5,0.5)-- (-2,0.5);
\draw [line width=0.8pt] (-2,0.5)-- (-1.5,1);
\draw [line width=0.8pt] (-1.5,1)-- (-0.5,0.5);
\draw [line width=0.8pt] (-0.5,0.5)-- (-1,-0.5);
\draw [line width=0.8pt] (-1,-0.5) circle (1.7cm);
\draw [line width=0.8pt] (-1.5,0.5) circle (3.5355339059327378cm);
\draw [line width=0.8pt] (-1.5,0.5)-- (-1.6,0.2);
\draw [line width=0.8pt] (-1.5,0.5)-- (-1.4,0.2);
\draw [line width=0.8pt] (-1.6,0.2)-- (-1.8,0.2);
\draw [line width=0.8pt] (-1.6,0.2)-- (-1.8,0);
\draw [line width=0.8pt] (-1.6,0.2)-- (-1.6,0);
\draw [line width=0.8pt] (-2,0.5)-- (-2.2,0.2);
\draw [line width=0.8pt] (-1.5,1)-- (-1.3,1.1);
\draw [line width=0.8pt] (-1.5,1)-- (-1,1);
\draw [line width=0.8pt] (-1,1)-- (-0.6,1);
\draw [line width=0.8pt] (-1,1)-- (-0.8,0.8);
\draw [line width=0.8pt] (-1.3,1.1)-- (-1.1,1.1);
\draw [line width=0.8pt] (-0.5,0.5)-- (-0.4,0.8);
\draw [line width=0.8pt] (-0.5,0.5)-- (-0.2,0.6);
\draw [line width=0.8pt] (-0.5,0.5)-- (-0.2,0.2);
\draw [line width=0.8pt] (-0.2,0.2)-- (-0.4,0);
\draw [line width=0.8pt] (-0.4,0)-- (-0.2,-0.2);
\draw [line width=0.8pt] (-0.2,0.2)-- (0.2,0.2);
\draw [line width=0.8pt] (-1,-0.5)-- (-1.4,-0.6);
\draw [line width=0.8pt] (-1,-0.5)-- (-1.2,-0.8);
\draw [line width=0.8pt] (-1.4,-0.6)-- (-1.8,-1);
\draw [line width=0.8pt] (-1.4,-0.6)-- (-1.4,-1);
\draw [line width=0.8pt] (-1.2,-0.8)-- (-1.2,-1.2);
\draw [line width=0.8pt] (-1.2,-0.8)-- (-1,-1);
\draw [line width=0.8pt] (-1,-0.5)-- (-0.4,-0.8);
\draw [line width=0.8pt] (-1,-0.5)-- (-0.6,-1.2);
\draw [line width=0.8pt] (-0.6,-1.2)-- (-0.2,-1.2);
\draw [line width=0.8pt] (-0.2,-1.2)-- (0,-1);
\draw [line width=0.8pt] (-0.4,-0.8)-- (-0.4,-0.4);
\draw [line width=0.8pt] (-0.4,-0.8)-- (-0.2,-0.6);
\draw [line width=0.8pt] (-1.8,-1)-- (-1.6,-1.6);
\draw [line width=0.8pt] (-1.8,-1)-- (-2.2,-1);
\draw [line width=0.8pt] (-2.2,-1)-- (-2.2,-0.8);
\draw [line width=0.8pt] (-2.2,-1)-- (-2.6,-0.6);
\draw [line width=0.8pt] (-1.6,-1.6)-- (-1.4,-1.8);
\draw [line width=0.8pt] (-1.2,-1.2)-- (-1.2,-1.6);
\draw [line width=0.8pt] (-1.2,-1.2)-- (-0.8,-1.6);
\draw [line width=0.8pt] (-0.8,-1.6)-- (-1,-1.8);
\draw [line width=0.8pt] (-0.8,-1.6)-- (-0.6,-1.8);
\begin{scriptsize}
\draw [fill=black] (-1.5,0.5) circle (1pt);
\draw[color=black] (-1.4,0.62) node {$u_0$};
\draw [fill=black] (-2,0.5) circle (1pt);
\draw [fill=black] (-1.5,1) circle (1pt);
\draw [fill=black] (-0.5,0.5) circle (1pt);
\draw [fill=black] (-1,-0.5) circle (1pt);
\draw[color=black] (-1.2,-0.38) node {$u_M$};
\draw[color=black] (2.5,4.5) node {\Large{$Q$}};
\draw [fill=black] (-1.6,0.2) circle (1pt);
\draw [fill=black] (-1.4,0.2) circle (1pt);
\draw [fill=black] (-1.8,0.2) circle (1pt);
\draw [fill=black] (-1.8,0) circle (1pt);
\draw [fill=black] (-1.6,0) circle (1pt);
\draw [fill=black] (-2.2,0.2) circle (1pt);
\draw [fill=black] (-1.3,1.1) circle (1pt);
\draw [fill=black] (-1,1) circle (1pt);
\draw [fill=black] (-0.6,1) circle (1pt);
\draw [fill=black] (-0.8,0.8) circle (1pt);
\draw [fill=black] (-1.1,1.1) circle (1pt);
\draw [fill=black] (-0.4,0.8) circle (1pt);
\draw [fill=black] (-0.2,0.6) circle (1pt);
\draw [fill=black] (-0.2,0.2) circle (1pt);
\draw [fill=black] (-0.4,0) circle (1pt);
\draw [fill=black] (-0.2,-0.2) circle (1pt);
\draw [fill=black] (0.2,0.2) circle (1pt);
\draw [fill=black] (-1.4,-0.6) circle (1pt);
\draw [fill=black] (-1.2,-0.8) circle (1pt);
\draw [fill=black] (-1.8,-1) circle (1pt);
\draw [fill=black] (-1.4,-1) circle (1pt);
\draw [fill=black] (-1.2,-1.2) circle (1pt);
\draw [fill=black] (-1,-1) circle (1pt);
\draw [fill=black] (-0.4,-0.8) circle (1pt);
\draw [fill=black] (-0.6,-1.2) circle (1pt);
\draw [fill=black] (-0.2,-1.2) circle (1pt);
\draw [fill=black] (0,-1) circle (1pt);
\draw [fill=black] (-0.4,-0.4) circle (1pt);
\draw [fill=black] (-0.2,-0.6) circle (1pt);
\draw [fill=black] (-1.6,-1.6) circle (1pt);
\draw [fill=black] (-2.2,-1) circle (1pt);
\draw [fill=black] (-2.2,-0.8) circle (1pt);
\draw [fill=black] (-2.6,-0.6) circle (1pt);
\draw [fill=black] (-1.4,-1.8) circle (1pt);
\draw [fill=black] (-1.2,-1.6) circle (1pt);
\draw [fill=black] (-0.8,-1.6) circle (1pt);
\draw [fill=black] (-1,-1.8) circle (1pt);
\draw [fill=black] (-0.6,-1.8) circle (1pt);
\draw [fill=black] (2.3,0.5) circle (0.5pt);
\draw [fill=black] (2.5,0.5) circle (0.5pt);
\draw [fill=black] (2.7,0.5) circle (0.5pt);
\draw [fill=black] (-5.7,0.5) circle (0.5pt);
\draw [fill=black] (-5.5,0.5) circle (0.5pt);
\draw [fill=black] (-5.3,0.5) circle (0.5pt);
\draw [fill=black] (-1.5,4.7) circle (0.5pt);
\draw [fill=black] (-1.5,4.5) circle (0.5pt);
\draw [fill=black] (-1.5,4.3) circle (0.5pt);
\draw [fill=black] (-1.5,-3.7) circle (0.5pt);
\draw [fill=black] (-1.5,-3.5) circle (0.5pt);
\draw [fill=black] (-1.5,-3.3) circle (0.5pt);
\end{scriptsize}
\end{tikzpicture}
\caption{An illustration of the proof of Theorem~\ref{thm G infinity}~\eqref{item local limit}. The smaller ball is $B(u_M, (\log n)^C)$ and the bigger ball is $B(u_0, 2(\log n)^C)$. The tree contained in the smaller ball is $T_{u_M}$.}
\label{fig 3}
\end{figure}
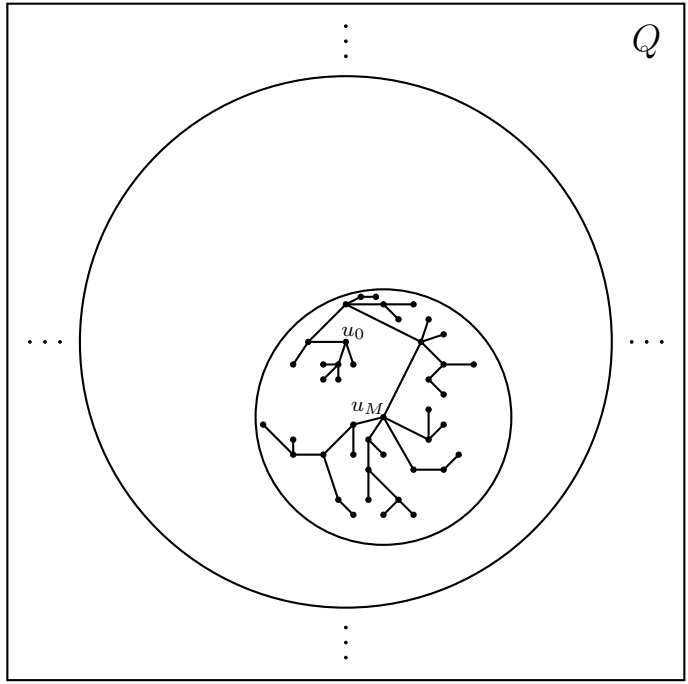

For the proof of Theorem~\ref{thm G infinity}\eqref{item loc finite}, we recall the notion of cubic net defined in Section~\ref{subsec cubic nets}.

\begin{proof}[Proof of Theorem~\ref{thm G infinity}\eqref{item loc finite}]
Fix a vertex $v$ with position $p = p(v)\in \mathbb R^d$ and arrival time $x = x(v)\in (0,1]$ (note that there is a.s.\ no vertex arriving at time $0$). 
Fix a sequence of real numbers $(\ell_i)_{i\ge 1}$ tending to infinity so that the interiors of all boundary cubes in all cubic nets $(N(p, \ell_i))_{i\ge 1}$ are disjoint. 
For all $i\ge 1$, denote by $\mathcal A_i$ the event that each boundary cube in $N(p, \ell_i)$ contains a vertex with arrival time in the interval $(0, x_v)$. Since $\lim_{i\to \infty} |[-\ell_i, \ell_i]^d|/K^d\to \infty$, we have that $\mathbb P(\mathcal A_i) \longrightarrow 1$ as $i\to \infty$. Thus,
\begin{equation*}
\mathbb P(\deg_{G_{\infty}}(v) = \infty) = \mathbb P([p-1/2, p+1/2]^d\cap V(G_{\infty}) = \infty) + \mathbb P(\cap_{i\ge 1} \overline{\mathcal A_i})\le 0 + \lim_{i\to \infty} \mathbb P(\overline{\mathcal A_i}) = 0.
\end{equation*}
We conclude that
\begin{align*}
&\mathbb P(\exists v\in V(G_{\infty}): \deg_{G_{\infty}}(v) = \infty)\\ 
=\hspace{0.3em} 
&\mathbb P(\exists v\in V(G_{\infty}): \deg_{G_{\infty}}(v) = \infty\mid |G_{\infty}| \text{ is countable}) + \mathbb P(|G_{\infty}| \text{ is not countable})\\
\le\hspace{0.3em} 
&0 + \mathbb P\left(\exists (y_i)_{i=1}^d\in \mathbb Z^d: \prod_{i=1}^d [y_i-1, y_i+1]\cap V(G_{\infty}) = \infty\right) = 0,
\end{align*}
and this proves Theorem~\ref{thm G infinity}\eqref{item loc finite}.
\end{proof}

To prove Theorem~\ref{thm G infinity}\eqref{item recurrence}, we need some preparation.

\begin{observation}\label{ob containment}
Fix a real number $r\in (0,1]$ and set $\varepsilon = (512 d^2)^{-2d}$ and $M = 16 d^2 r^{-1/d}$. 
Fix a $d$-dimensional cube $Q\subseteq \mathbb R^d$ of side length $M$. 
Then, $Q$ satisfies the following two conditions with probability at least 
$1 - (8d^2)^{-d}$: 
it contains no vertex of $G_{\infty}$ with arrival time in the interval $[(1-\varepsilon)r, (1+\varepsilon)r]$, and it contains a vertex with arrival time in the interval $(0, (1-\varepsilon)r)$.
\end{observation}
\begin{proof}
On the one hand, the number of vertices in $Q$ with arrival time in the interval $[(1-\varepsilon)r, (1+\varepsilon)r]$ is a Poisson random variable with mean $2\varepsilon r M^d$. 
Hence, this number is at least 1 with probability
\begin{equation*}
1 - \exp(-2\varepsilon r M^d)\le 1 - \exp(-2(32d^2)^{-d})\le 2(32 d^2)^{-d},
\end{equation*}
where in the last step we used the inequality $1\le t + \exp(-t)$ for all $t\ge 0$. 
Moreover, there are no vertices in $Q$ with arrival time in the interval $(0, (1-\varepsilon)r)$ with probability
\begin{equation*}
\exp(-(1-\varepsilon)r M^d)\le \exp(-(16 d^2)^d/2)\le (16 d^2)^{-d},
\end{equation*}
where at the last step we used that $\exp(-t/2)\le t^{-1}$ for all $t\ge 16$. The observation follows by a union bound.
\end{proof}

\begin{lemma}\label{lem finite desc trees}
The descending tree of any vertex in $G_{\infty}$ is finite a.s.
\end{lemma}
\begin{proof}
We argue by contradiction. 
Fix a rational number $r\in (0,1]$ and define $M$ and $\varepsilon$ as in Observation~\ref{ob containment}. 
Denote by $\mathcal A_r$ the event that there is an infinite descending path $(u_i)_{i\ge 0}$ in $G_{\infty}$ such that $\lim_{i\to \infty} X_{u_i}\in [r-\varepsilon/2, r+\varepsilon/2]$ 
(note that the arrival times of $u_i$ decrease with $i$ and thus the limit is well defined).
Note that since $G_{\infty}$ is a.s.\ locally finite by Theorem~\ref{thm G infinity}\eqref{item loc finite}, it is sufficient to show that, for all rational $r\in (0,1]$, the event $\mathcal A_r$ has probability 0.

Fix a tessellation $\mathcal T$ of $\mathbb R^d$ into axis-parallel cubes of side length $M$ and color in blue the $d$-cubes of $\mathcal T$ (the vertices therein are also colored in blue) that do not satisfy the conditions from Observation~\ref{ob containment}. 
Thus, all cubes are colored in blue independently and with the same probability. We equip $\mathcal T$ with the following distance: 
for any cubes $Q_1,Q_2\in \mathcal T$ with centers $q_1,q_2$, set $d_{\mathcal T}(Q_1,Q_2)$ to be the $\ell^1$-distance between $q_1$ and $q_2$. 
Note that a vertex $u_i$ with arrival time in $[(1-\varepsilon)r, (1+\varepsilon)r]$ must belong to a blue cube $Q$, and if there are no other blue cubes at distance at most $d^2$ from $Q$ (with respect to $d_{\mathcal T}$), then $u_{i+1}$ must be positioned in the $\ell^1$-ball $B_{\mathcal T}(Q, d^2)$ for $d_{\mathcal T}$: 
indeed, fix any point $p_1$ outside $B_{\mathcal T}(Q,d^2)$, any point $p_2\in Q$ and set $p_4$ as the intersection point of the segment $p_1p_2$ with the boundary of $B_{\mathcal T}(Q,d^2)$. 
Also, let $p_3$ by an arbitrary point in some cube of $\mathcal T$, containing $p_4$ on its boundary (see Figure~\ref{fig 4}). Then, by the fact that every point on the boundary of $B_{\mathcal T}(Q,d^2)$ is at $\ell^2$-distance more than $Md^2/d = Md$ from $Q$ and the triangle inequality, we have
\begin{equation*}
|p_1p_2| = |p_1p_4|+|p_4p_2| > |p_1p_4| + Md\ge |p_1p_4|+|p_4p_3|\ge |p_1p_3|.
\end{equation*}

\begin{figure}
\centering
\begin{tikzpicture}[line cap=round,line join=round,x=1cm,y=1cm]
\clip(-5.3,-7.02) rectangle (16,2.02);
\draw [line width=0.8pt] (3,2)-- (3,-7);
\draw [line width=0.8pt] (3,-7)-- (4,-7);
\draw [line width=0.8pt] (4,-7)-- (4,2);
\draw [line width=0.8pt] (4,2)-- (3,2);
\draw [line width=0.8pt] (2,1)-- (2,-6);
\draw [line width=0.8pt] (2,-6)-- (5,-6);
\draw [line width=0.8pt] (5,-6)-- (5,1);
\draw [line width=0.8pt] (5,1)-- (2,1);
\draw [line width=0.8pt] (1,0)-- (1,-5);
\draw [line width=0.8pt] (1,-5)-- (6,-5);
\draw [line width=0.8pt] (6,-5)-- (6,0);
\draw [line width=0.8pt] (6,0)-- (1,0);
\draw [line width=0.8pt] (0,-1)-- (0,-4);
\draw [line width=0.8pt] (0,-4)-- (7,-4);
\draw [line width=0.8pt] (7,-4)-- (7,-1);
\draw [line width=0.8pt] (7,-1)-- (0,-1);
\draw [line width=0.8pt] (-1,-2)-- (-1,-3);
\draw [line width=0.8pt] (-1,-3)-- (8,-3);
\draw [line width=0.8pt] (8,-3)-- (8,-2);
\draw [line width=0.8pt] (8,-2)-- (-1,-2);
\draw [line width=0.8pt] (3.66,-2.64)-- (7.7,-5.56);
\draw [line width=0.8pt] (5.38,-4.48)-- (6,-4.331287128712872);
\draw [line width=0.8pt] (5.38,-4.48)-- (7.7,-5.56);
\begin{scriptsize}
\draw [fill=black] (3.66,-2.64) circle (1pt);
\draw[color=black] (3.6,-2.45) node {\large{$p_2$}};
\draw [fill=black] (7.7,-5.56) circle (1pt);
\draw[color=black] (8,-5.5) node {\large{$p_1$}};
\draw [fill=black] (6,-4.331287128712872) circle (1pt);
\draw[color=black] (6.3,-4.2) node {\large{$p_4$}};
\draw [fill=black] (5.38,-4.48) circle (1pt);
\draw[color=black] (5.3,-4.3) node {\large{$p_3$}};
\end{scriptsize}
\end{tikzpicture}
\caption{A representation of the geometric part of the proof of Lemma~\ref{lem finite desc trees} in the case $d = 2$.}
\label{fig 4}
\end{figure}
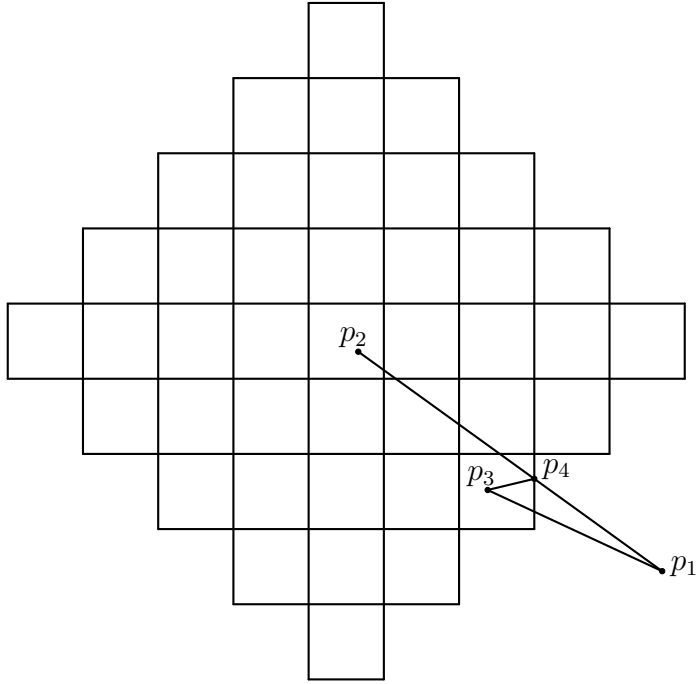

Now, consider the clusters (that is, connected components) formed by the set of cubes at $\ell^1$-distance at most $d^2$ from a blue cube. We show that a.s.\ each of these clusters is finite.
Let $H$ be the graph with vertices at the centers of the cubes in $\mathcal T$ and edges given by the pairs of vertices $u,v\in H$ such that the $\ell^1$-distance between them is at most $Md^2$ and at least one of them is a center of a blue cube.
Thus, every vertex has less than $(d^2+1+d^2)^{d} < (4d^2)^d$ neighbors. 
However, by Observation~\ref{ob containment}, a vertex is blue with probability at most $(8d^2)^{-d}$. Since $(4d^2)^d\cdot (8d^2)^{-d} < 1$, a comparison with a branching process shows that the component of every vertex in $H$ is a.s.\ finite. 
Hence, on this a.s.\ event, the sequence $(u_i)_{i\ge 0}$ may intersect only finitely many blue cubes: indeed, blue cubes whose centers are not connected by an edge in $H$ cannot contain consecutive terms of the sequence which both have arrival times in the interval in $[(1-\varepsilon)r, (1+\varepsilon)r]$. 
This means that $(u_i)_{i\ge 0}$ must contain an accumulation point, which happens with probability 0. 
Thus, for any rational number $r\in [0,1]$, the event $\mathcal A_r$ has probability 0. 
Since the intervals $[(1-\varepsilon/2)r, (1+\varepsilon/2)r]_{r\in \mathbb Q\cap (0,1]}$ cover the interval $(0,1]$ itself, we have
\begin{equation*}
\mathbb P(\text{there is an infinite descending path $(u_i)_{i\ge 0}$ in $G_{\infty}$}) \le \mathbb P(\cup_{r\in \mathbb Q\cap (0,1]} \mathcal A_r) \le \sum_{r\in \mathbb Q\cap (0,1]} \mathbb P( \mathcal A_r) = 0,
\end{equation*}
and the lemma is proved.
\end{proof}

\begin{proof}[Proof of Theorem~\ref{thm G infinity}\eqref{item recurrence}]
Fix an arbitrary vertex $v_0\in G_{\infty}$ and let $(v_i)_{i\ge 0}$ be a sequence of vertices of $G_{\infty}$ such that $v_i$ is the parent of $v_{i-1}$ for all $i\ge 1$. 
Then, by Lemma~\ref{lem finite desc trees}, the connected component of $v_0$ is a.s.\ isomorphic to a copy of $\mathbb N_0$ with a finite tree attached to every node $i\in \mathbb N_0$ (where $i\in \mathbb N_0$ corresponds to $v_i$). 
The last graph is recurrent for the simple random walk since $\mathbb N_0$ is itself recurrent.
\end{proof}

Before moving on with the proof of Theorem~\ref{thm G infinity}\eqref{item connected} we note that, since the process we consider can be represented as a marked Poisson Point Process (with marks representing arrival times), by the adaptation of Proposition~8.13 of~\cite{Last} to marked Poisson Point Processes (see also Exercise~10.1 of~\cite{Last}), our process is ergodic.
In particular, for all $k\in \mathbb N\cup \{\infty\}$, $G_{\infty}$ has $k$ connected components with probability 0 or 1.

Now, we show the more precise result that a.s.\ the number of components is $1$. 
Let us introduce some notation. 
To begin with, attribute a different color $c_v$ to every vertex $v\in V(G_{\infty})$. 
Consider a sequence of independent Poisson Point Processes $(\mathcal {P}_i)_{i\ge 0}$ with intensity 1, with $\mathcal P_0 = V(G_{\infty})$ (vertices arrive at the same time as in $G_\infty$),
and for every $i\ge 1$, equip the points in $\mathcal P_i$ with i.i.d.\ arrival times distributed uniformly in the interval $(i, i+1]$. 
Then, for every $i\ge 1$, connect every vertex $w$ in $\mathcal {P}_i$ with its (almost surely unique) nearest neighbor with smaller arrival time $\pnt(w)$ (which is in $\mathcal P_0\cup \ldots \cup \mathcal P_i$), 
and let $w$ adopt the color of $\pnt(w)$.
Finally, for every vertex $v\in V(G_{\infty})$, define $R_v$ as the closure of the set of points in $\mathbb R^d$ in color $c_v$, and define the boundary set as
\[D = \{x\in \mathbb R^d: \exists u,v\in V(G_{\infty}), x\in R_u\cap R_v\}.\]
We remark that $D$ is a closed set.

The following main result of~\cite{BBCS23} is key in our proof. 
We note that the original result was stated for two colors and in finite volume (more precisely, in the square $[0,1]^d$).
Both restrictions are not problematic: To go beyond the first restriction, one may consider all pairs of two different colors separately as the boundary in the multi-color case is included in the union of the boundaries between all pairs of colors in the two-color case.  
Moreover, for the second restriction, we may restrict our attention to sufficiently large squares.

\begin{theorem}[see Theorem~1 in~\cite{BBCS23}]\label{thm Curien}
The Hausdorff dimension of $D$ is a.s.\ in the interval $(d-1, d)$. In particular, its Lebesgue measure is a.s.\ $0$.
\end{theorem}

Now, we are ready to prove the connectivity of $G_{\infty}$.

\begin{proof}[Proof of Theorem~\ref{thm G infinity}\eqref{item connected}]
From Theorem~\ref{thm Curien} we know that the origin is a.s.\ contained in the open set $\mathbb R^d\setminus D$. In particular, for every $\varepsilon > 0$ there is a real number $r = r(\varepsilon) > 0$ such that the ball $B(0, r)$ is contained in $\mathbb R^d\setminus D$ with probability $1-\varepsilon$.

Fix a real number $M > 0$. 
We will show that a.s.\ all vertices at distance at most $M$ from the origin have a common ancestor in $G_{\infty}$.
For every $n\ge 0$, consider the random partition of $\mathbb R^d$ where every point $v \in \cup_{i\ge 0} \mathcal P_i$ receives color $c_v$, if $x_v\le 1/n$, and otherwise receives the color of its nearest ancestor (in terms of graph distance) with arrival time in $(0, 1/n]$. Now, for each $n$, consider a homothety of $\mathbb R^d$ centered at the origin and with factor $n^{-1/d}$ and let $S_n$ be the a.s.\ unique monochromatic region containing the origin in this rescaled partition.  
In particular, crucially, $(S_n)_{n\ge 1}$ are random closed sets with the same distribution.

Now, fix $n\ge (M/r)^d$. 
Then, the homothety with factor $n^{-1/d}$ is such that the original non-rescaled ball $B(0,M)$ is contained in the rescaled ball $B(0, r)$,  and by our choice of $r$ we know that $S_n$ contains $B(0, r)$ in its interior with probability at least $1-\varepsilon$.
As a consequence, all vertices of $G_{\infty}$ in the (non-rescaled) ball $B(0,M)$ have a common ancestor with arrival time in the interval $(0, (r/M)^d]$ with probability at least $1-\varepsilon$. 
Since this holds for every $\varepsilon > 0$, we have that a.s.\ all points in $B(0,M)$ are in the same connected component of $G_{\infty}$. 
At the same time, this is true for every $M$, which finishes the proof by taking a union bound over all integer $M > 0$.
\end{proof}

\section{\texorpdfstring{Leaves and cherries in $G_n$}{}}\label{sec trees}
This section is dedicated to the proofs of 
Corollary~\ref{cor:cherries} and Theorem~\ref{thm leaves}.

\subsection{Expected number of cherries}\label{subs cherries}

As an application of Theorem~\ref{thm G infinity}\eqref{item local limit}, we compute the expected number of cherries (that is, paths of length two rooted at their central vertex) in the 1-NN tree and in the uniform attachment tree (intuitively corresponding to the case $d = \infty$, as mentioned in the introduction). We note that, in theory, the method of computation may be generalized to other subtrees and higher dimensions. Let us point out that, albeit the expectation can still be expressed via a finite number of integrals, the integral expressions become complicated to compute explicitly (even numerically).

In this subsection, we denote by $Ch^1_n$ the number of cherries in the 1-NN tree where the central vertex has smallest arrival time, and by $Ch^{\infty}_n$ the number of cherries in the uniform attachment tree with the same property. We note that the number of cherries where the earliest arrived vertex is an endpoint is $n - \deg_{G_n}(0)$, which is $n-o(n)$ both a.a.s.\ and in expectation.

\begin{lemma}\label{lem cherries d=1}
$\mathbb E Ch^1_n = \dfrac{(3-2\log 2 + o(1))n}{12} \approx 0.134 n$ and $\mathbb E Ch^{\infty}_n = n - o(n)$.
\end{lemma}
\begin{proof}
First, a direct computation shows that
\begin{equation*}
\mathbb E Ch^{\infty}_n =\!\!\!\!\! \sum_{0\le i < j < k\le n-1} \mathbb P(j,k \text{ attach to } i) = \sum_{0\le i < j < k\le n-1} \dfrac{1}{jk} = \sum_{1\le j < k\le n-1} \dfrac{1}{k} = \sum_{2\le k\le n-1} \dfrac{k-1}{k} = n - O(\log n).
\end{equation*}

Now, we turn to $\mathbb E Ch^{1}_n$. By Theorem~\ref{thm G infinity}\eqref{item local limit}, it suffices to compute the number of cherries rooted at a vertex located at $p_0 = 0$ in the PPP(1) on the real line.

Let $x_0 < x_1 < x_2$ be the arrival times of the vertices in the cherry, and let $p_1, p_2$ be the positions of the second and the third vertex. 
Then, either the second and the third vertices lie both to the left or both to the right of the root, in which one has $p_0 = 0 < p_2 < p_1 - p_2$ or $p_1 - p_2 < p_2 < p_0$ (which are symmetric and have the same probability), or $p_0$ is between $p_1$ and $p_2$.
We consider the two cases separately.

In the first case, suppose that $p_0 = 0 < p_2 < p_1 - p_2$. Then, conditionally on $x_2$, the position $p_2$ of the vertex with arrival time $x_2$ is distributed according to an exponential random variable $\mathcal E(x_2)$ (note that this is the closest vertex on the right of 0 with arrival time in the interval $[0, x_2]$). 
Moreover, since the vertex at 0 is its parent, the interval $[0,2p_2]$ must contain no vertex with arrival time in $[0, x_2]$. 
Conditionally on the above event, the vertex with arrival time $x_1$ arrives independently at a position $p_1$ distributed as $2p_2+\mathcal E(x_1)$. 
Again, the interval $[0,2p_1]\setminus [0, 2p_2]$ must contain no vertex with arrival time in $[0, x_1]$. 
Thus, the expectation in this case is
\begin{align}
2\int_{0\le x_0 < x_1 < x_2\le 1}\int_{p_2=0}^{\infty}\int_{p_1=2p_2}^{\infty} x_1 x_2 \exp(- 2x_2p_2-2x_1(p_1-p_2)) dp_1dp_2dx_0dx_1dx_2.\label{eq summand 1}
\end{align}
The integration over $p_1$ and $p_2$ in~\eqref{eq summand 1} gives\footnote{Integrating~\eqref{eq summand 1} over $p_1, p_2$: \tiny{\url{https://www.wolframalpha.com/input?i=integrate+2+x_1+x_2+exp\%28-+2x_2p_2-2x_1\%28p_1-p_2\%29\%29+for+p_1+from+2p_2+to+infinity\%2C+p_2+from+0+to+infinity\%2C+0+\%3C\%3D+x_0+\%3C+x_1+\%3C+x_2+\%3C\%3D+1}}}
\begin{align}
\int_{x_2=0}^1 \int_{x_0 = 0}^{x_2} \int_{x_1=x_0}^{x_2} \dfrac{x_2}{2(x_1+x_2)} dx_1dx_0dx_2,\label{eq:x-s}
\end{align}
which is equal\footnote{Integrating~\eqref{eq:x-s}: \tiny{\url{https://www.wolframalpha.com/input?i=integrate+\%28x_2\%2F\%282x_1\%2B2x_2\%29\%29+for+x_0+from+0+to+1\%2C+for+x_1+from+x_0+to+1\%2C+for+x_2+from+x_1+to+1}}} to $\tfrac{1}{6}(1-\log 2)$.
By a similar argument, the expectation in the second case yields
\begin{align}
& 2\int_{0\le x_0 < x_1 < x_2\le 1}\int_{p_1=0}^{\infty}\int_{p_2=-\infty}^{0} x_1 x_2\exp(-2x_1p_1 - 2x_2|p_2|) dp_1dp_2dx_0dx_1dx_2\nonumber\\
=\hspace{0.3em}
& 2\int_{0\le x_0 < x_1 < x_2\le 1}\int_{p_1=0}^{\infty}\int_{q_2=0}^{\infty} x_1 x_2 \exp(-2x_1p_1 - 2x_2q_2) dp_1dq_2dx_0dx_1dx_2\nonumber\\
=\hspace{0.3em}
& 2\int_{0\le x_0 < x_1 < x_2\le 1} \dfrac{1}{4} dx_0dx_1dx_2 = \dfrac{1}{12}.\label{eq summand 2}
\end{align}
The lemma is proved by summing~\eqref{eq:x-s} and~\eqref{eq summand 2}.
\end{proof}

\subsection{\texorpdfstring{Number of leaves in $G_n$}{}}\label{subs leaves}
In this subsection, we estimate the expected number of leaves of $G_n$. 
There are two notable differences with respect to the count of the number of cherries in the previous subsection. 
First, one needs to take into consideration not only the presence of edges but also of non-edges when counting leaves. 
This task is more difficult since simple integral expressions cannot be deduced as before. 
Second, we remark that these are the only bounds on local statistics in this paper which are independent of the dimension. 

\begin{theorem}\label{thm leaves 1}
For every $d\ge 1$, the expected number of leaves in the tree $G_n$ is bounded from above by $(1/2+o(1))n$.
\end{theorem}

We note that the uniform attachment tree has an expected number of $n/2$ leaves. This is well-known and easy to see:
\begin{equation*}
    \mathbb E[|\text{leaves in the UST on }n \text{ vertices}|] = \sum_{i=1}^{n-1} \prod_{j=i+1}^{n-1} \dfrac{j-1}{j} = \sum_{i=1}^{n-1} \dfrac{i}{n-1} = \dfrac{n}{2}.
\end{equation*}
Monte-Carlo simulations suggest that, for any $d\ge 1$, the number of leaves in $G_n$ concentrates around $C_d n$ where $C_d < 1/2$ but increases with $d$. 
Since the uniform attachment tree may intuitively be seen as the $\infty$-NN tree (as we mentioned in the introduction), we believe that $C_d\to 1/2$ as $d$ tends to infinity.

Let us prepare the ground for the proof of Theorem~\ref{thm leaves 1} with several preliminary results. We start with an easy observation.
Fix a function $\psi = \psi(n) = \omega(1)$. 

\begin{observation}\label{ob 5.10}
A.a.s.\ for every point $x\in \mathbb T^d_n$, the Voronoi cell of $x$ in $\{x\}\cup V(G_{\psi})$. 
\end{observation}
\begin{proof}
Fix a tessellation of $\mathbb T^d_n$ into cubes of side length less than $n^{1/d}/(4\sqrt{d})$. 
On the one hand, a.a.s.\ every cube contains at least one vertex of $G_{\psi}$. 
On the other hand, every point $y$ at distance at least $n^{1/d}/4$ from $x$ does not belong to the Voronoi cell of $x$ in $\{x\}\cup V(G_{\psi})$ since the cube that contains $y$ also contains a vertex at distance less than $\sqrt{d}\cdot n^{1/d}/(4\sqrt{d}) = n^{1/d}/4$ from $y$, as desired. 
\end{proof}

Denote by $V^j_i$ the Voronoi cell of $p(i)$ after embedding $[j-1]_0$ in $\mathbb T^d_n$, 
and denote $2V^j_i = \{x\in \mathbb T^d_n \mid  (x+p(i))/2\in V^j_i\}$. 
A configuration is \emph{non-degenerate} if the positions of no two vertices coincide.
\begin{lemma}\label{lem inclusion}
Fix a vertex $i$, a non-degenerate configuration on the first $j\ge i+1$ vertices embedded in $\mathbb T^d_n$, and suppose that $V^j_i\subseteq B(p(i),n^{1/d}/4)$. Let $\mu$ be the uniform random measure on $\mathbb T^d_n$ and $\mu_0$ be the uniform random measure on $\mathbb T^d_n\setminus V^j_i$. One may couple $\mu$ and $\mu_0$ so that the Voronoi cell $V^{j+1}_{i, \mu}$, obtained after sampling a vertex $u$ according to $\mu$, is contained in the Voronoi cell $V^{j+1}_{i, \mu_0}$, obtained after sampling a vertex $u_0$ according to $\mu_0$.
\end{lemma}
\begin{proof}
Note that since $V^j_i\subseteq B(p(i),n^{1/d}/4)$, $2V^j_i$ has volume $2^d|V^j_i|$. 
We couple $\mu$ and $\mu_0$ as follows. Let $U$ be a random variable with distribution $\text{Bernoulli}(|V^j_i|/n)$. 
If $U = 0$, sample $u$ uniformly at random in $\mathbb T^d_n\setminus V^j_i$ and set $u_0 = u$. 
If $U = 1$, then sample an additional random variable $U_0$ with distribution $\text{Bernoulli}\left(\frac{(2^d-1)|V^j_i|}{n-|V^j_i|}\right)$, independent of $U$. 
If $U_0 = 0$, sample $u$ uniformly at random in $V^j_i$ and $u_0$ uniformly at random in $\mathbb T^d_n\setminus 2V^j_i$, independently from $u$. 
In this case, $u_0$ does not alter $V^j_i$ at all. 
Otherwise, if $U_0 = 1$, we sample $u$ uniformly at random in $V^j_i$, and then construct $u_0$ as the unique point in $2V^j_i\setminus V^j_i$ that can be written as $u + 2^\ell(u-p(i))$ for some positive integer $\ell$ (note that if $u + 2^\ell(u-p(i))\in 2V^j_i\setminus V^j_i$, then $u + 2^{\ell+1}(u-p(i))$ is already outside $2V^j_i\setminus V^j_i$).
Indeed, to see that $u_0$ is uniformly distributed in $2V^j_i\setminus V^j_i$, it suffices to observe that the preimage of any open set $A$ in $2V^j_i\setminus V^j_i$ has $\mathrm{Unif}(V^j_i)$-measure 
\begin{equation*}
    \dfrac{1}{|V^j_i|} \sum_{i\ge 1} \dfrac{|A|}{2^{id}} = \dfrac{|A|}{(2^d-1)|V^j_i|} = \dfrac{|A|}{|2V^j_i\setminus V^j_i|},
\end{equation*}
where the first formula follows from the fact that the preimage of $A$ is obtained by homotheties with center $p(i)$ and coefficients $(2^i)_{i\ge 1}$ which contract its volume by a factor of $(2^{id})_{i\ge 1}$, respectively. This ensures that $u\sim \mu$ and $u_0\sim \mu_0$.

It remains to show that $V^{j+1}_{i, \mu}$ is contained in $V^{j+1}_{i, \mu_0}$. 
If $u$ does not fall into $V^j_i$, then the two Voronoi cells coincide. 
If $u$ falls into $V^j_i$ and $U_0=0$, 
then $V^j_i = V^{j+1}_{i, \mu_0}$ so $V^{j+1}_{i,\mu}\subseteq V^{j+1}_{i, \mu_0}$ since the Voronoi cell cannot increase by adding vertices. 
Finally, if $u$ falls into $V^j_i$ and $U_0=1$, we have that $u_0, u$ and $p(i)$ are collinear and are found on their common line in this order. 
For every point $x\in \mathbb T^d_n$, denote the bisecting hyperplane of $x$ and $i$ (consisting of the set of points at equal distance from $x$ and $i$) by $h_x$. 
Then, $h_u\cap V^j_i$ separates $i$ from $h_{u_0}\cap V^j_i$ in $V^j_i$, which means that once again $V^{j+1}_{i, \mu}\subseteq V^{j+1}_{i, \mu_0}$, which finishes the proof of the lemma.
\end{proof}
 
Fix an integer $i\in [n-2]_0$ and, for every $j\ge i+1$, let $\mathcal A_{i,j}$ be the event that vertex $j$ does not attach to $i$ via an edge. 
The next lemma uses Lemma~\ref{lem inclusion} iteratively to obtain a more general result.

\begin{lemma}\label{lem conditioning}
Fix vertices $i$, $j\in [i+1, n-1]$ and a non-degenerate configuration on the first $i+1$ vertices embedded in $\mathbb T^d_n$. 
Suppose that, for every point $x\in \mathbb T^d_n$, the Voronoi cell of $x$ in $\{x\}\cup [i]_0$ is contained in $B(x,n^{1/d}/4)$. 
Then, the Voronoi cell $V^j_i$ of $p(i)$ after embedding $i+1,\dots,j$ in $\mathbb T^d_n$ uniformly and independently may be coupled with the Voronoi cell $V^j_{i,0}$ of $p(i)$ after embedding $i+1,\dots,j$ in $\mathbb T^d_n$ uniformly and independently conditionally on $\cap_{k=i+1}^{j} \mathcal A_{i,k}$ so that $V^j_i\subseteq V^j_{i,0}$ a.s.
\end{lemma}
\begin{proof}
We prove the lemma by induction. The statement for $j = i+1$ is ensured by Lemma~\ref{lem inclusion}. 
Suppose that the induction hypothesis is satisfied for $j-1\ge i+1$: that is, there is a coupling ensuring that $V^{j-1}_i\subseteq V^{j-1}_{i,0}$. 
By applying Lemma~\ref{lem inclusion} for the configuration on $p(0), \ldots, p(j-1)$ conditioned on $\cap_{k=i+1}^{j-1} \mathcal A_{i,k}$ (which is a.s.\ non-degenerate), we conclude that one may couple
$V^j_{i,0}$ with $V^{j}_{i,1}$, defined as the Voronoi cell of $p(i)$ after embedding $[j]_0$ in $\mathbb T^d_n$ conditionally on $\cap_{k=i+1}^{j-1} \mathcal A_{i,k}$ only but not on $j\notin V^{j-1}_{i,0}$, 
so that $V^j_{i,1} \subseteq V^{j}_{i,0}$. 
Also, since a.s.\ $V^{j-1}_{i}\subseteq V^{j-1}_{i,0}$ by the induction hypothesis, we have that $V^{j}_i\subseteq V^{j}_{i,1}$: indeed, $V^{j-1}_i$ and $V^{j-1}_{i,0}$ have the same center $p(i)$, 
so after embedding $j$ uniformly in $\mathbb T^d_n$, both $V^{j}_i = V^{j-1}_i\cap H$ and $V^{j}_{i,1} = V^{j-1}_{i,0}\cap H$ for some half-space $H$, and so $V^{j}_i$ must be contained in $V^{j}_{i,1}$ a.s. 
We conclude that under the composition of the two couplings (the one given by the induction hypothesis and the one constructed using Lemma~\ref{lem inclusion}) $V^{j}_i\subseteq V^{j}_{i,1}\subseteq V^{j}_{i,0}$ a.s., which finishes the induction hypothesis and the proof of the lemma.
\end{proof}

\begin{corollary}\label{cor 6.3}
For every $j\ge i+1$ we have $\mathbb P(\overline{\mathcal A_{i,j}}\mid\cap_{k=i+1}^{j-1} \mathcal A_{i,k})\ge \mathbb P(\mathcal A_{i,j}) = \dfrac{1}{j}$.
\end{corollary}
\begin{proof}
By Lemma~\ref{lem conditioning}, the volume of the Voronoi cell $V^j_i$ of $i$ at time $j$ is positively correlated with $\mathds{1}_{\cap_{k=i+1}^{j-1} \mathcal A_{i,k}}$. The equality comes from the fact that the parent of $j$ is uniformly chosen among $[j-1]_0$.
\end{proof}

\begin{proof}[Proof of Theorem~\ref{thm leaves 1}]
Fix $\psi=\psi(n) = \omega(1)$ satisfying $\psi(n) = o(n)$ and condition on the a.a.s.\ event in Observation~\ref{ob 5.10}. Fix an integer $i\in [n-1], i\ge \psi$. By Corollary~\ref{cor 6.3},
\begin{equation*}
\mathbb P(i \text{ is a leaf in }G_n) = \mathbb P(\cap_{j=i+1}^{n-1} \mathcal A_{i,j}) = \mathbb P(A_{i,i+1}) \prod_{j=i+1}^{n-1} \mathbb P(\mathcal A_{i,j+1}\mid \cap_{k=i+1}^{j} \mathcal A_{i,k})\le \prod_{j=i+1}^{n-1} \dfrac{j-1}{j} = \dfrac{i}{n-1}.
\end{equation*}

\noindent
We conclude that
\begin{equation*}
\mathbb E[|\text{leaves in }G_n|]\le (\psi+1)+\sum_{i=\psi}(n)^{n-1} \dfrac{i}{n-1} = \left(\dfrac{1}{n-1}\sum_{i = 1}^{n-1} i\right) + \psi + O(1+\psi^2(n)/n) = \dfrac{n}{2} + O(\psi),
\end{equation*}
and the proof is finished.
\end{proof}

We now provide a corresponding lower bound, which confirms that the expected number of leaves is $\Theta(n)$ in every dimension $d\ge 1$.
\begin{lemma}\label{lemma LB leaves}
For every dimension $d\ge 1$, the expected number of leaves is at least $(1/4e^4+o(1))n$.
\end{lemma}
\begin{proof}
Define $k = \lfloor n/2\rfloor$. 
Let $\mathcal X_k$ be the family of point sets in $\mathbb T^d_n$ of size $k$, and $\mu_k$ be the uniform measure on $\mathcal X_k$. 
For $\nu\in \mathcal X_k$, define $F(\nu)$ as the expectation of the volume of the Voronoi cell of a uniform random point $x\in \mathbb T^d_n\setminus \nu$ in $\{x\}\cup \nu$.
Note that, for $\nu_k$ sampled from $\mathcal X_k$ according to $\mu_k$, we have $\mathbb E[F(\nu_k)] = \frac{n}{k + 1} < 2$, which is also the expected size of a Voronoi cell in the presence of $k+1$ random points. Thus, the event $\mathcal A = \{F(\nu_k)\le 4\}$ has probability more than $1/2$ by Markov's inequality for $\overline{\mathcal A}$.

For all $i\in \{k,\dots,n-1\}$, denote by $\mathcal B_i$ (resp. $\mathcal C_i$) the event that the Voronoi cell of $p(i)$ in the configuration $p(0),\ldots, p(k), p(i)$ has volume at most 8 (resp.~contains no vertex among $\{k,\dots,n-1\}\setminus i$).
Note that, conditionally on $\mathcal A$, the event $\overline{\mathcal B_i}$ has probability at most $1/2$ by Markov's inequality. Thus, for every integer $i\in [k,n-1]$, the probability that $i$ is a leaf is bounded from below by
\begin{equation*}
\mathbb P(\mathcal C_i)\ge \mathbb P(\mathcal C_i\mid\mathcal B_i)\mathbb P(\mathcal B_i\mid\mathcal A)\mathbb P(\mathcal A)\ge \left(1 - \dfrac{8}{n}\right)^{n-k-1}\cdot \dfrac{1}{2}\cdot \dfrac{1}{2} = \dfrac{1+o(1)}{4e^4},
\end{equation*}
and the lemma is proved.
\end{proof}

\begin{proof}[Proof of Theorem~\ref{thm leaves}]
By Theorem~\ref{thm leaves 1} and Lemma~\ref{lemma LB leaves}, the bounds hold in expectation. 
By Theorem~\ref{thm concentration} applied with $g$ being the number of leaves in a tree and $L = 2$ (by replacing one edge, the number of leaves changes by at most 2), $t = n^{2/3}$ and $\phi = C_0\log n$, we deduce that the number of leaves is concentrated around its expected value, which finishes the proof.
\end{proof}

\begin{remark}
In the proof of Lemma~\ref{lemma LB leaves}, we showed that $\mathbb E[F(\nu_k)] = 2+o(1)$. 
In fact, one may also show that $F(\nu_k) = 2+(\log n)^{O(1)}/n$ both a.a.s.\ and in expectation. 
Since this result is not our main focus, we only sketch how this can be done: 
sticking to the notation introduced in the proof of Lemma~\ref{lemma LB leaves}, define more generally $\mathcal X$ to be the family finite point sets in $\mathbb T^d_n$. 
Furthermore, extend $F$ to $\mathcal X$ with the same definition and define the difference operator $D_x(F) := F(\nu_k \cup \{x\}) - F(\nu_k)$. 
Finally, we use the first order Poincar\'e inquality in its discrete form, which says that
$\mathrm{Var}(F)\le \int_{\mathcal X} \mathbb E[(D_x F)^2]\hspace{0.2em} \lambda(dx)$ (see~\cite{SY}).
\end{remark}

\begin{remark}
Although Theorem~\ref{thm leaves} provides lower and upper bounds for the expected number of leaves in any dimension, the exact constant in front of $n$ seems hard to obtain. We provide an exact expression only in dimension 1: by Theorem~\ref{Palm:Penrose} and Theorem~\ref{thm G infinity}\eqref{item local limit}, it is sufficient to estimate the probability that a vertex $v_0$ at position 0 is a leaf in $G_{\infty}$. 

For some integer $k\ge 1$, let $v_1, \dots, v_{k+1}$ be the sequence of vertices with decreasing arrival times $x_1, \dots, x_{k+1}\in [0,1]$ and positions $p_1, \dots, p_{k+1}\in (0, \infty)$ defined as follows:
\begin{itemize}
    \item $v_{k+1}$ is the left-most vertex to the right of $v_0$ with arrival time $x_{k+1}\in (0, x_0)$ (where $x_0$ is the arrival time of $v_0$),
    \item for every $i\in [k]$, $v_i$ is the left-most vertex between $v_{i-1}$ and $v_{k+1}$ with $x_{i+1}\in (x_0, x_i)$.
\end{itemize}

One may verify that $v_0$ has no neighbor with position in $(0, \infty)$ if and only if, for all $i\in [k]$, $p_i > p_{i+1}-p_i$. Setting $s_i = p_i - p_{i-1}$ (with $p_0 = 0$), the probability of the event described in the two bullets can be computed by multiple integrations based on an exploration of the 1-dimensional PPP(1) on $[0, \infty)$ starting from 0. More precisely, the probability of the above event is given by summing the following multiple integral for all integers $k\ge 0$:

\footnotesize{
\begin{align*}
&\int_{x_0=0}^1\int_{x_1 = x_0}^1\int_{x_2 = x_0}^{x_1}\dots \int_{x_k = x_0}^{x_{k-1}} \int_{x_{k+1}=0}^{x_0} \int_{s_1=0}^{+\infty} \int_{s_2 = 0}^{s_1}\int_{s_3 = 0}^{s_1+s_2}\dots \int_{s_{k+1}=0}^{s_1+\dots+s_k}\\
&(1-x_0)\exp(-s_1(1-x_0)) (x_1-x_0)\exp(-s_2(x_1-x_0))\dots (x_{k-1}-x_0)\exp(-s_k(x_{k-1}-x_0)) x_0 \exp(-s_{k+1} x_0) dx_0 \prod_{i=1}^{k+1} dx_i ds_i.
\end{align*}
}

\noindent
\normalsize{Of course, the probability of avoiding that a vertex on the left of $v_0$ attaches to $v_0$ is computed in the same way. Although we are not able to compute the above sum of multiple integrals, numerical evidence over 100 independent trials with 10000 vertices leads to an empirical average of 4551.7 leaves with standard deviation 34.32.}
\end{remark}

\section{Conclusion and open questions}\label{sec conclusion}
In this paper, we conducted a rigorous analysis of the online $d$-dimensional nearest neighbor tree. 
Furthermore, we introduced and analyzed a natural infinite counterpart of this model. 
Needless to say, besides the already mentioned Conjecture~\ref{conj:local}, many questions remain for further thought: 
\begin{enumerate}
\item We firmly believe that some of our techniques extend to the model of $d$-dimensional $k$-nearest neighbor graph. It might be interesting to analyze this model in more depth.
\item Our partial results and simulations indicate that both the number of leaves and the number of cherries should be increasing with the dimension, and therefore should yield a good estimator for detecting the dimension. Note that the asymptotic count of the cherries was derived recently by~\cite{Cas23} but his result solves the problem for sufficiently large $d$ only.
\item As we pointed out in Subsection~\ref{relwork}, a lot of work on similar models has been done to find the root of the tree. 
More precisely, one might wonder if, for every $\varepsilon > 0$, one may design a set of vertices of constant size $K=K(\varepsilon)$ that contains the root of the graph with probability $1-\varepsilon$. 
It is not clear to us whether statistics providing good estimators in other models can be used in our model as well.
\item An interesting question in the infinite setup could be to study the asymptotic rate of expansion of the graph $G_{\infty}$: given $v\in V(G_{\infty})$, how fast does the number of vertices at graph distance $k$ from $v$ grow as $k\to \infty$?
\end{enumerate}

\section*{Acknowledgements} The authors would like to thank G\'{a}bor Lugosi and Vasiliki Velona for bringing the topic to our attention and for discussions in an early stage of this paper. 
The authors would also like to thank Bas Lodewijks for a careful proofreading, 
and to David Aldous and Andrew Wade for bringing several additional references to our attention. 
The first author would like to thank Ivailo Hartarsky for a discussion around the connectivity of $G_{\infty}$. 
We are also grateful for many useful comments and remarks by the two anonymous referees.

\bibliographystyle{plain}
\bibliography{References}

\end{document}